\documentclass{article}
\pdfoutput=1

\usepackage{amsmath}
\usepackage{amsfonts}
\usepackage{amssymb}
\usepackage{amsthm}
\usepackage{tikz-cd}
\usepackage{tikz}
\usepackage{amsxtra}
\usepackage{wrapfig}
\usepackage{color}
\usepackage{geometry}
\usepackage{url}
\usepackage[hidelinks]{hyperref}
\usepackage[percent]{overpic}
\usepackage{mathrsfs}
\usepackage{stmaryrd}

\usepackage{setspace}

\newtheorem{thm}{Theorem}[subsection]
\newtheorem{prop}[thm]{Proposition}

\newtheorem{cor}[thm]{Corollary}

\newtheorem{lem}[thm]{Lemma}
\theoremstyle{definition}
\newtheorem{defn}[thm]{Definition}
\newtheorem{setup}[thm]{Setup}
\newtheorem*{thma}{Theorem A}
\newtheorem*{thmb}{Theorem B}
\newtheorem*{conji}{Conjecture}
\newtheorem{conj}[thm]{Conjecture}
\theoremstyle{remark}
\newtheorem{ex}[thm]{Example}
\newtheorem{rmk}[thm]{Remark}

\newcommand{\cat}[1]{{\mathbf{#1}}}
\newcommand{\p}{\paragraph{}}
\newcommand{\spec}{\operatorname{Spec}}
\newcommand{\onto}{\twoheadrightarrow}
\newcommand{\into}{\hookrightarrow}
\newcommand{\from}{\leftarrow}
\newcommand{\lot}{\otimes^{\mathbb{L}}}
\newcommand{\per}{\ensuremath{\cat{per}}}
\newcommand{\stab}{\underline{\mathrm{CM}}}

\DeclareMathOperator{\id}{id}
\let\im\relax\DeclareMathOperator{\im}{im}
\let\ker\relax\DeclareMathOperator{\ker}{ker}
\DeclareMathOperator{\coker}{coker}

\let\hom\relax\newcommand{\hom}{\mathrm{Hom}}
\newcommand{\enn}{\mathrm{End}}
\DeclareMathOperator{\tor}{Tor}
\DeclareMathOperator{\ext}{Ext}

\newcommand{\Z}{\mathbb{Z}}
\newcommand{\N}{\mathbb{N}}
\newcommand{\A}{\mathbb{A}}

\usepackage[utf8]{inputenc}
\usepackage[T1]{fontenc}
\usepackage{lmodern}

\title{Noncommutative deformation theory, the derived quotient, and DG singularity categories}
\author{Matt Booth}

\newcommand{\R}{{\mathrm{\normalfont\mathbb{R}}}}
\numberwithin{equation}{section}

\newcommand{\dq}{\ensuremath{A/^{\mathbb{L}}\kern -2pt AeA} }

\newcommand{\thick}{\ensuremath{\cat{thick} \kern 0.5pt}}
\newcommand{\dloc}{\mathbb{L}_S(A)}

\newcommand{\grdef}{\text{\raisebox{.15ex}{$\mathscr{D}$}}\mathrm{ef}}

\newcommand{\dgh}{\mathrm{HOM}}
\newcommand{\dge}{\mathrm{END}}

\newcommand{\dgart}{\cat{dgArt}_k^{\leq 0}}
\newcommand{\dga}{\cat{dga}_{k}^{\leq 0}}
\newcommand{\proart}{{\cat{pro}(\cat{dgArt}_k^{\leq 0})}}

\makeatletter
\newcommand{\holim@}[2]{%
	\vtop{\m@th\ialign{##\cr
			\hfil$#1\operator@font holim$\hfil\cr
			\noalign{\nointerlineskip\kern1.5\ex@}#2\cr
			\noalign{\nointerlineskip\kern-\ex@}\cr}}%
}
\newcommand{\holim}{%
	\mathop{\mathpalette\holim@{\leftarrowfill@\textstyle}}\nmlimits@
}
\makeatother

\usepackage[nottoc]{tocbibind}
\allowdisplaybreaks
\begin{document}
	\date{}
\maketitle{}
\begin{center}
\textsc{Abstract}
\begin{figure}[ht]\hspace{0.1\linewidth}
			\begin{minipage}[c]{0.8\linewidth}
			We show that Braun-Chuang-Lazarev’s derived quotient prorepresents a naturally defined noncommutative derived deformation functor. Given a noncommutative partial resolution of a Gorenstein algebra, we show that the associated derived quotient controls part of its dg singularity category. We use a recent result of Hua and Keller to prove a recovery theorem, which can be thought of as providing a solution to a derived enhancement of a conjecture made by Donovan and Wemyss about the birational geometry of threefold flops.
			\end{minipage}
\end{figure}
\end{center}

\section{Introduction}
This is the first in a series of papers generalising Donovan-Wemyss's contraction algebra \cite{DWncdf} to a new invariant, the derived contraction algebra, which will be constructed as a derived quotient in the sense of Braun-Chuang-Lazarev \cite{bcl}. If $A$ is a differential graded algebra (dga) and $e$ is an idempotent in the cohomology algebra $H(A)$, the derived quotient $\dq$ is a dga that is universal with respect to homotopy annihilating $e$. Our first main theorem shows that the derived quotient admits an interpretation in terms of noncommutative derived deformation theory. This generalises prorepresentability results of Efimov-Lunts-Orlov \cite{ELO2} to the singular setting.
\p Following Kalck and Yang \cite{kalckyang} \cite{kalckyang2} we consider derived quotients associated to noncommutative partial resolutions of hypersurface singularities, and prove that the quasi-isomorphism class of the derived quotient recovers a thick subcategory of the dg singularity category of the base. When the partial resolution is a resolution, we recover the whole of the dg singularity category. A recent result of Hua and Keller \cite{huakeller}, analogous to Dyckerhoff's computation of the Hochschild cohomology of the dg category of matrix factorisations \cite{dyck}, then allows us to conclude that the associated derived quotient completely recovers the geometry of the hypersurface singularity. We work in a general setup and only introduce geometric hypotheses when necessary in order to obtain a unified theory that works for both full and partial resolutions. This unifying approach will be used in an upcoming preprint \cite{dcalg} to generalise results of Hua and Keller \cite{huakeller} obtained in the context of smooth Calabi-Yau threefolds to the singular setting. Furthermore, it will allow us to study derived contraction algebras of partial resolutions of Kleinian singularities, where it is imperative to work with singular varieties. Throughout we will work over an algebraically closed field $k$ of characteristic zero.

\subsection{Deformation theory}
The noncommutative deformation theory of modules over a ring was originally developed by Laudal \cite{laudalpt}, and has recently found many applications within algebraic geometry \cite{todatwists} \cite{DWncdf} \cite{kawamatapointed}. Typically, one is interested in deforming objects of derived or homotopy categories, which has been studied in detail by Efimov-Lunts-Orlov \cite{ELO} \cite{ELO2} \cite{ELO3}. It is important to have good control over prorepresenting objects of deformation functors, and our first main theorem identifies the derived quotient as controlling a noncommutative derived deformation problem:
\begin{thma}[\ref{maindefmthm}]Let $A$ be a $k$-algebra and $e \in A$ an idempotent. Suppose that $A/AeA$ is a local algebra and that $\dq$ is cohomologically locally finite. Let $S$ be $A/AeA$ modulo its radical, regarded as a right $A$-module. Then $\dq$ is naturally a pro-Artinian dga, and moreover prorepresents the functor of framed noncommutative derived deformations of $S$.
\end{thma}
If $A$ has finite global dimension, then Theorem A is an easy consequence of the prorepresentability results of \cite{ELO2}, and hence the full theorem can be viewed as a generalisation to the singular setting. We call a dga $B$ cohomologically locally finite if each $H^j(B)$ is a finite-dimensional $k$-vector space; we note that cohomological local finiteness of $\dq$ can be checked explicitly when $A$ is presented as the path algebra of a quiver with relations (\ref{derquotcohom}, \ref{h1finite}) and is in general a quite weak finiteness condition. A framed deformation of $S$ is a deformation of $S$ that respects a fixed choice of isomorphism $S \cong k$ (\ref{frmdefs}). We remark that the condition that $A/AeA$ is local can probably be dropped if one uses pointed deformations, as in Laudal \cite{laudalpt} or Kawamata \cite{kawamatapointed}.

\p In characteristic zero, Deligne's philosophy that differential graded Lie algebras control commutative deformation problems is closely related to the Koszul duality between the Lie and commutative operads \cite{unifying}. In order to prove Theorem A, we must prove a Koszul duality result for noncommutative algebras (\ref{kdfin}), which can be interpreted as a strictification result for homotopy pro-Artinian dgas. We will then prove (\ref{prorepfrm}) that the functor of framed deformations of a simple one-dimensional $A$-module $S$ is prorepresented by the Koszul dual of the derived endomorphism algebra $\R\enn_A(S)$, and Theorem A will follow by combining the two. In particular, the pro-Artinian structure on $\dq$ comes from viewing $\dq$ as the dual of a coalgebra.
\subsection{Singularity categories}
If $R$ is a ring, its singularity category is the triangulated category $D_{\mathrm{sg}}(R):=D^b(R)/\per(R)$, which can be seen as quantifying the type of singularities of $R$. Singularity categories were introduced by Buchweitz \cite{buchweitz} who proved that when $R$ is Gorenstein, $D_{\mathrm{sg}}(R)$ is equivalent to the stable category $\stab R$ of maximal Cohen-Macaulay $R$-modules. Singularity categories were later studied for schemes by Orlov \cite{orlovtri}, who gave applications to homological mirror symmetry, and differential graded enhancements of singularity categories have been studied recently by Blanc-Robalo-To\"{e}n-Vezzosi \cite{motivicsingcat} where they are constructed using the dg quotient \cite{kellerdgquot} \cite{drinfeldquotient}. 

\p Motivated by the Bondal-Orlov conjecture \cite{bondalorlov}, noncommutative (crepant) resolutions of singularities were introduced by Van den Bergh \cite{vdb} \cite{vdbnccr} and have been studied extensively since. A noncommutative resolution of a ring $R$ is in particular a ring of the form $A=\enn_R(R\oplus M)$ for some $R$-module $M$, and hence comes with a natural idempotent $e=\id_R$. This led Kalck and Yang to study relative singularity categories of algebras with idempotents in detail \cite{kalckyang} \cite{kalckyang2}, and the derived quotient $\dq$ already appears in their papers (although not by that name). Following their work, we investigate the relationship between $\dq$ and the dg singularity category ${D^\mathrm{dg}_\text{sg}(R)}$ in detail, and when the module $M$ defining $A$ satisfies some certain Cohen-Macaulay type conditions we show (\ref{qisolem}) that $\dq$ is the truncation to nonpositive degrees of the endomorphism dga of $M$ considered as an object of the dg singularity category ${D^\mathrm{dg}_\text{sg}(R)}$.

\p When $R=k\llbracket x_1,\ldots, x_n\rrbracket/\sigma$ is a complete local isolated hypersurface singularity, the two triangulated categories $D_\mathrm{sg}(R)$ and $\stab R$ are triangle equivalent to a third category, the category of \textbf{matrix factorisations} $\mathrm{MF}(\sigma)$. This has a natural enhancement over $\Z/2$-graded complexes -- and hence becomes a dg category by extending the $\Z/2$-graded morphism complexes periodically -- and the triangle equivalence between $\mathrm{MF}(\sigma)$ and $D_\mathrm{sg}(R)$ lifts to a quasi-equivalence of $\Z$-graded dg categories \cite{motivicsingcat}.

\p Dyckerhoff \cite{dyck} proved that the Hochschild cohomology of the 2-periodic dg category of matrix factorisations is the Milnor algebra of $R$, which recovers $R$ by the formal Mather-Yau theorem \cite{gpmather}. Recently, Dyckerhoff's theorem was improved upon by Hua and Keller \cite{huakeller}, who showed that the Milnor algebra is the zeroth Hochschild cohomology of the underlying $\Z$-graded dg category of matrix factorisations (which is a different object to the $\Z/2$-graded Hochschild cohomology). We use their result to prove a recovery theorem:
\begin{thmb}[\ref{recov}] Fix $n \in \N$ and let $R:=k\llbracket x_1,\ldots, x_n \rrbracket/\sigma$ be an isolated hypersurface singularity. Let $M$ be a maximal Cohen-Macaulay $R$-module that generates $D_\mathrm{sg}(R)$, and let $A:=\enn_{R}(R\oplus M)$ be the associated partial resolution of $R$ with $e=\id_{R}$. Assume that $A/^\mathbb{L}AeA$ is cohomologically locally finite and that $A/AeA$ is a local algebra. Then the quasi-isomorphism type of $\dq$ recovers the isomorphism type of $R$ amongst all rings satisfying the above conditions.
\end{thmb}
\p When $A$ is a resolution, $M$ automatically generates the singularity category (\ref{genrmk}). A sketch proof of Theorem B is as follows: because $R$ is a complete local hypersurface, the shift functor of ${D_\text{sg}(R)}$ is 2-periodic \cite{eisenbudper}. Using this we obtain a degree $-2$ element $\eta \in \dq$, homotopy unique up to multiplication by units in $H(\dq)$, with the property that $\eta:H^j(\dq) \to H^{j-2}(\dq)$ is an isomorphism for all $j\leq 0$. We prove that the derived localisation of $\dq$ at $\eta$ is the endomorphism dga of $M \in {D^\mathrm{dg}_\text{sg}(R)}$ (\ref{etaex}, \ref{uniqueeta}). Under some mild finiteness assumptions, $\dq$ determines the dg subcategory of ${D^\mathrm{dg}_\text{sg}(R)}$ generated by $M$ (\ref{recovwk}). By assumption this is all of ${D^\mathrm{dg}_\text{sg}(R)}$, and the result of Hua and Keller allows us to determine $R$ from this dg category. We conjecture (\ref{perconj}, \ref{perconjrmk}) that in fact $\dq$ determines the $\Z/2$-graded dg category of matrix factorisations, and hence that one can prove Theorem B using only Dyckerhoff's results plus the formal Mather-Yau theorem.
\subsection{Application: The derived contraction algebra}
Our eventual aim is to prove derived analogues of Donovan-Wemyss's results about contraction algebras \cite{DWncdf} \cite{contsdefs} \cite{enhancements}, and in the process recover some results from the recent preprint of Hua-Keller \cite{huakeller}. We briefly describe the construction and properties of the contraction algebra, and indicate how these will be generalised in future work \cite{dcalg}. Given a threefold flopping contraction of an irreducible rational curve $X \to \spec R$, Van den Bergh \cite{vdb} constructs a noncommutative partial resolution $A=\enn_R(R\oplus M)$ of $R$ (in the sense of \ref{partrsln}) equipped with a derived equivalence to $X$. The contraction algebra, which is an Artinian local algebra, is defined as the quotient $A/AeA$ of $A$ by $e=\id_R$. We will define the derived contraction algebra to be the derived quotient $\dq$. Let $S$ be the quotient of $A/AeA$ by its radical, regarded as an $A$-module (note that $S$ is a 1-dimensional $k$-vector space). Across the derived equivalence to $X$, the $A$-module $S$ corresponds to a twist of the structure sheaf of the flopping curve. Donovan-Wemyss prove that $A/AeA$ represents the functor of noncommutative deformations of $S$ \cite[3.1]{DWncdf}, which our Theorem A generalises to the derived setting. They conjecture that, when $A$ is smooth, $A/AeA$ is a complete invariant of $R$:
 \begin{conji}[\cite{DWncdf}, 1.4]
 	Let $X \to \spec R$ and $X' \to \spec R'$ be flopping contractions of an irreducible rational curve in a smooth projective threefold, where $R$ and $R'$ are complete local rings. If the associated contraction algebras are isomorphic, then $R \cong R'$.
 \end{conji}
In \cite{dcalg}, we will show that a derived version of the Donovan-Wemyss conjecture follows from Theorem B of this paper. The original conjecture would follow if one could prove that the usual contraction algebra determines the quasi-isomorphism type of the derived contraction algebra.

\p The deformation-theoretic description of $\dq$ as a Koszul dual will allow us to explicitly compute derived contraction algebras; we give an example in \ref{atiyahflop}. The deformation theory of \cite{huakeller} requires a smoothness assumption in order to apply the prorepresentability result of \cite{ELO2}, which in not required for our Theorem A. Our unifying approach will allow us in \cite{dcalg} to study derived contraction algebras associated to partial resolutions of Kleinian singularities, where the simple module $S$ frequently has infinite projective dimension. Donovan-Wemyss also prove that the contraction algebra sheafifies \cite{enhancements}; it would be interesting to check whether the derived contraction algebra sheafifies as well.

\subsection{Notation and conventions}
Throughout this paper, $k$ will denote an algebraically closed field of characteristic zero. We will need this assumption, although many assertions we make will work in much greater generality. Modules are right modules, unless stated otherwise. Consequently, noetherian means right noetherian, global dimension means right global dimension, et cetera. Unadorned tensor products are by default over $k$. All complexes, unless stated otherwise, are $\Z$-graded cochain complexes, i.e. the differential has degree 1. If $X$ is a complex, let $X[i]$ denote `$X$ shifted left $i$ times': the complex with $X[i]^j=X^{i+j}$ and obvious differential.  Recall that the \textbf{mapping cone} $\mathrm{cone}(f)$ of a degree zero map $f:X \to Y$ of complexes is (a representative of) the homotopy cokernel of $f$; concretely it is given by $X[1]\oplus Y$ with differential that combines $f$ with the differentials on $X$ and $Y$. The \textbf{mapping cocone} of $f$ is $\mathrm{cone}(f)[-1]$; it is a representative of the homotopy kernel. If $X$ is a complex of modules we will denote its cohomology complex by $H(X)$ or just $HX$. If $x$ is a homogeneous element of a complex of modules, we denote its degree by $|x|$.

\p A $k$-algebra is a $k$-vector space with an associative unital $k$-bilinear multiplication. A \textbf{differential graded algebra} (\textbf{dga} for short) over $k$ is a complex of $k$-vector spaces $A$ with an associative unital chain map $\mu:A\otimes A \to A$, which we refer to as the multiplication. Note that the condition that $\mu$ be a chain map forces the differential to be a derivation for $\mu$.  A $k$-algebra is equivalently a dga concentrated in degree zero, and a graded $k$-algebra is equivalently a dga with zero differential. We will sometimes refer to $k$-algebras as \textbf{ungraded algebras} to emphasise that they should be considered as dgas concentrated in degree zero. A dga is \textbf{graded-commutative} or just \textbf{commutative} if all graded commutator brackets $[x,y]=xy-(-1)^{|x||y|}yx$ vanish. Commutative polynomial algebras are denoted with square brackets $k[x_1,\ldots, x_n]$ whereas noncommutative polynomial algebras are denoted with angle brackets $k\langle x_1,\ldots, x_n\rangle$. 

\p If $A$ is an algebra, write $\cat{Mod-}A$ for its category of right modules, $\cat{mod-}A\subseteq\cat{Mod-}A$ for its category of finitely generated modules, $D(A):=D(\cat{Mod-}A)$ for its unbounded derived category, $D^b(A):=D^b(\cat{mod-}A)$ for its bounded derived category, and $\cat{per}(A) \subseteq D^b(A)$ for the subcategory on perfect complexes (i.e. those complexes quasi-isomorphic to bounded complexes of finitely-generated projective modules). Recall that an object $X$ of a category $\mathcal C$ is \textbf{compact} if $\hom_\mathcal{C}(X,-)$ commutes with filtered colimits. We then have $\cat{per}(A)\cong\{ \text{compact objects in } D(A)\}$.A \textbf{dg-module} (or just a \textbf{module}) over a dga $A$ is a complex of vector spaces $M$ together with an action map $A \otimes M \to M$ satisfying the obvious identities (equivalently, a dga map $A \to \enn_k(M)$). Note that a dg-module over an ungraded ring is exactly a complex of modules. If $A$ is a dga, write $D(A)$ for its unbounded derived category: this is the category of all dg-modules over $A$ localised along the quasi-isomorphisms. If $A$ is a dga we define $\cat{per}(A):=\{ \text{compact objects in } D(A)\}$.

\p We will freely use the formalism of Quillen model categories; references include \cite{dwyerspalinski} and \cite{hovey}. We will use the projective model structure on dg-modules (the `$q$-model structure' of \cite{sixmodels}), where every object is fibrant. Over an ungraded ring, every cofibrant complex is levelwise projective, and bounded above complexes are cofibrant precisely when they are levelwise projective. The homotopy category of the model category of dg-modules is exactly the unbounded derived category.

\p Let $V$ be a dg-$k$-vector space. The \textbf{total dimension} or just \textbf{dimension} of $V$ is $\sum_{n\in \Z}\mathrm{dim}_k V^n$. Say that $V$ is \textbf{finite-dimensional} or just \textbf{finite} if its total dimension is finite. Say that $V$ is \textbf{locally finite} if each $\mathrm{dim}_kV^n$ is finite. Say that $V$ is \textbf{cohomologically locally finite} if the cohomology dg-vector space $HV$ is locally finite. There are obvious implications $$\text{finite }\implies\text{ locally finite }\implies\text{cohomologically locally finite}.$$We use the same terminology in the case that $V$ admits extra structure (e.g. that of a dga).

\subsection{Structure of the paper}
The structure of this paper is as follows: in Section 2 we introduce the category of noncommutative pro-Artinian dgas, and prove a Koszul duality result for suitably finite dgas (\ref{kdfin}), which can be viewed as a strictification result. In Section 3 we recall some noncommutative derived deformation theory, and prove that the functor of framed deformations of a simple module is prorepresentable by the Koszul dual of the controlling dga (\ref{prorepfrm}). In Section 4 we introduce the derived quotient, prove Theorem A (\ref{maindefmthm}), and introduce the dg singularity category (\ref{dgsing}). In Section 5 we consider partial resolutions of Gorenstein rings (\ref{partrsln}), and we conclude with Theorem B (\ref{recov}).

\subsection{Acknowledgements} The author is a PhD student at the University of Edinburgh, supported by the Engineering and Physical Sciences Research Council, and this work will form part of his thesis. He would like to thank his supervisor Jon Pridham for his guidance, as well as Michael Wemyss for his continued support. He would like to thank Zheng Hua and Bernhard Keller for their helpful comments, including pointing out a mistake in an early proof of Theorem B. He would also like to thank Jenny August, Martin Kalck, and Dong Yang for useful discussions and comments.

\section{Koszul duality}\label{kd}
\p In this section we will introduce the Koszul dual of a dga, and prove a duality result (\ref{kdfin}) for a large class of reasonably finite dgas. Section \ref{defmthy} will give a deformation-theoretic interpretation of these results. In this section every dga we consider will be \textbf{augmented}, meaning the canonical map $k \to A$ admits a retraction. The \textbf{augmentation ideal} of an augmented dga is $\bar A:=\ker(A \to k)$. Sending $A$ to $\bar A$ sets up an equivalence between augmented dgas and nonunital dgas. The inverse functor freely appends a unit, and indeed $A$ is isomorphic to $\bar{A}\oplus k$ as augmented dgas. Say that a augmented dga $A$ is \textbf{Artinian local} if $A$ has finite total dimension and $\bar A$ is nilpotent. In other words, Artinian local means `finite-dimensional over $k$, and local with residue field $k$'. Most dgas of interest to us in this section will be concentrated in nonpositive cohomological degrees.
\subsection{Bar and cobar constructions}
In this part we will follow Chapters 1 and 2 of Loday-Vallette \cite{lodayvallette}. Just like a dga is a monoid in the monoidal category of dg-vector spaces over $k$, a \textbf{differential graded coalgebra} or \textbf{dgc} for short is a comonoid in this category. More concretely, a dgc is a dg-$k$-vector space $(C,d)$ equipped with a comultiplication $\Delta:C \to C \otimes C$ and a counit $\epsilon:C \to k$, satisfying the appropriate coassociative and counital identities, and such that $d$ is a coderivation for $\Delta$.  A \textbf{coaugmentation} on a dgc is a section of $\epsilon$; if $C$ is coaugmented then $\bar C:= \ker \epsilon$ is the \textbf{coaugmentation coideal}. It is a dgc under the reduced coproduct $\bar{\Delta}x = \Delta x -x\otimes 1 -1\otimes x$, and $C$ is isomorphic as a nonunital dgc to $\bar C \oplus k$.. A coaugmented dgc $C$ is \textbf{conilpotent} if every $x \in \bar C$ is annihilated by some suitably high power of $\Delta$.
\begin{ex}
	If $V$ is a dg-vector space, then the tensor algebra ${T}^c(V)$ is a dg-coalgebra when equipped with the \textbf{deconcatenation coproduct} ${T}^c(V) \to {T}^c(V) \otimes {T}^c(V)$ which sends $v_1\cdots v_n$ to $\sum_i v_1\cdots v_i \otimes v_{i+1}\cdots v_n$. The differential is induced from the differential on $V^{\otimes n}$. It is easy to see that ${T}^c(V)$ is conilpotent, since $\Delta^{n+1}(v_1\cdots v_n)=0$. In fact, ${T}^c$ is the cofree conilpotent coalgebra functor: if $C$ is conilpotent then $C \to{T}^c(V)$ is determined completely by the composition $l:C \to {T}^c(V) \to V$.
\end{ex}
\begin{defn}
	Let $A$ be an augmented dga. Put $V:=\bar A [1]$, the shifted augmentation ideal. Let $d_V$ be the usual differential on $TV$. Let $d_B$ be the \textbf{bar differential}: send $a_1\otimes\cdots \otimes a_n$ to the signed sum over $i$ of the $a_1\otimes\cdots\otimes a_ia_{i+1}\otimes\cdots \otimes a_n$ and extend linearly. The signs come from the Koszul sign rule; see \cite[2.2]{lodayvallette} for a concrete formula. One can check that $d_B$ is a degree 1 map from $V^{\otimes n+1} \to V^{\otimes n}$, and that it intertwines with $d_V$. Hence, one obtains a third and fourth quadrant bicomplex $C$ with rows $V^{\otimes{n}}[-n]$. By construction, the direct sum total complex of $C$ is $TV$, with a new differential $\partial=d_V+d_B$. The \textbf{bar construction} of $A$ is the complex $BA:=(TV,\partial)$. One can check that the deconcatenation coproduct makes $BA$ into a dgc. Note that the degree 0 elements of $A$ become degree $-1$ elements of $BA$.
\end{defn}
\begin{ex}Let $A$ be the graded algebra $k[\epsilon]/\epsilon^2$, with $\epsilon$ in degree 0. Then $\bar A [1]$ is $k\epsilon$ placed in degree $-1$. Since $\epsilon$ is square zero, the bar differential is identically zero. So $BA$ is the tensor coalgebra $k[\epsilon]$, with $\epsilon$ in degree -1.
\end{ex}

\begin{defn}
	Let $C$ be a coaugmented dgc. One can analogously define a \textbf{cobar differential} $d_\Omega$ on the tensor algebra $T(\bar{C}[-1])$ by sending $c_1\otimes\cdots \otimes c_n$ to the signed sum over $i$ of the ${c_1\otimes\cdots \otimes  \bar{\Delta}c_i\otimes\cdots c_n}$, and the \textbf{cobar construction} on $C$ is the dga $\Omega C:=(T(\bar{C}[-1]),d_C+d_\Omega)$.
\end{defn}

Bar and cobar are adjoints:
\begin{thm}[\cite{lodayvallette}, 2.2.6]
	If $A$ is an augmented dga and $C$ is a conilpotent dgc, then there is a natural isomorphism $$\hom_{\cat{dga}}(\Omega C, A)\cong \hom_{\cat{dgc}}(C, BA)$$
\end{thm}
The bar construction preserves quasi-isomorphisms; the idea is to filter $BA$ by setting $F_p BA$ to be the elements of the form $a_1\otimes\cdots\otimes a_n$ with $n\leq p$, and look at the associated spectral sequence \cite[2.2.3]{lodayvallette}. However, the cobar construction does not preserve quasi-isomorphisms in general. Bar and cobar give canonical resolutions:
\begin{thm}[\cite{lodayvallette}, 2.3.2]
	Let $A$ be an augmented dga. Then the counit $\Omega B A \to A$ is a quasi-isomorphism. Similarly, the unit is a quasi-isomorphism for coaugmented conilpotent dgcs.
\end{thm}
\subsection{Koszul duality for Artinian local algebras}
\begin{prop}\label{barcofiblem}
	Let $A$ be an augmented dga. Let $S$ be $k$, considered as an $A$-module via the augmentation. Then $BA$ is a model for the derived tensor product coalgebra $S^* \lot_A S$.
\end{prop}
\begin{proof}
	Let $\tilde S$ be the totalisation of the bar double complex $\cdots\to S\otimes_k A\otimes_k A \to S\otimes_k A$, which resolves $S^*$. It is standard \cite[2.2.2]{lodayvallette} that $BA$ is quasi-isomorphic to $\tilde S \otimes_A S$. So the statement follows if we can show that $\tilde S$ is $A$-cofibrant. If $A$ is an ungraded algebra, then this is clear since $\tilde S$ is levelwise projective. Since $A$ might be an unbounded dga, we have to do a little more work. We need to show that $\tilde S$ is semifree: i.e. it has an increasing exhaustive filtration whose quotients are free over $A$ (in other words, this is a cell decomposition of $\tilde S$). But this is easy to show: just filter $\tilde S$ in the direction of the bar differential.
\end{proof}

Let $(C,\Delta,\epsilon)$ be a dgc. Then $\Delta$ dualises to a map $(C\otimes C)^* \to C^*$, and combining this with the natural inclusion $C^* \otimes C^* \to (C \otimes C)^*$ yields a semigroup structure on $C^*$. In fact, $(C^*,\Delta^*,\epsilon^*)$ is a dga. The dual statement is not in general true -- if $A$ is a dga then the multiplication need not dualise to a map $A^* \otimes A^* \to A^*$. However, if $A$ is an Artinian local dga, then it does (because the natural map $A^* \otimes A^* \to (A \otimes A)^*$ is an isomorphism), and indeed $A^*$ is a dgc. If $C$ is coaugmented, then $C^*$ is augmented, and if $C$ is conilpotent, then $\bar{C^*}$ is nilpotent. Similarly, if $A$ is Artinian local then $A^*$ is coaugmented and conilpotent.

\begin{defn}Let $A$ be an augmented dga. The \textbf{Koszul dual} of $A$ is the dga $A^!:=(BA)^*$.
\end{defn}
Note that $BA$ is coaugmented, so $A^!$ is again augmented. Because both $B$ and the linear dual preserve quasi-isomorphisms, so does $A \mapsto A^!$. Let $A$ be an augmented dga, and let $S$ be the $A$-module $k$. Then Lemma \ref{barcofiblem} combined with the (derived) hom-tensor adjunction gives a quasi-isomorphism of $A^!$ with $\R\enn_A(S)$.  The key statement about the Koszul dual is the following:
\begin{prop}
Let $A$ be a nonpositive Artinian local dga. Then $A^!$ is naturally isomorphic as a dga to $\Omega(A^*)$.\end{prop}
\begin{proof}For brevity, we will replace $A$ by its augmentation ideal $\bar A$. The dgc $BA$ is the direct sum total complex of the double complex whose rows are $A^{\otimes n}$. Hence, $A^!$ is the direct product total complex of the double complex with rows $(A^{\otimes n})^*$. However, because $A$ was nonpositive $A^!$ is also the direct sum total complex of this double complex. Because $A$ is nonpositive and locally finite, the natural map $(A^*)^{\otimes n} \to (A^{\otimes n})^*$ is an isomorphism. Hence $A^!$ is the direct sum total complex of the double complex with rows $(A^*)^{\otimes n}$, which -- after checking that the bar differential dualises to the cobar differential -- is precisely the definition of $\Omega (A^*)$.
\end{proof}
\begin{cor}\label{kdforart}
Let $A$ be a nonpositive Artinian local dga. Then $A$ is naturally quasi-isomorphic to $A^{!!}$.
\end{cor}
\begin{proof}
$A^{!!}$ is by definition $(B(A^!))^*$. By the previous Proposition, $A^!$ is isomorphic to $\Omega(A^*)$. Because $C \to B \Omega C$ is a quasi-isomorphism for conilpotent dgcs, $A^* \to B(A^!)$ is a dgc quasi-isomorphism. Dualising and using exactness of the linear dual gets us a dga quasi-isomorphism $A^{!!}\to A^{**}$. But $A$ is Artinian local, and hence isomorphic to $A^{**}$.
\end{proof}
\begin{rmk}Note that we did not use the local hypothesis; we just used that $A$ was nonpositive and finite.
\end{rmk}

\subsection{The model structure on pro-Artinian algebras}
\begin{defn}[\cite{kashschap}, \S6]\label{procats}
Let $\mathcal{C}$ be a category. A \textbf{pro-object} in $\mathcal{C}$ is a formal cofiltered limit, i.e. a diagram $J \to \mathcal{C}$ where $J$ is a small cofiltered category. We denote such a pro-object by $\{C_j\}_{j \in J}$. The category of pro-objects $\cat{pro}\mathcal{C}$ has morphisms $$\hom_{\cat{pro}\mathcal{C}}(\{C_i\}_{i \in I}, \{D_j\}_{j \in J}) := \varprojlim_j \varinjlim_i \hom_{\mathcal{C}}(C_i,D_j).$$
\end{defn}
If $\mathcal{C}$ has cofiltered limits, then there is a `realisation' functor $\varprojlim: \cat{pro}\mathcal{C} \to \mathcal{C}$. If $C$ is a constant pro-object, then it is easy to see that one has $\hom_{\cat{pro}\mathcal{C}}(C, \{D_j\}_{j \in J}) \cong  \hom_{\mathcal{C}}(C,\varprojlim_jD_j)$.
\begin{defn}
Let $\mathcal{C}$ be a category. The \textbf{ind-category} of $\mathcal{C}$ is $\cat{ind}\mathcal{C}:=\cat{pro}(\mathcal{C}^\text{op})^\text{op}$ (c.f. \ref{procats} for the definition of pro-categories). Less abstractly, an object of $\cat{ind}\mathcal{C}$ is a formal filtered colimit $J \to \mathcal{C}$, and the morphisms are $\hom_{\cat{pro}\mathcal{C}}(\{C_i\}_{i \in I}, \{D_j\}_{j \in J}) := \varprojlim_i \varinjlim_j \hom_{\mathcal{C}}(C_i,D_j)$.
\end{defn}
If $\mathcal{C}$ has filtered colimits, then there is a `realisation' functor $\varinjlim: \cat{ind}\mathcal{C} \to \mathcal{C}$. In this situation, if $D \in \mathcal{C}$ is a constant ind-object then one has $\hom_{\cat{ind}\mathcal{C}}(\{C_i\}_{i \in I}, D)\cong \hom_\mathcal{C}( \varinjlim_iC_i,D)$. 
\begin{defn}
Let $\dga$ be the the category of nonpositively cohomologically graded augmented dgas over $k$, and let $\dgart$ be the subcategory on Artinian local dgas. We refer to an object of $\proart$ as a \textbf{pro-Artinian dga}.
\end{defn}
\begin{rmk}
We caution that in this paper, ``pro-Artinian'' means ``pro-(Artinian local)''. We will not consider non-local profinite dgas.
\end{rmk}
There is a limit functor $\varprojlim: \proart \to \dga$ which sends a cofiltered system to its limit. Moreover, this is right adjoint to the functor $\dga \to \proart$ which sends a dga $A$ to the cofiltered system $\hat A$ of its Artinian local quotients. We list some standard results on the structure of $\proart$.
\begin{prop}[\cite{kashschap}, 6.1.14]
Let $f:A \to B$ be a morphism in $\proart$. Then $f$ is isomorphic to a level map: a collection of maps $\{f_\alpha: A_\alpha \to B_\alpha\}_{\alpha \in I}$ between Artinian local algebras, where $I$ is cofiltered.
\end{prop}
\begin{prop}[\cite{grothendieckpro}, Corollary to 3.1]\label{strictpro}
Every object of $\proart$ is isomorphic to a \textbf{strict} pro-object, i.e. one for which the transition maps are surjections.
\end{prop}
\begin{cor}\label{limisexact}
The functor $\varprojlim: \proart \to \dga$ is exact.
\end{cor}
\begin{proof}
A strict pro-object satisfies the Mittag-Leffler condition.
\end{proof}
We aim to enhance the limit functor $\varprojlim: \proart \to \dga$ to a right Quillen functor. We start by describing the model structures involved.
\begin{thm}[\cite{hinichhom}]
The category $\dga$ is a model category, with weak equivalences the quasi-isomorphisms and fibrations the levelwise surjections.
\end{thm}
Note that in $\dga$, every object is fibrant, and the cofibrant objects are precisely the semifree dgas (i.e. those that become free graded algebras after forgetting the differential). For example, $\Omega B A \to A$ is a functorial cofibrant resolution of $A$. 
\begin{thm}\label{proartmodel}
The category $\proart$ is a model category, with weak equivalences those maps $f$ for which each $H^n f: H^nA \to H^nB$ is an isomorphism of profinite $k$-vector spaces, and fibrations those maps $f$ for which $\varprojlim f$ is a levelwise surjection.
\end{thm}
\begin{proof}
The proof of \cite[4.3]{unifying} adapts to the noncommutative case.
\end{proof}
\begin{prop}\label{limreflects}
	The functor $\varprojlim: \proart \to \dga$ both preserves and reflects weak equivalences.
\end{prop}
\begin{proof}
	Every vector space is canonically ind-finite \cite[1.3]{unifying}, so that the linear dual provides a contravariant equivalence from $\cat{pro}(\cat{fd-vect}_k)$ to $\cat{vect}_k$. In other words, let $g:U \to V$ be a map of profinite vector spaces. We may take $g$ to be a level map $\{g_\alpha:U_\alpha \to V_\alpha\}_\alpha$. Then $g$ is an isomorphism if and only if $\varinjlim_\alpha g_\alpha^*: \varinjlim_\alpha V_\alpha^* \to \varinjlim_\alpha U_\alpha^*$ is an isomorphism. Dualising again, we obtain a map $\varprojlim_\alpha U_\alpha^{**} \to \varprojlim_\alpha V_\alpha^{**}$, which canonically agrees with $\varprojlim g = \varprojlim_\alpha g_\alpha$ since the $U_\alpha$ and $V_\alpha$ are finite-dimensional. Hence, $g$ is an isomorphism if and only if $\varprojlim g$ is. Similarly, $g$ is a injection or a surjection if and only if $\varprojlim g$ is. Hence, $\varprojlim:\cat{pro}(\cat{fd-vect}_k) \to \cat{vect}_k$ is exact. Let $f$ be a morphism in $\proart$. By definition, $f$ is a weak equivalence if and only if each $H^nf \in \cat{pro}(\cat{fd-vect}_k)$ is an isomorphism. But this is the case if and only if each $\varprojlim H^n f$ is an isomorphism. Because $\varprojlim:\cat{pro}(\cat{fd-vect}_k) \to \cat{vect}_k$ is exact, this is the case if and only if each $H^n\varprojlim f$ is an isomorphism, which is exactly the condition for $\varprojlim f$ to be a weak equivalence.
\end{proof}
\begin{prop}
The functor $\varprojlim: \proart \to \dga$ is right Quillen.
\end{prop}
\begin{proof}
It is clear that $\varprojlim$ preserves fibrations. It preserves all weak equivalences by \ref{limreflects}.
\end{proof}
\begin{rmk}
	Morally, one would like to say that the model structure on $\proart$ is transferred from that of $\dga$ along the right adjoint $\varprojlim$, since one has that a map $f$ of pro-Artinian algebras is a fibration or a weak equivalence precisely when $\varprojlim f$ is. However, this is not strictly the case: the standard argument (as in e.g. \cite{cranstransfer}) requires that the left adjoint $A \mapsto \hat A$ preserves small objects. If this is the case, then since the ungraded algebra $k[x]$ is small in $\dga$, the object $\widehat{k[x]}$ is small in $\proart$. However, let $V$ be any infinite-dimensional vector space, and let $V_i$ be its filtered system of finite-dimensional subspaces. Let $k\oplus V_i$ be the square-zero extension (see \ref{dersctn} for a definition). In $\proart$, the colimit of the $k\oplus V_i$ is the square-zero extension $k\oplus V^{**}$. One can check that $\varinjlim_i\hom(\widehat{k[x]},k\oplus V_i)\cong V$, but that $\varinjlim_i\hom({\widehat{k[x]},k\oplus V^{**}})\cong V^{**}$. Hence, $\widehat{k[x]}$ is not small.
\end{rmk}

\subsection{The model structure on conilpotent coalgebras}\label{coalgmodel}
\begin{thm}[\cite{lefevre}, 1.3.1.2a]
The category $\cat{cndgc}_k$ of coaugmented conilpotent dgcs admits a model structure where the weak equivalences $f$ are those maps for which $\Omega f$ is an algebra quasi-isomorphism, and the cofibrations are the levelwise monomorphisms.
\end{thm}
Every weak equivalence is a quasi-isomorphism, but the converse is not true \cite[2.4.3]{lodayvallette}. Every object in $\cat{cndgc}_k$ is cofibrant, and the fibrant objects are the semicofree conilpotent coalgebras, i.e. those that are tensor coalgebras after forgetting the differential (the tensor coalgebra $T^cV$ is the cofree conilpotent coalgebra on $V$). The bar-cobar resolution $C \to B\Omega C$ provides a functorial fibrant resolution, just as the cobar-bar resolution $\Omega B A \to A$ is a cofibrant dga resolution:
\begin{prop}Let $C$ be a conilpotent dgc. The natural map $C \to B\Omega C$ is a fibrant resolution.
\end{prop}
\begin{proof}
	Follows from 1.3.2.2 and 1.3.2.3 of \cite{lefevre}.
\end{proof}
\begin{prop}[\cite{lefevre}, 1.3.1.2b]
The pair $(\Omega, B)$ is a Quillen equivalence between $\cat{cndgc}_k$ and $\cat{dga}_k$, the category of unbounded dgas.
\end{prop}

\begin{prop}\label{sharpprop}
The functor $\{A_\alpha\}_\alpha \mapsto\varinjlim_\alpha A_\alpha^*$ is an equivalence $\proart\to (\cat{cndgc}^{\geq 0}_k)^\text{op}$.
\end{prop}
\begin{proof}
Via the linear dual, an Artinian local dga is the same thing as a finite-dimensional coaugmented conilpotent dg-coalgebra over $k$. This provides an equivalence between $(\proart)^\text{op}$ and $\cat{ind}(\cat{fd-cndgc}_k^{\geq 0})$. A classical theorem of Sweedler says that a (non-dg) coalgebra is the filtered colimit of its finite-dimensional subcoalgebras. The same remains true for dgcs \cite[1.6]{getzlergoerss}. In particular, the image of the colimit functor $\varinjlim: \cat{ind}(\cat{fd-cndgc}_k^{\geq 0}) \to \cat{dgc}_k^{\geq 0}$ is precisely $\cat{cndgc}_k^{\geq 0}$. I claim that $\varinjlim$ is fully faithful. For this it is enough to prove that finite-dimensional conilpotent coalgebras are compact: i.e. given a finite-dimensional conilpotent dgc $C$, and $D=\{D_\alpha\}_\alpha$ a filtered system of finite dgcs, that there is a natural isomorphism $\varinjlim_\alpha\hom(C,D_\alpha)\to \hom(C,\varinjlim_\alpha D_\alpha)$. This follows because since $C$ is finite-dimensional, every map $C \to \varinjlim_\alpha D_\alpha $ must factor through one of the $D_\alpha$ \cite[1.9]{getzlergoerss}. 
\end{proof}
\begin{defn}\label{csharp}
If $C$ is a nonnegative dgc, let $C^\sharp\in \proart$ denote the levelwise dual of its filtered system of finite-dimensional sub-dgcs.
\end{defn}
It is easy to see that $C \mapsto C^\sharp$ is the inverse functor to  $\{A_\alpha\}_\alpha \mapsto\varinjlim_\alpha A_\alpha^*$, and that $C^*$ and $ \varprojlim C^\sharp$ are isomorphic dgas.
\begin{prop}
The functor $C \mapsto C^\sharp$ is a Quillen equivalence $(\cat{cndgc}^{\geq 0}_k)^\text{op} \to \proart$.
\end{prop}
\begin{proof}
If one restricts the model structure on $\cat{cndgc}_k$, then one gets a model structure on $\cat{cndgc}^{\geq 0}_k$, with weak equivalences precisely the quasi-isomorphisms of dgcs \cite[1.3.5.1]{lefevre}.  Moreover, the bar-cobar resolution $C \to B\Omega C$ is still a fibrant resolution functor (as long as one takes the good truncation to get a weakly equivalent dgc concentrated in nonnegative degrees). I claim that the equivalence $C \mapsto C^\sharp$ is actually a Quillen equivalence. For this, it is enough to check that it is right Quillen. Let $g^\text{op}: D \to C$ be a fibration in $(\cat{cndgc}^{\geq 0}_k)^\text{op}$, i.e. $g:C \to D$ is a levelwise injection of dgcs. Hence, $g^*: D^* \to C^*$ is a levelwise surjection of dgas. But $g^*\cong \varprojlim g^\sharp$, so that $g^\sharp$ is a fibration in $\proart$. Similarly, suppose that $g^\text{op}: D \to C$ is a weak equivalence, i.e. $g: C \to D$ is a quasi-isomorphism. Dualising, $g^*\cong \varprojlim g^\sharp$ is a quasi-isomorphism, and so $g^\sharp$ is a weak equivalence, since, by \ref{limreflects}, $f$ is a weak equivalence of pro-Artinian dgas if and only if $\varprojlim f$ is a quasi-isomorphism.
\end{proof}
So we have a diagram of Quillen adjunctions (we only draw the right adjoints) $$(\cat{cndgc}^{\geq 0}_k)^\text{op} \xrightarrow{C \mapsto C^\sharp} \proart \xrightarrow{\varprojlim} \cat{dga}_k \xrightarrow{B} \cat{cndgc}_k$$where the left-hand map is an equivalence. Note that the composition is not isomorphic to the identity, even on the level of homotopy categories: a nonnegatively graded dgc $C$ is mapped to $B(C^*)$, which has essentially no chance of being even quasi-isomorphic to $C$.

\subsection{Koszul duality for pro-Artinian algebras}
\begin{defn}
	Say that a nonpositive dga $A\in \dga$ is \textbf{good} if \begin{itemize}
		\item $A$ is quasi-isomorphic to $\varprojlim \mathcal{A}$ for some pro-Artinian dga $\mathcal{A}$.
		\item $A$ is cohomologically locally finite.
	\end{itemize}
\end{defn}
In the presence of the second condition, the first condition is actually equivalent to requiring simply that $H^0(A)$ is an Artinian local algebra. One direction of this equivalence is clear: if $A$ is quasi-isomorphic to $\varprojlim \mathcal{A}$ and $H^0(A)$ is finite, then it is a finite-dimensional local ring and hence Artinian local. The other direction is a nontrivial result we later prove as \ref{goodchar}, which will require deformation-theoretic methods. We will of course not use this fact until then. We are about to prove a Koszul duality result for the class of good dgas -- we begin by noting an important finiteness property.
\begin{prop}\label{kdfinite}
	Let $A\in \dga$ be a nonpositive dga. If $A$ is cohomologically locally finite then so are $BA$ and $A^!$.
\end{prop}
\begin{proof}Since the linear dual is exact, the statement for $A^!$ is equivalent to the statement for $BA$. To prove the latter, filter $BA$ by the tensor powers of $A$ to obtain a spectral sequence ${H^p(A^{\otimes q}) \Rightarrow H^{p-q}(BA)}$. Since there are only finitely many nonzero $H^p(A^{\otimes q})$ with $p-q$ fixed, and they are all finite-dimensional, $H^{p-q}(BA)$ must also be finite-dimensional.
\end{proof}
\begin{rmk}
One can also prove \ref{kdfinite} by applying the $A_\infty$ bar construction to an $A_\infty$ minimal model for $A$, which yields a locally finite model for $BA$.	
\end{rmk}
\begin{thm}\label{kdgood}
	Let $A$ be a dga. If $A$ is good then $A$ is quasi-isomorphic to its double Koszul dual.
\end{thm}
\begin{proof}
	By assumption there is a pro-Artinian dga $\mathcal{A}$ and a quasi-isomorphism between $A$ and $\varprojlim \mathcal{A}$. Since the bar construction and the Koszul dual preserve quasi-isomorphisms, we may assume that $A=\varprojlim \mathcal{A}$. Moreover, we may assume that $\mathcal{A}$ is strict (\ref{strictpro}). Let $\mathcal{C}=\mathcal{A}^*$ be the corresponding ind-conilpotent dgc, and put $C:= \varinjlim \mathcal{C}$. It is clear that $\mathcal{A}^! \cong \Omega \mathcal{C}$ as ind-dgas. Taking colimits and using cocontinuity of $\Omega$ we get $\varinjlim \mathcal{A}^! \cong \Omega C$. Hence, $B\varinjlim \mathcal{A}^!$ is weakly equivalent to $C$. Dualising, we see that $(\varinjlim \mathcal{A}^!)^!$ is quasi-isomorphic to $C^* \cong A$. So it is enough to show that $\varinjlim \mathcal{A}^!$ is quasi-isomorphic to $A^!$.
	
	\p Put $\mathcal{D}:=B\mathcal{A}$; it is a pro-conilpotent dgc that is concentrated in nonpositive degrees. Put $D:=\varprojlim \mathcal{D}$, where we take the limit in the category of conilpotent coalgebras. Note that this limit exists because the category of conilpotent coalgebras is a coreflective subcategory of the category of all coalgebras \cite[1.3.33]{aneljoyal}, and a coreflective subcategory of a complete category is complete. Since $B$ is continuous we have $D \cong BA$. There is a natural algebra map $\phi:\varinjlim \mathcal{D}^* \to D^*$. Note that $\mathcal{D}^*=\mathcal{A}^!$ and that $D^* = A^!$, so it is enough to show that $\phi$ is a quasi-isomorphism. For $n \in \Z$, consider the induced linear map $$\psi_n:\qquad\varinjlim(H^n(\mathcal{D}^*)) \xrightarrow{\cong}H^n(\varinjlim\mathcal{D}^*) \xrightarrow{H^n\phi} H^n(D^*) \xrightarrow{\cong}H^{-n}(D)^*$$where we have used exactness of filtered colimits and the linear dual, and dualise it to obtain a map $$\chi_n:\qquad H^{-n}(D) \to H^{-n}(D)^{**} \xrightarrow{\psi_n^*} (\varinjlim(H^n(\mathcal{D}^*)))^* \xrightarrow{\cong} \varprojlim H^n(\mathcal{D^*})^{*} \xrightarrow{\cong} \varprojlim H^{-n}(\mathcal{D}^{**})$$where we have used exactness of the linear dual again along with the fact that contravariant Hom sends colimits to limits. By \ref{kdfinite}, $D=BA$ is cohomologically locally finite, which means that $\qquad H^{-n}(D) \to H^{-n}(D)^{**}$ is an isomorphism. Similarly, each level ${\mathcal{D}_\alpha}$ of $\mathcal{D}$ is locally finite, since it is the bar construction on an Artinian local dga. In particular, the natural map $H^{-n}(\mathcal{D}_\alpha) \to H^{-n}(\mathcal{D}_\alpha^{**})$ which sends $[v]$ to $[\mathrm{ev}_v]$ is an isomorphism. Let $[u] \in H^{-n}(D)$; one can compute that $\chi_n([u])=([\mathrm{ev}_{u_\alpha}])_\alpha$, where $u_\alpha$ is the image of $u$ under the natural map $D \to \mathcal{D}_\alpha$. Hence, the composition $H^{-n}(D) \xrightarrow{\chi_n}\varprojlim H^{-n}(\mathcal{D}^{**}) \xrightarrow{\cong} \varprojlim H^{-n}(\mathcal{D})$ of $\chi_n$ with the inverse to the natural isomorphism sends $[u]$ to $[u_\alpha]_\alpha$. But this is precisely the natural map $H^{-n}(D) \to \varprojlim H^{-n} (\mathcal{D})$. Since $B$ preserves surjections, and we chose $\mathcal{A}$ to be strict, $\mathcal{D}=B\mathcal{A}$ is a strict pro-dgc, and in particular satisfies the Mittag-Leffler condition. Hence, the natural map $H^{-n}(D) \to \varprojlim H^{-n} (\mathcal{D})$ is an isomorphism for all $n$. Now it follows that $\chi_n$, $\psi_n$, and $H^n\phi$ are isomorphisms for all $n$. Hence, $\phi$ is a quasi-isomorphism.
\end{proof}
\begin{rmk}
Note that instead of requiring that $A$ itself be cohomologically locally finite, it is enough to require that $BA$ is cohomologically locally finite. If the dgc $BA$ is cohomologically locally finite and admits a minimal model, then $A$ has a resolution $\Omega B A$ with finitely many generators in each level, which can be thought of as a finiteness condition.
\end{rmk}

\subsection{Derivations}\label{dersctn}
If $A\to B$ is a map of commutative $k$-algebras and $M$ is a $B$-module, then a derivation $A \to M$ is the same as a map of $B$-augmented $k$-algebras $A \to B \oplus M$, where $B \oplus M$ is the square-zero extension of $B$ by $M$. When $A=B$ and the map is the identity, a derivation is the same as a section of the projection $A\oplus M \to A$. In this part, we make the same observation in the noncommutative derived world. We prove a key technical result stating that derived derivations from a good dga are the same as derived derivations from its limit (\ref{rder}). Note that when our algebras are noncommutative, one must use bimodules in order to talk about derivations. We broadly follow Tabuada \cite{tabuadaNCAQ}, who is generalising the seminal work of Quillen \cite{quillender} for ungraded commutative algebras. A reference for ungraded noncommutative algebras is \cite{ginzburgnc}.
\begin{defn}
Let $B$ be a dga and $M$ a $B$-bimodule. The \textbf{square-zero extension} of $B$ by $M$ is the dga $B \oplus M$ whose underlying dg-vector space is $B\oplus M$, with multiplication given by $(b,m).(b',m')=(bb',bm+mb')$. If $A \to B$ is a dga map then a \textbf{derivation} $A \to M$ is a map of $B$-augmented dgas $A \to B\oplus M$, which is equivalently a morphism $A \to B\oplus M$ in the overcatgory $\cat{dga}_k / B$. The set of derivations $A \to M$ is $\mathrm{Der}_B(A,M):=\hom_{\cat{dga}_k / B}(A, B \oplus M)$.
\end{defn}
One can easily check that a derivation $A \to M$ is the same as an $A$-linear map $A \to M$ satisfying the graded Leibniz formula.
\begin{prop}[\cite{tabuadaNCAQ}, 4.6]
The square-zero extension functor $B-\cat{bimod} \to \cat{dga}_k / B$ admits a left adjoint $A\mapsto \Omega(A)_B$, which we refer to as the functor of \textbf{noncommutative K\"ahler differentials}.
\end{prop}
\begin{proof}
An application of Freyd's adjoint functor theorem.
\end{proof}
\begin{rmk}
If $A \to B$ is a morphism of commutative $k$-algebras, then $\Omega(A)_B$ does not agree with the usual commutative K\"ahler differentials. Indeed, one has $\Omega(A)_A\cong \ker(\mu: A \otimes_k A \to A)$, and the commutative K\"ahler differentials are $\Omega(A)_A/\Omega(A)_A^2$. In general one has $\Omega(A)_B\cong B\otimes_A \Omega(A)_A \otimes_A B$ which is the pullback of the bimodule $\Omega(A)_A$ along $A \to B$.
\end{rmk}
\begin{cor}
Let $A \to B$ be a dga map and $M$ a $B$-bimodule. Then the set $\mathrm{Der}_B(A,M)$ is naturally a dg-vector space.
\end{cor}
\begin{proof}
The category of dg-$B$-bimodules is naturally enriched over $\cat{dgvect}_k$.
\end{proof}
The category $\cat{dga}_k / B$ is a model category, with model structure induced from that on $\cat{dga}_k$. The category $B-\cat{bimod}:=B\otimes_k B^{\text{op}}-\cat{Mod}$ is also a model category in the usual way. It is easy to see that the square-zero extension functor $B\oplus-:B-\cat{bimod} \to \cat{dga}_k/B$ is right Quillen. Since every object in $B-\cat{bimod}$ is fibrant, $B\oplus-$ is its own right derived functor. However, since not every dga is cofibrant, the noncommutative K\"ahler differentials have a nontrivial left derived functor.
\begin{defn}
The \textbf{noncommutative cotangent complex} functor is $\mathbb{L}(-)_B:=\mathbb{L}\Omega(-)_B$, the total left derived functor of $\Omega(-)_B$.
\end{defn}
The model category $B-\cat{bimod}$ is a dg model category, in the sense that it is enriched over $\cat{dgvect}_k$ in a way compatible with the model structure. We denote the enriched hom-complexes by $\dgh$ (see \ref{dgcats} for more about dg categories). The interested reader should consult Hovey \mbox{\cite[4.2.18]{hovey}} for a rigorous definition of enriched model category; in the terminology used there a dg model category is a $\mathrm{Ch}(k)$-model category. In particular, $B-\cat{bimod}$ has a well-defined notion of derived hom-complexes, and we may use the Quillen adjunction with $\cat{dga}_k / B$ to define complexes of derived derivations.
\begin{defn}
Let $A \to B$ be a dga map and let $M$ be a $B$-bimodule. Let $QA \to A$ be a cofibrant resolution. The space of \textbf{derived derivations} from $A$ to $M$ is the dg-vector space $$\R\mathrm{Der}_B(A,M):=\mathrm{Der}_B(QA,M)\cong\dgh_B(\mathbb{L}(A)_B,M)$$where we use the notation $\dgh$ to mean the enriched hom.
\end{defn}
Different choices of resolution for $A$ yield quasi-isomorphic spaces of derived derivations. One has an isomorphism $H^0(\R\mathrm{Der}_B(A,M))\cong \hom_{\mathrm{Ho}(\cat{dga}_k / B)}(A, B \oplus M)$.
\p We mimic the above constructions for pro-Artinian dgas. We will only be interested in the case when the base algebra $B$ is ungraded, which will avoid the need to define bimodules over pro-Artinian dgas in generality (we remark on how to do this in \ref{probimods}). We give an example of such a situation in \ref{hoprolem}. Suppose that $B$ is an Artinian local $k$-algebra and that $\mathcal{A}$ is a pro-Artinian dga with a map to $B$. If $M$ is a dg-$B$-bimodule, then $M$ is naturally a bimodule over $\varprojlim\mathcal{A}$.
\begin{defn}
Let $B$ be an Artinian local $k$-algebra and let $\mathcal{A}\to B$ be a pro-Artinian dga with a map to $B$. Let $M$ be a finite dg-$B$-bimodule concentrated in nonnegative degrees. Note that $B \oplus M$ is still Artinian local. A \textbf{derivation} $\mathcal{A}\to M$ is a map $\mathcal{A}\to B \oplus M$ in the overcategory ${\proart}/B$. The set of all derivations $\mathcal{A}\to M$ is denoted $\mathrm{Der}_B(\mathcal{A},M)$.
\end{defn}
If $\mathcal{A}=\{\mathcal{A}_\alpha \}_\alpha$ with each $\mathcal{A}_\alpha $ Artinian local, then $\mathrm{Der}_B(\mathcal{A},M)\cong \varinjlim_\alpha \mathrm{Der}_B(\mathcal{A}_\alpha, M)$. Hence, $\mathrm{Der}_B(\mathcal{A},M)$ naturally acquires the structure of a dg-vector space. We wish to define derived derivations as derivations from a resolution; before we do this we need to check that the definition makes sense.
\begin{lem}\label{prolemma}
Let $B$ be an Artinian local $k$-algebra and let $M$ be a finite dg-$B$-bimodule concentrated in nonnegative degrees. The functor $\mathrm{Der}_B(-,M):\proart/B \to \cat{dgvect}_k$ preserves weak equivalences between cofibrant objects.
\end{lem}
\begin{proof}
Let $\mathcal{C}$ be the category of finite dg-$B$-bimodules concentrated in nonnegative degrees. The square-zero extension functor extends to a functor $B\oplus-:\cat{pro}\mathcal{C}\to \proart/B$. This functor has a left adjoint $\Omega(-)_B$, given by applying the noncommutative K\"ahler differentials functor levelwise. The functor $B\oplus-$ is clearly right Quillen and hence $\Omega(-)_B$ is left Quillen.
\end{proof}
\begin{rmk}\label{pronccot}
Taking this argument seriously leads one to define the \textbf{pro-noncommutative cotangent complex} $\mathbb{L}(\mathcal{A})_B\in \mathrm{Ho}(\cat{pro}(B-\cat{bimod}))$ of an object $\mathcal{A}\in \proart/B$.
\end{rmk}
\begin{defn}
Let $B$ be an Artinian local $k$-algebra and let $\mathcal{A}\to B$ be a pro-Artinian dga with a map to $B$. Let $Q\mathcal{A} \to \mathcal{A}$ be a cofibrant resolution. Let $M$ be a finite dg-$B$-bimodule concentrated in nonnegative degrees. The space of \textbf{derived derivations} from $\mathcal{A}$ to $M$ is the dg-vector space $$\R\mathrm{Der}_B(\mathcal{A},M):=\mathrm{Der}_B(Q\mathcal{A},M).$$
\end{defn}
By \ref{prolemma}, different choices of resolution for $\mathcal A$ yield quasi-isomorphic spaces of derived derivations. One has an isomorphism $H^0(\R\mathrm{Der}_B(\mathcal{A},M))\cong \hom_{\mathrm{Ho}(\proart / B)}(\mathcal{A}, B \oplus M)$. The main technical result of this section is that the two notions of derived derivation match up for good dgas. We will first prove this for $B\cong k$, where the proof is simpler (because the action of $\mathfrak{m}_\mathcal A$ on $M$ is trivial), and then we will adapt the argument to general $B$ via filtering by the action of $\mathfrak{m}_B$ to reduce to the case $B\cong k$. Observe that any pro-Artinian dga $\mathcal{A}$ admits an augmentation $\mathcal{A}\to k$.
\begin{lem}\label{prerder}
Let $\mathcal{A}$ be a pro-Artinian dga. Let $M$ be a finite-dimensional dg-$k$-vector space concentrated in nonpositive degrees. Assume that $A:=\varprojlim \mathcal{A}$ is cohomologically locally finite. Then there is a quasi-isomorphism $$\R\mathrm{Der}_{k}(\mathcal{A},M) \cong \R\mathrm{Der}_{k}(A,M) .$$
\end{lem}
\begin{proof}
The idea is that the cofibrant resolutions agree. For brevity, we will omit the bar notation for (co)augmentation (co)ideals. Consider first the space $\R\mathrm{Der}_{k}(A, M)$. Since $\Omega B A \to A$ is a cofibrant resolution, we have $\R\mathrm{Der}_{k}(A, M)$ quasi-isomorphic to $ \mathrm{Der}_k(\Omega B A, M)$. Now, $\Omega B A$ is freely generated by $BA[-1]$, so $\mathrm{Der}_k(\Omega B A, M) \cong \dgh_k(BA[-1],M)$ as dg-vector spaces. Since $M$ is finite-dimensional, $\dgh_k(BA[-1],M)$ is the same as $A^![1]\otimes_k M$ (here is where we are using that $B\cong k$; the underlying graded vector spaces are always isomorphic but in general the differential of the right hand side acquires a twist from the action of $\mathfrak{m}_B$ on $M$). Consider now the space $\R\mathrm{Der}_{k}(\mathcal{A},M)$. We use the equivalence of $\proart$ with conilpotent dgcs, along with the fact that $C \to B \Omega C$ is a fibrant dgc resolution, to see that $\R\mathrm{Der}_{k}(\mathcal{A},M)$ is quasi-isomorphic to $ \mathrm{Der}_k((B\Omega(\varinjlim\mathcal{A}^*))^\sharp,M)$. The dgc $B\Omega(\varinjlim\mathcal{A}^*)$ is cofreely cogenerated by $\Omega(\varinjlim \mathcal{A}^*)[1]$, so that the pro-Artinian algebra $(B\Omega(\varinjlim \mathcal{A}^*))^\sharp$ is freely generated by the profinite vector space $\Omega(\varinjlim \mathcal{A}^*)^\sharp[-1]$. Hence we have isomorphisms of dg-vector spaces$${\mathrm{Der}_k((B\Omega(\varinjlim\mathcal{A}^*))^\sharp,M)\cong  \dgh_{\cat{pro}(\cat{fdvect}_k)}(\Omega(\varinjlim \mathcal{A}^*)^\sharp[-1],M)\cong \dgh_k(M^*,\Omega(\varinjlim\mathcal{A}^*)[1])}.$$Again, because $M$ is finite-dimensional and $B\cong k$, this is isomorphic to $\Omega(\varinjlim \mathcal{A}^*)[1]\otimes_k M$. So it suffices to show that $\Omega(\varinjlim \mathcal{A}^*)$ and $A^!$ are quasi-isomorphic as dg-vector spaces. This is similar to the proof of \ref{kdgood}: first note that $\Omega(\varinjlim \mathcal{A}^*) \cong \varinjlim \Omega( \mathcal{A}^*)$ because $\Omega$ is cocontinuous, and that $\Omega( \mathcal{A}^*) \cong \mathcal{A}^!$ because each level of $\mathcal{A}$ is Artinian local. Hence, as in \ref{kdgood}, $\Omega(\varinjlim \mathcal{A}^*)$ is quasi-isomorphic to $A^!$, as required. Note that this last fact uses that $BA$ is cohomologically locally finite, which is the only place we use the hypothesis that $A$ is cohomologically locally finite.
\end{proof}
Now we will extend the argument to cover all $B$.
\begin{thm}\label{rder}
Let $B$ be an Artinian local $k$-algebra and let $\mathcal{A}\to B$ be a pro-Artinian dga with a map to $B$. Let $M$ be a finite-dimensional $B$-bimodule concentrated in nonpositive degrees. Assume that $A:=\varprojlim \mathcal{A}$ is cohomologically locally finite. Then there is a quasi-isomorphism $$\R\mathrm{Der}_{B}(\mathcal{A},M) \cong \R\mathrm{Der}_{B}(A,M) .$$
\end{thm}
\begin{proof}
We follow the proof of \ref{prerder}. This time we have to care about twists. We still have a quasi-isomorphism $\R\mathrm{Der}_{k}(A, M)\cong \dgh_k(BA[-1],M)$. However, the differential on $\dgh_k$ is twisted by the action of $B$ on $M$: explicitly, if $f\in \dgh_k(BA[-1],M)$, then $df$ gains an extra term $\Delta f$, defined by $\Delta f (v)=f(v_{(1)}).v_{(2)}+v_{(1)}.f(v_{(2)})$ where we are using Sweedler notation for the comultiplication $\Delta(v)=v_{(1)}\otimes v_{(2)}$. Across the isomorphism of underlying graded vector spaces $\dgh_k(BA[-1],M)\cong A^![1]\otimes_k M$, the twist $\Delta$ on $\dgh_k(BA[-1],M)$ corresponds to a twist in the differential on $A^![1]\otimes_k M$; let $A^![1]\otimes^\Delta_k M$ denote $A^![1]\otimes_k M$ equipped with this twisted differential. Filtering $M$ by the action of $\mathfrak{m}_B$ gives a finite filtration $F^p= A^![1]\otimes^\Delta_k M.\mathfrak{m}^p_B$ on $A^![1]\otimes^\Delta_k M$. The associated graded pieces are $\mathrm{gr}^p_F:=F^p/F^{p+1}\cong A^![1]\otimes^\Delta_k\mathrm{gr}^p_M$, where we put $\mathrm{gr}^p_M:=M.\mathfrak{m}^p_B/M.\mathfrak{m}^{p+1}_B$. One obtains a convergent spectral sequence $( A^![1]\otimes^\Delta_k\mathrm{gr}^p_M)^q \implies H^{p+q}(A^![1]\otimes^\Delta_k M)$. The twist in the differential of $A^![1]\otimes^\Delta_k M$ disappears upon passing to the associated graded pieces and so one has $\mathrm{gr}^p_F\cong A^![1]\otimes_k\mathrm{gr}^p_M$. In other words, the natural map $A^![1]\otimes_k M\to A^![1]\otimes^\Delta_k M$ is an isomorphism on associated graded pieces. By considering the same spectral sequence for $A^![1]\otimes_k M$, we see that the natural map is actually a quasi-isomorphism. Hence we get a quasi-isomorphism $\R\mathrm{Der}_{k}(A, M)\cong A^![1]\otimes_k M$, as before. The argument to show that $\R\mathrm{Der}_{B}(\mathcal{A},M)\cong\Omega(\varinjlim \mathcal{A}^*)[1]\otimes_k M$ is similar. The proof that $\Omega(\varinjlim \mathcal{A}^*)[1]\otimes_k M\cong A^![1]\otimes_k M$ is the same as before.
\end{proof}
\begin{rmk}
Continuing on from \ref{pronccot}, the above proof gives a quasi-isomorphism between $\varprojlim\mathbb{L}(\mathcal{A})_B$ and $ \mathbb{L}(A)_B$. In other words, the (pro-)noncommutative cotangent complex functor commutes up to quasi-isomorphism with $\varprojlim$, as long as we assume some finiteness conditions.
\end{rmk}
\begin{rmk}\label{probimods}Let $A$ be a pro-Artinian dga. In the spirit of \cite{quillender}, one could define a \textbf{pro-$A$-bimodule} to be a group object in the category of pro-Artinian dgas with a map to $A$. Equivalently, this is an $A$-bimodule in the category of profinite dg vector spaces. In this framework, one can also define pro-derivations, pro-noncommutative K\"ahler differentials, and the pro-noncommutative cotangent complex. If the underlying vector space of a pro-bimodule $M$ is constant, then $M$ is a bimodule over some $A_\alpha$, and hence a bimodule over $A_\beta$ for all $ \beta\to\alpha$. The proofs of \ref{prerder} and \ref{rder} adapt to cover the case when $M$ is a constant pro-$A$-bimodule, and also provide comparisons between the cotangent complexes.
\end{rmk}

\subsection{Koszul duality for homotopy pro-Artinian algebras}
The main result of this part is a characterisation of good dgas (\ref{goodchar}), which can also be thought of as a strictification result. Call a dga $A$ \textbf{homotopy pro-Artinian} if $HA$ is pro-Artinian (in the sense that it is a limit of a pro-Artinian dga). We prove that a certain class of homotopy pro-Artinian dgas (namely those for which the pro-structure is that of the Postnikov tower) are good. We obtain as a corollary a Koszul duality result for this class of dgas. For a very general approach to some of the ideas of this part, see \cite[7.4]{lurieha}.
\begin{defn}Let $A \to B$ be a map of dgas. Say that a map $A' \to A$ of dgas is a \textbf{homotopy square-zero extension} over $B$ if there is a $B$-bimodule $M$ such that $A'$ is the homotopy fibre product of a diagram of the form $A \xrightarrow{\delta} B \oplus M \xleftarrow{0} B$ where $\delta$ is a derived derivation and $0$ is the zero derivation.
\end{defn}
We will be interested in the case when $B=H^0A$, and $A$ is a good dga. In this case, one can lift the natural map $A \to B$ to a map of pro-Artinian dgas. Observe that since $\varprojlim: \proart \to \dga$ is exact (\ref{limisexact}), we may regard it as the homotopy limit $\holim$; we will use this without further acknowledgement.
\begin{lem}\label{hoprolem}
	Let $\mathcal{A}\in \proart$. Assume that $H^0\varprojlim \mathcal{A}$ is finite-dimensional. Then there is a map of pro-Artinian dgas $\mathcal{A}\to H^0\varprojlim \mathcal{A}$.
\end{lem}
\begin{proof}
	Put $\mathcal{A}=\{\mathcal{A}_\alpha\}_\alpha$ with each $\mathcal{A}_\alpha$ Artinian. For all $\alpha$, there is a structure map $\mathcal{A}\to\mathcal{A}_\alpha$ and hence a map $\mathcal{A}\to H^0(\mathcal{A}_\alpha)$. These maps assemble into an element of $\varprojlim_\alpha\hom_{\proart}(\mathcal{A},H^0\mathcal{A}_\alpha)$. Because $\varprojlim$ is the homotopy limit we have $\varprojlim_\alpha H^0(\mathcal{A}_\alpha)\cong H^0\varprojlim \mathcal{A}$. Because this is Artinian by hypothesis, we get an isomorphism 
	\begin{align*}
\varprojlim_\alpha\hom_{\proart}(\mathcal{A},H^0\mathcal{A}_\alpha)&\cong \hom_{\proart}(\mathcal{A},\varprojlim_\alpha H^0\mathcal{A}_\alpha)\\&\cong \hom_{\proart}(\mathcal{A},H^0\varprojlim\mathcal{A}).\qedhere
	\end{align*}
\end{proof}
Our main examples of homotopy square-zero extensions will be provided by truncations. 
\begin{defn}
If $A$ is a dga, set $A_n:=\tau_{\geq -n}(A)$, the good truncation to degrees above $-n$. Explicitly, we have $$(A_n)^j=\begin{cases} A^j & j>-n \\ \coker(d:A^{-n-1}\to A^n) & j=-n \\ 0 & j<-n \end{cases}$$One has $H^jA_n\cong H^jA$ if $j\geq -n$ and $H^jA_n=0$ if $j<-n$.
\end{defn}
\begin{lem}\label{hsqlem}
Let $A \in \dga$ be a nonpositive dga. Then for every $n\geq 0$, the natural map $A_{n+1} \to A_n$ is a homotopy square-zero extension with base $H^0A$.
\end{lem}
\begin{proof}
There is a homotopy fibre sequence of $A$-bimodules $$H^{-n-1}(A)[-n-1] \to A_{n+1}\to A_n$$indicating that we should take $M$ to be a shift of $H^{-n-1}(A)$. Indeed, this sequence gives a map $A_n \to H^{-n-1}(A)[-n-2]$ in the homotopy category of $A$-bimodules, and so we put $M:=H^{-n-1}(A)[-n-2]$. Let $C$ be the mapping cone of $H^{-n-1}(A)[-n-1] \to A_{n+1}$ and write $B:=H^0A$. As in the proof of \cite[1.45]{unifying}, $C$ admits the structure of a dga, quasi-isomorphic to $A_n$, and moreover $A_{n+1}$ is the strict pullback of the diagram $B \to B \oplus M \from C$. The map $C \to B\oplus M$ is a fibration, and hence $A_{n+1}$ is the homotopy pullback.
\end{proof}
\begin{lem}\label{hplem}
Let $A \in \dga$ be a nonpositive dga. Suppose that for some $n$, $A_n$ is good and $H^{-n-1}(A)$ is finite-dimensional. Then $A_{n+1}$ is good.
\end{lem}
\begin{proof}
It is clear that $A_{n+1}$ is cohomologically locally finite. So we just need to prove that $A_{n+1}$ is quasi-isomorphic to something in the image of $\varprojlim$. By \ref{hsqlem}, $A_{n+1}$ is the homotopy pullback of a diagram $J$ of the form $H^0(A)\to H^0(A)\oplus M \from A_n$, where $M$ is a finite module. By \ref{rder} and \ref{hoprolem}, we may view $J$ as a diagram in $\proart$; let $P\in\proart$ be the homotopy pullback. Since $\varprojlim:\proart \to\dga$ is the homotopy limit and homotopy limits commute, we see that $\varprojlim P$ is the homotopy pullback of $J$ considered as a diagram in $\dga$. But by \ref{hsqlem} this is precisely $A_{n+1}$.
\end{proof}
\begin{rmk}
	If $A_n$ is Artinian local, then $A_{n+1}$ is homotopy Artinian local, but neither $A_{n+1}$ nor $P$ need be Artinian local.
\end{rmk}
\begin{prop}\label{goodchar}
	Let $A \in \dga$ be a nonpositive dga. The following are equivalent:
	\begin{itemize}
		\item $A$ is good.
		\item $A$ is cohomologically locally finite and $H^0(A)$ is local.
	\end{itemize}
\end{prop}
\begin{proof}The forward direction is clear. For the backwards direction, assume that $A$ is cohomologically locally finite and that $H^0(A)$ is local. We may also assume that $A$ is cofibrant. Using \ref{hplem} inductively, it is easy to see that each truncation $A_n$ is good. In fact, one can say more: we obtain for each $n$ a pro-Artinian dga $P_n$ together with an isomorphism $A_n \to \varprojlim P_n$ in $\mathrm{Ho}(\dga)$. Using that $A$ is cofibrant, we obtain a dga map $A \to \varprojlim P_n$ making the obvious triangles commute. We have quasi-isomorphisms $A \cong \holim_n \varprojlim P_n \cong \varprojlim \holim_n P_n$, because $\varprojlim$ is the homotopy limit. Hence $A$ is quasi-isomorphic to something in the image of $\varprojlim$.
\end{proof}
Theorem \ref{kdgood} immediately gives the following:
\begin{thm}\label{kdfin}
	Let $A \in \dga$ be a cohomologically locally finite dga such that $H^0(A)$ is local. Then $A$ is quasi-isomorphic to its double Koszul dual.
\end{thm}
\subsection{The pseudo-model category of good dgas}
Using \ref{goodchar} to write a good dga as an iterated sequence of homotopy square-zero extensions, we will extend the results of \ref{rder} to give an equivalence of pseudo-model categories between good dgas and those pro-Artinian dgas whose limits are good (the `pregood' ones). Because neither of these (pseudo)-model categories are dg model categories, we need to first translate \ref{rder} into the language of simplicial mapping spaces.
\begin{defn}
	Let $\mathcal{C}$ be a model category, and let $X,Y$ be two objects of $\mathcal{C}$. Following \mbox{{\cite[5.4.9]{hovey}}} write $\R\mathrm{Map}_{\mathcal{C}}({X},{Y}) \in \mathrm{Ho}(\cat{sSet})$ for the derived mapping space from $X$ to $Y$. 
\end{defn}
\begin{prop}\label{dugger}
Let $\mathcal{C}$ be a combinatorial dg model category. Let $X,Y$ be two objects of $\mathcal{C}$. Denote their derived hom-complex by $\R\dgh_{\mathcal{C}}(X,Y)\in D(k)$. Then the quasi-isomorphism type of $\R\dgh_{\mathcal{C}}(X,Y)$ determines the weak homotopy type of the derived mapping space $\R\mathrm{Map}_{\mathcal{C}}({X},{Y})$.
\end{prop}
\begin{proof}By results of Dugger \cite{duggerspectra} a combinatorial dg model category is naturally enriched over the category of symmetric spectra. The basic idea is to identify unbounded dg vector spaces as spectrum objects in the category of nonpositive dg vector spaces (via shifting and good truncation), and then apply the Dold-Kan correspondence \cite[8.4]{weibel} levelwise to end up with a spectrum object in simplicial sets. See also \cite{duggershipleyspectra} for an additive version where one ends up with spectrum objects in simplicial abelian groups. In particular, taking the zeroth level of the derived mapping spectra of $\mathcal{C}$ gives an enrichment of $\mathcal{C}$ over simplicial sets. For fibrant-cofibrant objects this enrichment must be weakly equivalent to the usual one. 
\end{proof}
\begin{cor}\label{dermapspace}
Let $B$ be an Artinian local $k$-algebra and let $\mathcal{A}\to B$ be a pro-Artinian dga with a map to $B$. Let $M$ be a finite-dimensional $B$-bimodule concentrated in nonpositive degrees. Assume that $A:=\varprojlim \mathcal{A}$ is cohomologically locally finite. Then there is a weak equivalence $$\R\mathrm{Map}_{\proart/B}(\mathcal{A},B\oplus M) \cong \R\mathrm{Map}_{\cat{dga}_k/B}(A,B\oplus M) .$$
\end{cor}
\begin{proof}
Follows from \ref{rder} and \ref{dugger}.
\end{proof}
\begin{defn}
Let $\cat{g}\dga \into \dga$ denote the full subcategory on good dgas. Call a pro-Artinian dga $A$ \textbf{pregood} if $\varprojlim A$ is good, and let $\cat{g}\proart \into \proart$ denote the full subcategory on pregood pro-Artinian dgas. 
\end{defn}
Note that a pro-Artinian dga is pregood if and only if for each $n$ the vector space $\varprojlim H^nA$ is finite. The following definition is a slight variant of \cite[4.1.1]{hag1}.
\begin{defn}
	A \textbf{pseudo-model category} $C$ is a full subcategory of a model category $M$ such that $C$ is closed under weak equivalences and homotopy pullbacks in $M$. 
\end{defn}
We will soon show in \ref{pseudomods} that $\varprojlim:\cat{g}\proart \to \cat{g}\dga$ is a Quillen equivalence of pseudo-model categories. The key step in proving this will be to check that the derived mapping spaces agree, which generalises \ref{dermapspace}.
\begin{prop}\label{dermap}
	Let $\mathcal{A}$, $\mathcal{A}'$ be pregood pro-Artinian dgas. Put $A:=\varprojlim \mathcal{A}$ and  $A':=\varprojlim \mathcal{A}'$. There is a weak equivalence of derived mapping spaces $$\R\mathrm{Map}_{\proart}(\mathcal{A},\mathcal{A}')\cong\R\mathrm{Map}_{\cat{dga}_k}({A},{A'}).$$
\end{prop}
\begin{proof}
First note that by \ref{dermapspace}, this is the case when $A'$ is a square-zero extension of $k$ by a finite module. Moreover, \ref{hplem} writes $A'$ as an iterated sequence of finite homotopy square-zero extensions, starting with the ungraded Artinian local algebra $H^0(A')$. Since $\R\mathrm{Map}(A,-)$ preserves homotopy limits, it hence suffices to prove the claim in the case when $A'=H^0(A')$. By the same logic, it is enough to prove that any ungraded Artinian local algebra $A'$ is an iterated sequence of finite homotopy square-zero extensions of $k$. The tower $A'/\mathfrak{m}_{A'}^n$ exhibits $A'$ as an iterated sequence of finite (classical) square-zero extensions starting from $k$, so it is enough to prove that if $\pi$ is a square-zero extension of ungraded Artinian local algebras, then $\pi$ is also a homotopy square-zero extension. But this follows from the fact that $\pi$ is surjective, so we have a quasi-isomorphism $\ker (\pi) \cong \mathrm{cocone}(\pi)$. 
\end{proof}
\begin{thm}\label{pseudomods}
Both $\cat{g}\dga \into \dga$ and $\cat{g}\proart \into \proart$ are pseudo-model categories, and $\varprojlim: \cat{g}\proart \to \cat{g}\dga$ is a Quillen equivalence.
\end{thm}
\begin{proof}It is clear that $\cat{g}\dga$ is closed under weak equivalences. Since $\varprojlim$ reflects weak equivalences by \ref{limreflects}, it follows that $\cat{g}\proart$ is also closed under weak equivalences. The closure of both under homotopy pullbacks follows exactly as in \ref{hplem}, using the equivalence of derived mapping spaces from \ref{dermap}. So both are pseudo-model categories. It is easy to see that $\varprojlim$ is right Quillen. By definition, it is also homotopy essentially surjective. It is homotopy fully faithful by \ref{dermap}.
\end{proof}

\section{Deformation theory}\label{defmthy}
In some sense, commutative formal deformation theory in characteristic zero is about the Koszul duality of the commutative and Lie operads. Indeed, given a commutative deformation problem, one expects it to be `controlled' in some way by a differential graded Lie algebra (\textbf{dgla}). This philosophy is originally due to Deligne, and first appears in print in a paper of Goldman and Millson \cite{goldmanmillson}. Hinich \cite{hinich} gave an interpretation of in terms of coalgebras, and the correspondence between commutative deformation problems and dglas was made precise by later work of Pridham \cite{unifying} and Lurie \cite{luriedagx}. Correspondingly, in view of the Koszul self-duality of the associative algebra operad, one should expect noncommutative deformation problems to be controlled by noncommutative algebras, and indeed this is true \cite[\S3]{luriedagx}.

\p In this section, we will make some of the above statements explicit. We will work primarily with set-valued and groupoid-valued deformation functors. We will define the Maurer-Cartan functor and deformation functor associated to a dga, and give some prorepresentability results (\ref{mcprorep}, \ref{defprorep}). As an application, we will consider deformations of modules, and prove a prorepresentability result (\ref{prorepfrm}) for framed deformations of simple modules. Framings correspond to nonunital dgas, and correspondingly we make use of nonunital dgas throughout.

\p Artinian local dgas are allowed to be concentrated in any degree throughout this section, although we will restrict to nonpositive dgas when necessary. The category of all Artinian local dgas is denoted $\cat{dgArt}_k$. If $\Gamma$ is an Artinian local dga, denote its maximal ideal by $\mathfrak{m}_\Gamma$. We will mention dglas for motivational purposes only; these are dg vector spaces together with a graded Lie bracket satisfying the graded Leibniz identity with respect to the differential. For more about commutative deformation theory via dglas, see \cite{manetti} or \cite{manettidgla}.
\subsection{The Maurer-Cartan functor and twisting morphisms}
We define the Maurer-Cartan functor associated to a nonunital dga. Note that nonunital dgas are equivalent to augmented dgas after appending a unit, and we freely make use of this equivalence.
\begin{defn}
	Let $E$ be a nonunital dga. The set of \textbf{Maurer-Cartan elements} of $E$ is the set $MC(E):=\{x \in E^1: \ dx+x^2=0\}$. Note that this agrees with the usual notion for dglas once we equip $E$ with the commutator bracket.
\end{defn}
\begin{rmk}
	Note that $x \in MC(E)$ if and only if the map $e\mapsto d(e)+xe$ is a differential on $E$.
\end{rmk}
\begin{defn}
	Let $E$ be a nonunital dga and let $C$ be a noncounital dgc. Then the complex $\hom_k(C,E)$ is a nonunital dga under the product given by $fg:=\mu_E \circ (f\otimes g)\circ \Delta_C$. This dga is the \textbf{convolution algebra}. A Maurer-Cartan element of the convolution algebra is known as a \textbf{twisting morphism}; the set of all twisting morphisms is denoted $\mathrm{Tw}(C,E)$.
\end{defn}
\begin{rmk}
	When $E$ and $C$ are augmented and coaugmented respectively, one should add the additional condition that twisting morphisms are zero when composed with the augmentation or the coaugmentation. This is the definition given in \cite{lodayvallette}.
\end{rmk}
\begin{lem}
	Let $E, Z$ be nonunital dgas, with $Z$ finite-dimensional. Then there is a natural isomorphism $$\mathrm{Tw}(Z^*,E)\cong {MC}(E \otimes Z).$$
\end{lem}
\begin{proof}
	There is a standard linear isomorphism $E \otimes Z \to \hom_k(Z^*,E)$, and one can check that this is a map of nonunital dgas after equipping $\hom_k(Z^*,E)$ with the convolution product.
\end{proof}
\begin{defn}
Let $E$ be a nonunital dga. The \textbf{nonunital bar construction} of $A$ is the counital dgc $B_\mathrm{nu}E:=B(E\oplus k)$, where $B$ is the usual bar construction for augmented dgas. In other words, the underlying graded coalgebra of $B_\mathrm{nu}E$ is the tensor coalgebra $T^c(E)$, and the differential is the usual bar differential. Similarly, there is a nonunital cobar construction $\Omega_\mathrm{nu}$ for nonunital dgcs.
\end{defn}
\begin{rmk}
We caution that every degree zero element of $E$ corresponds to a degree -1 element of $B_\mathrm{nu}E$. In particular, if the nonunital dga $E$ happens to have a unit and an augmentation, then $B_\mathrm{nu}E\neq BE$ since $B_\mathrm{nu}E$ will contain elements corresponding to the unit.
\end{rmk}
The functor of twisting morphisms is (up to (co)units) representable on either side:
\begin{thm}[\cite{lodayvallette}, 2.2.6]
	If $E$ is a nonunital dga and $C$ is a noncounital conilpotent dgc, then there are natural isomorphisms $$\hom_{\cat{dga}}(\Omega_\mathrm{nu} C, E\oplus k)\cong\mathrm{Tw}(C,E)\cong \hom_{\cat{dgc}}(C\oplus k, B_\mathrm{nu}E)$$
\end{thm}
\begin{rmk}
	Because \cite{lodayvallette} uses (co)augmented rather than non(co)unital (co)algebras, we need to reinsert the (co)units.
\end{rmk}
We recall from \ref{csharp} that if $C$ is a (counital) dgc then $C^\sharp$ denotes the pro-Artinian dga constructed by levelwise dualising the filtered system of finite sub-dgcs of $C$. If $E$ is a nonunital dga, write $B_\mathrm{nu} ^\sharp E:=(B_\mathrm{nu}E)^\sharp$ for the \textbf{continuous Koszul dual}. Clearly we have $\varprojlim B_\mathrm{nu} ^\sharp E \cong (B_\mathrm{nu}E)^*$ the usual Koszul dual. Exactly as in \ref{sharpprop}, if $Z$ is a finite-dimensional nonunital dga then we have isomorphisms $$\mathrm{Tw}(Z^*,E)\cong \hom_{\cat{dgc}}(Z^*\oplus k, B_\mathrm{nu}E)\cong \hom_{\cat{pro}(\cat{dgArt}_k)}(B_\mathrm{nu}^\sharp E, Z\oplus k)$$
\begin{defn}
	Let $E$ be a nonunital dga. The \textbf{Maurer-Cartan functor} is the functor $\mathrm{MC}(E):\cat{dgArt}_k \to \cat{Set}$ which sends $\Gamma$ to the set $\mathrm{MC}(E)(\Gamma):={MC}(E\otimes \mathfrak{m}_\Gamma)$.
\end{defn}
The following Proposition is immediate:
\begin{prop}\label{mcprorep}
	Let $E$ be a nonunital dga. Then the functor $\mathrm{MC}(E)$ is prorepresented by $B_\mathrm{nu}^\sharp E$, in the sense that $\mathrm{MC}(E)$ and $\hom_{\cat{pro}(\cat{dgArt}_k)}(B_\mathrm{nu}^\sharp E, -)$ are naturally isomorphic.
\end{prop}

\subsection{The gauge action}\label{gauges}
In commutative deformation theory, given a nilpotent dgla $L$ there is a gauge action of $\mathrm{exp}(L^0)$ on $\mathrm{MC}(L)$. The corresponding quotient is the deformation functor associated to $L$. One can define a similar gauge action in the noncommutative world; we follow Efimov-Lunts-Orlov \cite{ELO}.
\begin{defn}
Let $E$ be a nonunital dga. The \textbf{gauge group} functor $\mathrm{Gg}(E):\cat{dgArt}_k \to \cat{Grp}$ sends $\Gamma$ to the set $1+(E\otimes \mathfrak{m}_\Gamma)^0$, which is a group under multiplication.
\end{defn}
\begin{rmk}
If $L$ is a dgla, its gauge group has as elements formal symbols $\exp(l)$ for $l \in L^0$, and multiplication given by the Baker-Campbell-Hausdorff formula. If $E\otimes \mathfrak{m}_\Gamma$ is made into a dgla using the commutator bracket then its dgla gauge group is isomorphic to the gauge group defined above, via the map that sends each formal exponential $\exp(a)$ to the sum $\sum_n\frac{a^n}{n^!}$, which exists because $\mathfrak{m}_\Gamma$ is nilpotent.

\end{rmk}
Let $E$ be a nonunital dga and $\Gamma$ an Artinian local dga. Let $x \in \mathrm{MC}(E)(\Gamma)$ and $g \in \mathrm{Gg}(E)(\Gamma)$. Define $g.x=gxg^{-1}+gd(g^{-1})$. One can check that this gives an action of $\mathrm{Gg}(E)(\Gamma)$ on $\mathrm{MC}(E)(\Gamma)$. Regarding $d+x$ as a twisted differential on $E\otimes \mathfrak{m}_\Gamma$, the action of the gauge group is the conjugation action on the space of differentials.
\begin{rmk}
If $E\otimes \mathfrak{m}_\Gamma$ is made into a dgla using the commutator bracket, then the gauge action above is the exponential of the usual dgla gauge action (as in e.g. \cite[V.4]{manetti}).
\end{rmk}
\begin{defn}
Let $E$ be a nonunital dga. The \textbf{deformation functor} associated to $E$ is the quotient $\mathrm{Def}(E):=\mathrm{MC}(E)/\mathrm{Gg}(E)$.
\end{defn}
One can extend the set-valued functor $\mathrm{Def}(E)$ to a groupoid-valued functor. Fix $E,\Gamma$ and take $x,y \in \mathrm{MC}(E)(\Gamma)$. Observe that the group $(E\otimes \mathfrak{m}_\Gamma)^{-1}$ acts on the set $\{g \in \mathrm{Gg}(E)(\Gamma): gx=y\}$ by setting $g.h=g+d(h)+yh+hx$. Say that $g_1,g_2$ are \textbf{homotopic} if they lie in the same orbit under this action. One checks that homotopy is preserved under composition.
\begin{rmk}
A different but equivalent definition of homotopy for dglas is given as part of the proof of Theorem 3.1 of \cite{manettidgla}; the idea is that the stabiliser of $x$ in $\mathrm{Gg}(E)(\Gamma)$ acts on the set $\{g \in \mathrm{Gg}(E)(\Gamma): gx=y\}$ in the obvious way.
\end{rmk}
\begin{defn}\label{delignegrpd}
Let $E$ be a nonunital dga. The \textbf{groupoid-valued deformation functor} $\grdef(E):\cat{dgArt}_k \to \cat{Grpd}$ sends $\Gamma$ to the groupoid with objects $\mathrm{MC}(E)(\Gamma)$, and morphisms $x \to y$ the homotopy classes of gauges $g$ with $gx=y$. Clearly we have $\pi_0\grdef(E)\cong \mathrm{Def}(E)$.
\end{defn}
\begin{rmk}\label{delignermk}
One could also consider the groupoid quotient $\grdef'(E):=\mathrm{MC}(E)//\mathrm{Gg}(E)$, which is perhaps a more natural choice of definition. Of course we also have $\pi_0 \grdef'(E)\cong \mathrm{Def}(E)$. The reason for introducing homotopy is that $\mathrm{Def}$ can be further enhanced to a simplicial set valued functor $\R\mathrm{Def}$ \cite[\S 4]{jondmodss}, and $\grdef$ is its fundamental groupoid. Morally, $\grdef'$ can be thought of as the brutal 1-truncation of $\R\mathrm{Def}$, and the `extra' homotopies between gauges come from nontrivial 2-simplices (the actual 1-truncation is dependent on a choice of model, since $\R\mathrm{Def}$ is really only defined up to weak equivalence). We also remark that \ref{mcdefms} is false if one uses $\grdef'$ instead of $\grdef$. In the literature, both $\grdef(E)$ and $\grdef'(E)$ are referred to as the \textbf{Deligne groupoid} functor. If $\Gamma$ is an ungraded Artinian local algebra, and $E$ is a nonnegative dga, then $\grdef(E)(\Gamma)\cong\grdef'(E)(\Gamma)$, so the difference between the two definitions only becomes apparent when deforming along dgas.
\end{rmk}
\begin{prop}
	If $E$ and $E'$ are quasi-isomorphic nonunital dgas then the functors $\grdef(E)$ and $\grdef(E')$ are equivalent.
\end{prop}
\begin{proof}
The proof of \cite[8.1]{ELO} adapts to the nonunital setting.
\end{proof}

\subsection{Deformation functors and prorepresentability}
We will define the noncommutative analogue of Manetti's (set-valued) extended deformation functors \cite{manettiextended}. Our main result is the prorepresentability statement \ref{defprorep}, which is a homotopical version of \ref{mcprorep}.
\begin{defn}
	A morphism $A \to B$ in $\cat{dgArt}_k$ is a \textbf{small extension} if it is surjective and the kernel is annihilated by $\mathfrak{m}_A$.
\end{defn}
\begin{defn}
	Say that a functor $F:\cat{dgArt}_k \to \cat{Set}$ is a \textbf{predeformation functor} if the following conditions are satisfied:\begin{enumerate}
		\item $F(k)$ is a singleton.
		\item Let $A \to C \from B$ be a cospan in $\cat{dgArt}_k$, and consider the induced map $$\eta:F(A\times_C B)\to F(A)\times_{F(C)} F(B)$$If $A \to C$ is a small extension, then $\eta$ is surjective. If $C\cong k$ then $\eta$ is an isomorphism.
		\item Let $A \to B$ be a small extension with acyclic kernel. Then the induced map $\rho:F(A)\to F(B)$ is surjective.
	\end{enumerate}
	Say that $F$ is a \textbf{deformation functor} if in addition the map $\rho$ of iii) is an isomorphism. In \cite{manettiextended}, the conditions ii) are referred to as the \textbf{generalised Schlessinger's conditions} and iii) is referred to as \textbf{quasismoothness}.
\end{defn}
\begin{thm}
There is a cofibrantly generated model structure on $\cat{pro}(\cat{dgArt}_k)$, such that deformation functors are precisely the homotopy prorepresentable functors. More precisely, a functor $F:\cat{dgArt}_k \to \cat{Set}$ is a deformation functor if and only if there is a pro-Artinian dga $P$ such that $F(\Gamma)\cong [P,\Gamma]$, the set of maps from $P$ to $\Gamma$ in the homotopy category $\mathrm{Ho}(\cat{pro}(\cat{dgArt}_k))$.
\end{thm}
\begin{proof}
The arguments of \cite[\S4.5]{unifying} carry over. More specifically, we are using 4.36 and 4.44.
\end{proof}
\begin{rmk}
The inclusion-truncation adjunction $\proart \longleftrightarrow \cat{pro}(\cat{dgArt}_k)$ becomes a Quillen adjunction when $\proart$ is given the model structure of \ref{proartmodel}.
\end{rmk}
\begin{prop}
	Every predeformation functor $F$ admits a map $F \to F^+$ to a deformation functor, universal amongst deformation functors under $F$.
\end{prop}
\begin{proof}
	The proof of \cite[2.8]{manettiextended} goes through, once one replaces the commutative construction $A[t,dt]_\epsilon$ with a suitable (noncommutative) path object for $A$. The idea is to mod out $F$ by the homotopies which a deformation functor must respect.
\end{proof}
\begin{thm}
	Let $E$ be a nonunital dga. Then $\mathrm{MC}(E)$ is a predeformation functor, $\mathrm{Def}(E)$ is a deformation functor, and $\mathrm{Def}(E)\cong \mathrm{MC}(E)^+$.
\end{thm}
\begin{proof}
	We follow \cite{manettiextended}, checking that the statements about dglas adapt to the noncommutative (and nonunital) setting. The first statement is 2.17, the second is 2.19, and the third is 3.4. The first two statements are checked by hand, whilst the third follows from obstruction theory.
\end{proof}
We can now prove our desired prorepresentability result, which is the noncommutative version of \cite[4.45]{unifying}. We follow the proof given there.
\begin{prop}\label{defprorep}
Let $E$ be a nonunital dga. Then $\mathrm{Def}(E)\cong [B_\mathrm{nu}^\sharp E,-]$.
\end{prop}
\begin{proof}
It suffices to prove that $[B_\mathrm{nu}^\sharp E,-]\cong \mathrm{MC}(E)^+$. By \ref{mcprorep}, we simply need to prove that $[B_\mathrm{nu}^\sharp E,-]\cong \hom(B_\mathrm{nu}^\sharp E, -)^+$. But this is easy: take a map $\hom(B_\mathrm{nu}^\sharp E, -) \to G$ to a deformation functor $G$. Since $G$ is a deformation functor, it is homotopy prorepresentable, so put $G\cong[P,-]$. The map $\hom(B_\mathrm{nu}^\sharp E, -) \to [P,-]$ gives us an element of $[P,B_\mathrm{nu}^\sharp E]$, which gives a map $[B_\mathrm{nu}^\sharp E, -] \to G$. Hence $[B_\mathrm{nu}^\sharp E,-]$ is universal amongst deformation functors under $\hom(B_\mathrm{nu}^\sharp E, -)$.
\end{proof}

\subsection{Deforming modules}
We are interested in the noncommutative derived deformation theory of modules over a dga. Recall that an underived deformation of an $A$-module $X$ over an Artinian local ring $\Gamma$ is an $A\otimes \Gamma$-module $\tilde X$ that reduces to $X$ modulo $\mathfrak{m}_\Gamma$. A derived deformation is defined similarly:
\begin{defn}
Let $A$ be a dga and $X$ an $A$-module. Let $\Gamma$ be an Artinian local dga. A \textbf{derived deformation} of $X$ over $\Gamma$ is a pair $(\tilde X, f)$ where $\tilde X$ is an $A\otimes \Gamma$-module and $f:\tilde X \lot_\Gamma k \to X$ is an isomorphism in $D(A)$. An \textbf{isomorphism} of derived deformations is an isomorphism $\phi:\tilde X_1 \to \tilde X_2$ in $D(A\otimes\Gamma)$ such that $f_1 = f_2 \circ ({\phi \lot_\Gamma k})$.
\end{defn}
Deformations are functorial with respect to algebra maps: given a map $\Gamma \to\Gamma'$ of Artinian local dgas, and a derived deformation $\tilde X$ of $X$ over $\Gamma$, then the derived tensor product $\tilde X \lot_\Gamma \Gamma'$ is a derived deformation of $X$ over $\Gamma'$.
\begin{defn}
Let $A$ be a dga and $X$ an $A$-module. The functor $\grdef_A(X):\cat{dgArt}_k \to \cat{Grpd}$ sends an Artinian local dga $\Gamma$ to the groupoid quotient$$\grdef_A(X)(\Gamma):=\{\text{derived deformations of }X \text{ over }\Gamma\}//(\text{isomorphism})$$Put $\mathrm{Def}_A(X):=\pi_0\grdef_A(X)$. We will just write $\grdef(X)$ if there is no ambiguity.
\end{defn}
If $F:\cat{dgArt}_k \to \mathcal{C}$ is a functor, then we denote its restriction to $\dgart$ by $F^{\leq 0}$.
\begin{thm}\label{mcdefms}
Let $A$ be an ungraded algebra and $X$ an $A$-module. There is an equivalence
$$\grdef_A^{\leq 0}(X) \cong \grdef^{\leq 0}(\R\enn_A(X))$$of groupoid-valued deformation functors, where the right-hand functor is the Deligne groupoid of \ref{delignegrpd}.
\end{thm}
\begin{proof}
This follows from 6.1 and 11.6 of \cite{ELO}. The idea is that a derived deformation of $X$ lifts to a `homotopy deformation' of a projective resolution $P \to X$, which is a deformation of the differential on $P$, which is precisely a Maurer-Cartan element of $\enn_A(P)\cong \R\enn_A(X)$. Isomorphisms between (homotopy) deformations become (homotopy classes of) gauges.
\end{proof}
\begin{rmk}
We remark that the statement of the theorem makes sense because the Deligne groupoid functor $\grdef$ is invariant under dga quasi-isomorphism.
\end{rmk}

\subsection{Framed deformations}\label{frdefs}
Assume throughout this section that $A$ is an ungraded algebra and that $S$ is a one-dimensional $A$-module. It immediately follows that $S$ is simple, so that the derived endomorphism algebra $E:=\R\enn_A(S)$ is augmented. Theorem \ref{mcdefms} tells us that $\grdef_A^{\leq 0}(S) \cong \grdef^{\leq 0}(E)$. Recalling that $\bar E$ denotes the augmentation ideal of $E$, we would like the functor $\grdef^{\leq 0}(\bar E)$, or at least its set of connected components $\mathrm{Def}^{\leq 0}(\bar E)$, to have some similar interpretation. To do this, we need to rigidify and consider framed deformations of $S$. If $X$ is a dg-$A$-module, we will use $X_k$ to mean $X$ considered just as a dg-vector space. We will also use $X_A$ when emphasising that $X$ should be viewed as an $A$-module. 
\begin{defn}\label{frmdefs}
Suppose that $X\in D(A \otimes\Gamma)$ is a derived deformation of $S_A$ over $\Gamma$. A \textbf{framing} of $X$ is an isomorphism $\nu_X: X_k \to S\otimes\Gamma$ in $D(k\otimes \Gamma)$ from $X_k$ to the trivial deformation of $S_k$. A \textbf{framed deformation} of $S$ is a pair $(X,\nu_X)$ consisting of a deformation of $S$ together with a framing. A \textbf{framed isomorphism} $F:(X,\nu_X) \to (Y,\nu_Y)$ is an isomorphism $F:X \to Y$ of deformations satisfying $\nu_X=\nu_Y\circ F_k$. The functor of \textbf{framed deformations} of $S$ is the functor $\mathrm{Def}^{\mathrm{fr}}_A(S):\cat{dgArt}_k \to \cat{Set}$ defined by $$\mathrm{Def}^{\mathrm{fr}}_A(S)(\Gamma):=\frac{\{\text{framed deformations of }S \text{ over }\Gamma\}}{(\text{framed isomorphism})}.$$
\end{defn}
\begin{rmk}\label{frdefsset}
Note that we have not defined a groupoid-valued functor of framed deformations. The reason for this is that the na\"ive groupoid quotient of framed deformations by framed isomorphisms does not have the correct deformation-theoretic meaning. If $R$ is a dga then we may consider $D(R)$ as a simplicial category, and similarly we may consider the one-object category $\{S\}\subset D(A)$ as a simplicial category. If one puts $\R\mathrm{Def}^\text{fr}_A(S)(\Gamma):=D(\Gamma\otimes A)\times^h_{D(A)}\{S\}$, where we are taking the homotopy fibre product of simplicial categories, then $\R\mathrm{Def}^\text{fr}_A(S)$ is a functor valued in simplicial sets. It is in fact a deformation functor, which follows from \cite[\S 2]{jondmodss} applied along the lines of \cite[\S 4]{jondmodss}. So the correct groupoid-valued functor of framed deformations is $\grdef^{\text{fr}}_A(S):=\Pi_1\R\mathrm{Def}^\text{fr}_A(S)$ where we take the fundamental groupoid. As we will soon show in \ref{naivegrpdlem}, the na\"ive definition is $\Pi_1(D(\Gamma\otimes A))\times^h_{\Pi_1(D(A))}\Pi_1(\{S\})$. The issue is that homotopy fibre products do not commute with $\Pi_1$, and indeed one does not have $\grdef^{\text{fr}}_A(S) \cong \Pi_1(D(\Gamma\otimes A))\times^h_{\Pi_1(D(A))}\Pi_1(\{S\})$ since the left hand side has extra homotopies coming from $\Pi_2\R\mathrm{Def}^\text{fr}_k(S)$. However, we do have an isomorphism $\mathrm{Def}^{\mathrm{fr}}_A(S) \cong \pi_0\R\mathrm{Def}^\text{fr}_A(S)$, since the fundamental groupoid commutes with homotopy fibre products on the level of connected components.
\end{rmk}We aim to prove that $\mathrm{Def}^{\mathrm{fr},\leq 0}_A(S)$ is isomorphic to the functor $\mathrm{Def}^{\leq 0}(\bar E)$, which we will prove as \ref{frdefmc}. We will do this by showing that an appropriate map of groupoids is an isomorphism on $\pi_0$. We will first need a few facts about the homotopy theory of groupoids.
\begin{prop}[\cite{stricklandgpds}, \S6]The category of groupoids admits a right proper model structure where the weak equivalences are the equivalences of categories, the fibrations are the isofibrations, and the cofibrations are the functors injective on objects.
	\end{prop}
\begin{lem}[\cite{stricklandgpds}, \S6]\label{grpdhfiblem}
Suppose that $\bullet \to B$ is a pointed groupoid, and $F: A \to B$ is a functor between groupoids. The homotopy fibre of $F$ is the groupoid with objects the pairs $(a,u)$ with $a \in A$ and $u:Fa \to \bullet \ $, and morphisms $(a,u) \to (a',u')$ those maps $v:a \to a'$ such that $u=u'\circ Fv$.
\end{lem}
For brevity we will often drop the Artinian local dgas from the notation and reason with groupoid-valued functors as if they were groupoids. The following lemma will reduce our study of framed deformations to the study of the Deligne groupoid of $E$, where we will be able to argue explicitly with gauges and homotopies.
\begin{lem}\label{naivegrpdlem}
The set-valued functor $\mathrm{Def}^{\mathrm{fr},\leq 0}_A(S)$ is $\pi_0$ of the homotopy fibre of the natural map of groupoid-valued functors $\grdef^{\leq 0}(E) \to \grdef^{\leq 0}(k)$ induced by the augmentation $E \to k$.
\end{lem}
\begin{proof}Fix a dga $\Gamma\in \dgart$. Let $\mathcal{G}$ be the groupoid whose objects are the framed deformations of $S$ over $\Gamma$ and whose morphisms are the framed isomorphisms. Clearly  $\mathrm{Def}^{\mathrm{fr},\leq 0}_A(S)(\Gamma)\cong \pi_0 \mathcal{G}$. It is easy to see that $\grdef^{\leq 0}_k(S)(\Gamma)$ is a groupoid with one object (this is where we are using the nonnegativity hypothesis), and an application of \ref{grpdhfiblem} allows us to conclude that we have an equivalence ${\mathcal{G}\cong\mathrm{hofib}(\grdef^{\leq 0}_A(S)(\Gamma) \to \grdef^{\leq 0}_k(S)(\Gamma))}$. The equivalences of \ref{mcdefms} assemble into a commutative square of groupoids $$\begin{tikzcd} \grdef^{\leq 0}_A(S)(\Gamma)\ar[r]\ar[d,"\simeq"] & \grdef^{\leq 0}_k(S)(\Gamma)\ar[d,"\simeq"] \\  \grdef^{\leq 0}(E)(\Gamma)\ar[r] & \grdef^{\leq 0}(k)(\Gamma)
	\end{tikzcd}$$ with vertical maps equivalences, and it follows that the homotopy fibres of the rows are equivalent. In particular, they have the same $\pi_0$.
\end{proof}
\begin{lem}\label{semidirectprod}
The sequence $\bar E \to E \to k$ induces a short exact sequence of gauge groups $\mathrm{Gg}(\bar E) \to \mathrm{Gg}(E) \to \mathrm{Gg}(k)$. Moreover, this sequence is split exact, so one has a semidirect product decomposition $\mathrm{Gg}(E)\cong  \mathrm{Gg}(\bar E)\rtimes \mathrm{Gg}(k)$.
\end{lem}
\begin{proof}
This is a simple check. Note that the semidirect product decomposition is not respected by homotopies between gauges.
\end{proof}

Now we are ready to prove a set-valued analogue of \ref{mcdefms} for framed deformations.
\begin{prop}\label{frdefmc}	
The functors $\mathrm{Def}^{\mathrm{fr},\leq 0}_A(S)$ and $\mathrm{Def}^{\leq 0}(\bar E)$ are isomorphic.
\end{prop}
\begin{proof}Let $\mathcal{H}$ denote the homotopy fibre of the map $\grdef^{\leq 0}(E) \to \grdef^{\leq 0}(k)$ induced by the augmentation. By \ref{grpdhfiblem}, we may identify $\mathcal{H}$ as the groupoid whose elements are the pairs $(X,u)$ with $X$ an object of $\grdef^{\leq 0}(E)$ and $u:\bullet \to \bullet$ an automorphism of the trivial deformation $\bullet \in \grdef^{\leq 0}(k)$. A morphism $(X,u)\to (X',u')$ is a map $\phi:X \to X'$ in $\grdef^{\leq 0}(E)$ such that $u=u'\circ \phi_k$ where $\phi_k$ denotes the image of $\phi$ under $\grdef^{\leq 0}(E) \to \grdef^{\leq 0}(k)$. Let $u:\bullet \to \bullet$ be an automorphism of the trivial deformation $\bullet \in \grdef^{\leq 0}(k)$. Then $u$ is the homotopy class of a gauge $g \in \mathrm{Gg}(k)$. By \ref{semidirectprod} we may regard $g$ as a gauge in $\mathrm{Gg}(E)$. If $[g]$ denotes the homotopy class of $g$, then $[g]:(X,u)\to ([g]X,\id_\bullet)$ is an isomorphism in $\mathcal{H}$. So $\mathcal{H}$ is equivalent to the groupoid $\mathcal{H}'$ whose objects are deformations $X \in \grdef^{\leq 0}(E)$ and whose morphisms are the maps $\phi:X \to X'$ such that $\phi_k=\id_\bullet$. Let $\mathcal{K}$ be the groupoid whose objects are the same as $\grdef^{\leq 0}(\bar E)$ and morphisms are the gauges $\mathrm{Gg}(\bar E)$ (i.e. $\mathcal{K}$ is the na\"ive Deligne groupoid $\grdef'^{\leq 0}(\bar E)$ of \ref{delignermk}). Observe that, since we are using nonpositive dgas, $\mathcal{K}$ has exactly the same objects as $\grdef^{\leq 0}(E)$. Let $F$ be the functor $\mathcal{K}\to \mathcal{H}'$ that is the identity on objects, and on morphisms sends $g$ to $[g]$. It is well defined since by \ref{semidirectprod} the image of $[g]$ is the identity in the group $\grdef^{\leq 0}(k)$. One has an isomorphism $\mathrm{Def}^{\leq 0}(\bar E)\cong \pi_0\mathcal{K}$, and by \ref{naivegrpdlem} one also has an isomorphism $\mathrm{Def}^{\mathrm{fr},\leq 0}_A(S)\cong\pi_0 \mathcal{H}'$. So to prove the proposition it suffices to show that $\pi_0F$ is an isomorphism. Because $F$ is a surjection on objects, $\pi_0 F$ is a surjection. To show that $\pi_0F$ is an injection, it is enough to check that if there is a morphism $\phi:X \to X'$ in $\mathcal{H}'$ then there is some gauge $g'\in \mathrm{Gg}(\bar E)$ such that $\phi=[g']$. Fix a $\Gamma\in \dgart$ and let $\phi:X \to X'$ be a map in $\mathcal{H}'(\Gamma)$. Pick a representative $g \in \mathrm{Gg}(E)(\Gamma)$ for $\phi$. Because $\phi_k=\id_\bullet$, there is a homotopy $h \in (k\otimes \mathfrak{m}_\Gamma)^{-1}$ such that $g_k.h=\id_\bullet$. Lift $h$ along the surjection $(E\otimes \mathfrak{m}_\Gamma)^{-1}\to (k\otimes \mathfrak{m}_\Gamma)^{-1}$ to obtain an $h'\in (E\otimes \mathfrak{m}_\Gamma)^{-1}$. Then $(g.h')_k=\id_\bullet$, and so by \ref{semidirectprod} we conclude that $g.h'\in \mathrm{Gg}(\bar E)(\Gamma)$. Hence, $g$ is homotopic to some $g'\in\mathrm{Gg}(\bar E)(\Gamma)$, and in particular $\phi=[g']$ is a morphism $X \to X'$.
\end{proof}
\begin{rmk}Continuing the discussion of \ref{frdefsset}, we remark that the proof of \ref{frdefmc} does not adapt to show that the na\"ive groupoid-valued functor of framed deformations is equivalent to $\grdef^{\leq 0}(\bar E)$. This is because the semidirect product decomposition of \ref{semidirectprod} does not preserve homotopies, so the functor $F$ is not an equivalence. To put it another way, the representative $g'\in \mathrm{Gg}(\bar E)$ of $\phi$ chosen at the end of proof of \ref{frdefmc} is not unique up to homotopy in $\grdef^{\leq 0}(\bar E)$. The obstruction to uniqueness is $\pi_2\R\mathrm{Def}^{\leq 0}(k)$, which in general is nontrivial.
\end{rmk}
We can now state our main theorem about prorepresentability.
\begin{thm}\label{prorepfrm}
Let $A$ be an ungraded algebra and $S$ a one-dimensional $A$-module. The set-valued deformation functors $\mathrm{Def}^{\mathrm{fr},\leq 0}_A(S)$ and $[B^\sharp \R\enn_A(S), -]$ are equivalent.
\end{thm}
\begin{proof}
Combine \ref{frdefmc} with \ref{defprorep}, observing that if $E$ is an augmented dga with augmentation ideal $\bar E$ there is an isomorphism $B^\sharp E \cong B_\mathrm{nu}^\sharp \bar E$ of pro-Artinian algebras.
\end{proof}
\begin{rmk}
The functor $\R\mathrm{Def}^\text{fr}_A(S)$ of \ref{frdefsset} is similarly prorepresentable. There is a natural map $\hom(B^\sharp \R\enn_A(S),-)\to \R\mathrm{Def}^\text{fr}_A(S)$, and the results of \cite[Sections 3.2 and 3.3]{jonrepdstacks} show that passing to right derived functors gives a weak equivalence $\R\mathrm{Map}(B^\sharp \R\enn_A(S),-)\to \R\mathrm{Def}^\text{fr}_A(S)$.
\end{rmk}

\begin{rmk}
If $S$ is a simple $A$-module, not necessarily one-dimensional, then the results of this section ought to remain true in some form. The algebra $\R\enn_k(S)$ may be larger, and so not every deformation of $S$ along a nonpositive dga may admit a framing. If $S$ is a direct sum of $m$ simple modules, then one should be able to carry out a similar analysis using \textbf{pointed} deformations, as in \cite{laudalpt} or \cite{kawamatapointed}, where the base ring is no longer $k$ but $k^m$. When $S$ is the direct sum of a finite semisimple collection of perfect $A$-modules, this is done in \cite{ELO2}. Removing the perfect hypothesis ought to be possible; one would have to repeat our Koszul duality arguments in the pointed setting.
\end{rmk}

\section{The derived quotient and the dg singularity category}
In this section we will remark on some properties of Braun-Chuang-Lazarev's derived quotient \cite{bcl}. We will mostly be interested in derived quotients of ungraded algebras by idempotents. Using ideas of the previous two sections, we will prove (\ref{maindefmthm}) that in certain situations the derived quotient admits a deformation-theoretic interpretation. We will consider the relationship between the derived quotient $\dq$ and the singuarity category of the corner ring $R=eAe$. In \ref{dgsing} we will define the dg singularity category, and we will prove a key technical result (\ref{qisolem}) allowing us to compare it to the dg category $\per\dq$ of perfect dg-modules over $\dq$.

\p The derived quotient is a natural object to study, and has been investigated before by a number of authors: for example it appears in Kalck and Yang's work \cite{kalckyang} \cite{kalckyang2} on relative singularity categories, de Thanhoffer de V\"olcsey and Van den Bergh's paper \cite{dtdvvdb} on stable categories, and Hua and Zhou's paper \cite{huazhou} on the noncommutative Mather-Yau theorem. Our study of the derived quotient will unify some of the aspects of all of the above work.

\subsection{Derived localisation}\label{derloc}
The derived quotient is a special case of a general construction -- the derived localisation. Let $A$ be any dga over $k$ (the construction works over any commutative base ring). Let $S \subseteq H(A)$ be any collection of homogeneous cohomology classes. Braun, Chuang and Lazarev define the \textbf{derived localisation} of $A$ at $S$, denoted by $\dloc$, to be a dga universal with respect to homotopy inverting elements of $S$.
\begin{defn}[\cite{bcl}, \S 3]
Let $Q A \to A$ be a cofibrant replacement of $A$. The \textbf{derived under category} $A \downarrow^\mathbb{L} \cat{dga}$ is the homotopy category of the under category $Q A \downarrow \cat{dga}$ of dgas under $Q A$. A $QA$-algebra $f:Q A \to Y$ is \textbf{$S$-inverting} if for all $s \in S$ the cohomology class $f(s)$ is invertible in $HY$. The \textbf{derived localisation} $\dloc$ is the initial object in the full subcategory of $S$-inverting objects of $A \downarrow^\mathbb{L} \cat{dga}$.
\end{defn} 
\begin{prop}[\cite{bcl}, 3.10, 3.4, 3.5]
The derived localisation exists, is unique up to unique isomorphism in the derived under category, and is quasi-isomorphism invariant.
\end{prop}
\begin{rmk}
The derived localisation is the homotopy pushout of the span \mbox{$A \from k\langle S \rangle \to k\langle S, S^{-1}\rangle$}.
\end{rmk}
\begin{defn}[\cite{bcl}, 4.2 and 7.1]
Let $X$ be an $A$-module. Say that $X$ is \textbf{$S$-local} if, for all $s\in S$, the map $s: X \to X$ is a quasi-isomorphism. Say that $X$ is \textbf{$S$-torsion} if $\R\hom_A(X,Y)$ is acyclic for all $S$-local modules $Y$. Let $D(A)_{S\text{-loc}}$ be the full subcategory of $D(A)$ on the $S$-local modules, and let $D(A)_{S\text{-tor}}$ be the full subcategory on the $S$-torsion modules.
\end{defn}
Similarly as for algebras, one defines the notion of the derived localisation $\mathbb{L}_S(X)$ of an $A$-module $X$. It is not too hard to prove the following:
\begin{thm}[\cite{bcl}, 4.14 and 4.15]
Localisation of modules is smashing, in the sense that \newline $X \to X \lot_A \dloc$ is the derived localisation of $X$. Moreover, restriction of scalars gives an equivalence of $D(\dloc)$ with $D(A)_{S\text{-loc}}$.
\end{thm}
One defines a \textbf{colocalisation functor} pointwise by setting $\mathbb{L}^S(X):=\mathrm{cocone}(X \to \mathbb{L}_S(X))$. An easy argument shows that $\mathbb{L}^S(X)$ is $S$-torsion. 
\begin{defn}\label{colocdga}The \textbf{colocalisation} of $A$ along $S$ is the dga $\mathbb{L}^S(A):=\R\enn_A\left(\oplus_{s \in S}\,\mathrm{cone}(A \xrightarrow{s} A)\right)$.
\end{defn}Note that the dga $\mathbb{L}^S(A)$ may differ from the colocalisation of the $A$-module $A$. If $S$ is a finite set, then the dga $\mathbb{L}^S(A)$ is a compact $A$-module, and we get the analogous:
\begin{thm}[\cite{bcl}, 7.6]
Let $S$ be a finite set. Then $D(\mathbb{L}^S(A))$ and $D(A)_{S\text{-tor}}$ are equivalent.
\end{thm}
Neeman-Thomason-Trobaugh-Yao localisation gives the following:
\begin{thm}[\cite{bcl}, 7.3]\label{ntty}
Let $S$ be finite. Then there is a sequence of dg categories $$\per\mathbb{L}^S(A) \to \per A \to \per \dloc$$which is exact up to direct summands.
\end{thm}
\begin{rmk}[\cite{bcl}, 7.9]\label{bclrcl}
The localisation and colocalisation functors fit into a recollement $$\begin{tikzcd}[column sep=huge]
D(A)_{S\text{-loc}} \ar[r]& D(A)\ar[l,bend left=25]\ar[l,bend right=25]\ar[r] & D(A)_{S\text{-tor}}\ar[l,bend left=25]\ar[l,bend right=25]
\end{tikzcd}$$We will see a concrete special case of this in \ref{recoll}.
\end{rmk}

\begin{defn}[\cite{bcl}, 9.1 and 9.2]
Let $A$ be a dga and let $e$ be an idempotent in $H^0(A)$. The \textbf{derived quotient} $\dq$ is the derived localisation $\mathbb{L}_{1-e}A$.
\end{defn}
Clearly, $\dq$ comes with a natural quotient map from $A$. One can write down an explicit model for $\dq$, at least when $k$ is a field.
\begin{prop}\label{drinfeld}
Let $A$ be a dga over $k$, and let $e\in H^0(A)$ be an idempotent. Then the derived quotient $\dq$ is quasi-isomorphic as a dga over $k$ to the dga $$B:=\frac{A\langle h \rangle}{(he=eh=h)}\;, \quad d(h)=e$$with $h$ in degree -1.
\end{prop}
\begin{proof}
This is essentially Proposition 9.6 of \cite{bcl}; because $k$ is a field, $A$ is flat (and in particular left proper) over $k$. One can check that the quotient map $A \to B$ is the obvious one.
\end{proof}
\subsection{Cohomology of the derived quotient}
In this part we will restrict to the case when $A$ is an ungraded $k$-algebra, and $e\in A$ is an idempotent. Write $R$ for the corner ring $eAe$. We will investigate the cohomology of the derived quotient $B:=\dq$. As an $A$-bimodule, $B$ can be thought of as an augmented Tor complex: let $\mu$ be the multiplication map $Ae \lot_{R} eA \to Ae \otimes_{R}eA \to A$, and let $l:A \to B$ be the localisation map.
\begin{prop}\label{dqexact}
There is an exact triangle of $A$-bimodules $Ae \lot_{R} eA \xrightarrow{\mu} A \xrightarrow{l} B \to $.
\end{prop}
\begin{proof}We use the description of \ref{drinfeld}. We check that as an $A$-bimodule, $B$ is quasi-isomorphic to $\mathrm{cone}(\mu)$. Since the map from $A$ to the cone is the identity in degree zero, this will be enough. Observe that a typical element of $B^{-n}$ looks like a path $x_0hx_1h\cdots hx_n$ where $x_0=x_0e$, $x_n=ex_n$, and $x_j=ex_je$ for $0<j<n$. In other words, $B^{-n} \cong Ae \otimes R\otimes\cdots R \otimes eA$, where the tensor products are taken over $k$ and there are $n$ of them. It is clear that this isomorphism is $A$-bilinear. It is not hard to see that the differential is the Hochschild differential: $x_0hx_1h\cdots hx_n$ gets sent to the sum $\sum_{i=0}^{n-1} (-1)^i x_0h\cdots hx_ix_{i+1}h\cdots hx_n$. Consider the truncation $$\tau_{\leq-1}B \simeq \cdots \to Ae\otimes R\otimes R\otimes eA\to Ae\otimes R\otimes eA\to Ae\otimes eA$$From the above, we see that this truncation is exactly the complex $T$ that computes the relative Tor groups $\tor^{R/k}(Ae,eA)$. Since $k$ is a field, the relative Tor groups are the same as the absolute Tor groups, and hence $T$ is a model for the derived tensor product bimodule $Ae\lot_{R}eA$. So we have $$B \simeq \mathrm{cone}(\tau_{\leq-1}B \to A) \simeq \mathrm{cone}(T \to A)\simeq \mathrm{cone}(Ae\lot_{R}eA \to A)$$and one checks that the map $\tau_{\leq-1}B \to A$ is actually the multiplication map $\mu$.
\end{proof}
\begin{rmk}
The exact triangle $Ae \lot_{R} eA \to A \to B \to $ already appears in \cite[\S 7]{kalckyang2}.
\end{rmk}
The following is immediately obtained by considering the long exact sequence associated to the exact triangle $Ae \lot_{R} eA \xrightarrow{\mu} A \xrightarrow{l} B \to $.
\begin{cor}\label{derquotcohom}
Let $A$ be an algebra over a field $k$, and let $e\in A$ be an idempotent.
Then the derived quotient $\dq$ is a nonpositive dga with cohomology spaces 
$$H^j(\dq)\cong\begin{cases}
0 & j>0
\\ A/AeA & j=0
\\ \ker(Ae\otimes_{R}eA \to A) & j=-1
\\ \tor^{R}_{-j-1}(Ae,eA) & j<-1
\end{cases}$$
\end{cor}
\begin{rmk}
The ideal $AeA$ is said to be \textbf{stratifying} if the map $Ae \lot_{R} eA \to AeA$ is a quasi-isomorphism. It is easy to see that $AeA$ is stratifying if and only if $H^0:\dq \to A/AeA$ is a quasi-isomorphism.
\end{rmk}
\subsection{Derived quotients of path algebras}
When $A$ is the path algebra of a quiver with relations, and $e$ is the idempotent corresponding to a set of vertices, then one can interpret the cohomology groups $H^j(\dq)$ in terms of the geometry of the quiver, at least for small $j$. We think of the modules $H^j(\dq)$ as being a (co)homology theory (with coefficients in $k$) for quivers with relations and specified vertices. We will be particularly interested in $H^{-1}(\dq)$, since the description of \ref{derquotcohom} is not particularly explicit.
\begin{defn}Let $(Q,I,V)$ be a triple consisting of a finite quiver $Q$ with a finite set of relations $I$ and a collection of vertices $V$. Let $A=kQ/(I)$ be the path algebra over $k$, and let $e$ be the idempotent of $A$ corresponding to $V$. Write $H_j(Q,I,V;k):=H^{-j}(\dq)$. Note the reindexing, so that $H_*(Q,I,V;k)$ is concentrated in nonnegative degrees.
\end{defn}
Note that we really consider $I$ and the ideal $(I)$ to be different; in particular $(I)$ is usually not finite. By \ref{cohomsupport}, each $H_j(Q,I,V;k)$ is a module over the $k$-algebra $H_0(Q,I,V;k)\cong A/AeA$. It is clear that $H_0(Q,I,V;k)\cong A/AeA$ is the algebra on those paths in the quiver that do not pass through $V$. Dually, $R=eAe$ is the algebra on those paths starting and ending in $V$. To analyse $H_1(Q,I,V;k)$ we need to introduce some new terminology.
\begin{defn}
A \textbf{marked relation} $m$ for the triple $(Q,I,V)$ is a formal sum $m=\sum_iu_i\vert v_i$ with each $u_i \in Ae$ and $v_i \in eA$, such that the composition $\sum_iu_iv_i$ is a relation from $I$. We think of the vertical bar as the `marking'.
\end{defn}
\begin{prop}The module $H_1(Q,I,V;k)$ is spanned over $A/AeA$ by the marked relations.
\end{prop}
\begin{proof}
We know that $H_1(Q,I,V;k)$ is the middle cohomology of the complex $$Ae\otimes_k R \otimes_k eA \xrightarrow{d} Ae \otimes_k eA \xrightarrow{\mu} A$$where $d$ is the Hochschild differential and $\mu$ is the multiplication. If we write a vertical bar instead of the tensor product symbol, it is immediate that each $A$-bilinear combination of marked relations is a $(-1)$-cocycle in this complex. The $(-1)$-cocycles are all of two forms: firstly, those $x$ such that $\mu (x)$ is zero in $kQ$, and secondly, those $x$ such that $\mu (x)\in (I)$. If $\mu (x)=0$ in $kQ$, then $x$ must just be of the form $x=\sum_i(p^i_1\vert p^i_2 p^i_3 - p^i_1p^i_2 \vert p^i_3)$ and it is easy to see that $x$ must be a coboundary, since $d(p\vert q \vert r)=pq\vert r -p\vert qr$. So $H_1(Q,I,V;k)$ is spanned by those $x$ such that $\mu (x)\in (I)$. But this means that $\mu(x)=\sum_i a_i r_i b_i$ where each $r_i$ is a relation, and each $a_i,b_i$ is in $A$. But then we see that $x=\sum_i a_i' m_i b_i'$, where each $m_i$ is a marked relation. So $H_1(Q,I,V;k)$ is spanned over $A$ by the marked relations. Pick a 1-cocycle $amb$, where $a,b \in A$ and $m$ is a marked relation. If either $a$ or $b$ are in $AeA$, then $amb$ is in fact a coboundary: for example if $a=uev$ then $amb=d(u\vert vmb)$. In other words, $H_1(Q,I,V;k)$ is actually spanned over $A/AeA$ by the marked relations. 
\end{proof}
\begin{cor}\label{h1finite}
Let $Q$ be a finite quiver and $I$ a finite set of relations in $kQ$. Set $A:=kQ/(I)$. Pick a set of vertices $V \subseteq Q$ and let $e\in A$ be the corresponding idempotent. Then if $A/AeA$ is finite-dimensional, so is $H_1(Q,I,V;k)$.
\end{cor}
\begin{proof}
There are a finite number of relations and hence a finite number of marked relations: since each relation is of finite length, there are only finitely many ways to mark it. This shows that $H_1(Q,I,V;k)$ is finite over the finite-dimensional algebra $A/AeA$, and hence finite-dimensional. Note that we can get an explicit upper bound for the dimension: write relation $i$ as a sum of monomials $q_i^j$, each of length $\ell_i^j$. It is easy to see that there are homotopies $\vert uv \simeq u\vert v \simeq uv \vert$ in $(\dq)^{-1}$, and so each monomial $q_i^j$ has at most $\max(1,\ell_i^j -1)$ markings that are not homotopic. Put $\ell_i:=\prod_j\max(1,\ell_i^j -1)$; then relation $i$ has at most $\ell_i$ markings, because we can mark each monomial individually. Put $\ell:=\sum_i \ell _i$, so that there are at most $\ell$ marked relations spanning $H_1(Q,I,V;k)$. Hence if $A/AeA$ has dimension $d$, an upper bound for the dimension of $H_1(Q,I,V;k)$ is $d^2\ell$.
\end{proof}

\begin{rmk}
One can use similar ideas to show that $H_2(Q,I,V;k)$ is spanned by cocycles of the form $u\vert v \vert w$, where $uv=v$ and $vw=w$. For if $\sum^n_iu_iv_i\vert w_i=\sum^n_iu_i\vert v_iw_i$, and all of the $u_i,v_i,w_i$ are monomials, then there exists some permutation $\sigma$ such that $u_iv_i=u_{\sigma i}$ and $w_i=v_{\sigma i}w_{\sigma i}$. Restricting to the orbits of $\sigma$, we may assume that $\sigma$ is a cycle, and write $u_iv_i=u_{i+1}$ and $w_{i-1}=v_iw_i$, where the subscripts are taken modulo $ n$. One then shows by induction on $r$ that $\sum_i^ru_i\vert v_i \vert w_i$ is homotopic to $u_1\vert v_1v_2\cdots v_{r-1}v_r\vert w_r$, and the claim follows upon taking $u=u_1, v=v_1\cdots v_n$, and $w=w_n$.
\end{rmk}
\begin{ex}\label{quiv1} Let $Q$ be the quiver $$
    
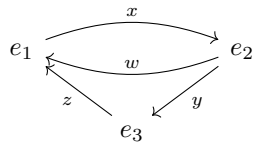
\begin{tikzcd}
e_1 \arrow[rr,bend left=20,"x"]  && e_2 \arrow[ll,bend left=20,"w"']\ar[ld,"y"] \\ & e_3 \ar[lu,"z"]&
\end{tikzcd}$$ with relations $w=yz$ and $xyz=yzx=zxy=0$. Let $e=e_1+e_2$, so that the corresponding set of vertices $V$ is $\{e_1,e_2\}$. It is not hard to compute that $\mathrm{dim}_k(R)=4$, $\mathrm{dim}_k(A)=9$, and $\mathrm{dim}_k(A/AeA)=1$. Moreover, $R$ is not a left or a right ideal of $A$, since $x \in R$ but neither $xy$ nor $zx$ are in $R$. We remark that $R$ need not be the path algebra of the `full subquiver' $Q_V$ on $V$, since relations outside of $V$ can influence $R$: observe that $xw$ is zero in $R$, but nonzero in $kQ_V$. We compute $\ell=7$, and hence our bound for the dimension $n$ of $H_1(Q,I,V;k)$ is $7$. One can check that there are at most $5$ (homotopy classes of) nontrivial marked relations, namely $\vert w -y\vert z$, $x\vert yz$, $yz\vert x$, $z\vert xy$, and $zx\vert y$. So a better estimate for $n$ is $5$. But this is still too large, as $x\vert yz$ and $yz\vert x$ are both coboundaries, and $z\vert xy \simeq zx \vert y$. So $n$ is at most 2. We see that $w -yz$ and $zxy$ are relations in $I$ that cannot be `seen' from $V$: they start and finish outside of $V$, but pass through $V$ (where we can mark them), and moreover do not contain any `subrelations' lying entirely in $V$. In fact, one can check using \ref{derquotcohom} that $n$ is precisely 2, and $H_1(Q,I,V;k)$ has basis $\{\vert w -y\vert z, z\vert xy\}$.
\end{ex}

\subsection{Recollements}
Loosely speaking, a recollement (see \cite{bbd} or \cite{jorgensen} for a definition) between three triangulated categories $(T',T,T'')$ is a collection of functors describing how to glue $T$ from a subcategory $T'$ and a quotient category $T''$. One can think of a recollement as a strong sort of short exact sequence $T' \to T \to T''$.
\begin{thm}[c.f. \cite{kalckyang}, Proposition 2.10 and \cite{bcl}, Remark 9.5]\label{recoll}
Let $A$ be an algebra over $k$, and let $e\in A$ be an idempotent. Write $B:=\dq$ and $R:=eAe$, and put \begin{align*}
i^*:=-\lot_A B, &\quad j_!:= -\lot_{R} eA
\\ i_*=\R\hom_{B}(B,-), &\quad j^!:=\R\hom_A(eA,-)
\\ i_!:=\lot_{B}B, & \quad j^*:=-\lot_A Ae
\\ i^! :=\R\hom_{A}(B,-), & \quad j_*:=\R\hom_{R}(Ae,-)
\end{align*}
Then the diagram of unbounded derived categories
$$\begin{tikzcd}[column sep=huge]
D(B) \ar[r,"i_*=i_!"]& D(A)\ar[l,bend left=25,"i^!"']\ar[l,bend right=25,"i^*"']\ar[r,"j^!=j^*"] & D(R)\ar[l,bend left=25,"j_*"']\ar[l,bend right=25,"j_!"']
\end{tikzcd}$$
is a recollement diagram.
\end{thm}
\begin{proof}We give a rather direct proof. It is clear that $(i^*,i_*=i_!,i^!)$ and $(j_!, j^!=j^*,j_*)$ are adjoint triples, and that $i_*=i_!$ is fully faithful. Fullness and faithfulness of $j_!$ and $j_*$ follow from \cite[2.10]{kalckyang}. The composition $j^*i_*$ is tensoring by the $B-R$-bimodule $B.e$, which is acyclic since $B$ is $e$-killing in the sense of \cite[\S9]{bcl}. The only thing left to show is the existence of the two required classes of distinguished triangles. Put $T:=Ae\lot_{R}eA$. First observe that 
\begin{align*}& i_!i^!\cong \R\hom_A(B,-)
\\ & j_*j^* \cong \R\hom_{R}(Ae,\R\hom_A(eA,-))\cong \R\hom_A(T,-)
\\ & j_!j^! \cong -\lot_A T
\\ &  i_*i^* \cong -\lot_A B\end{align*}
Now, recall from \ref{dqexact} the existence of the distinguished triangle of $A$-bimodules $$T \xrightarrow{\mu} A \xrightarrow{l} B \to$$
Taking any $X$ in $D(A)$ and applying $\R\hom_A(-,X)$ to this triangle, we obtain a distinguished triangle of the form $i_!i^!X \to X \to j_*j^* X \to$. Similarly, applying $X\lot_A-$, we obtain a distinguished triangle of the form $j_!j^!X \to X \to i_*i^* X \to$.
\end{proof}
\begin{rmk}This recollement is given in \cite[9.5]{bcl}, although they are not explicit with their functors. This generalises a recollement constructed by Cline, Parshall, and Scott \cite{cps}. If the ideal $AeA$ is stratifying, then the recollement above reduces to exactly that of \cite{cps}.
\end{rmk}
\begin{prop}\label{Rcoloc}
In the above setup, $D(R)$ is equivalent to the derived category of $(1-e)$-torsion modules.
\end{prop}
\begin{proof}
Recollements are determined completely by fixing one half (e.g. \cite[Remark 2.4]{kalck}). Now the result follows from the existence of the recollement of \ref{bclrcl}. More concretely, one can check that the colocalisation $\mathbb{L}^{1-e}A$ is $R$: because $(1-e)A$ is a summand of $A$, we have $\mathrm{cone}(A \xrightarrow{1-e} A)\cong eA$, and we know that $\R\enn_A(eA)\cong R$ because $eA$ is a projective $A$-module.
\end{proof}
We show that $\dq$ is a relatively compact $A$-module; before we do this we first introduce some notation.
\begin{defn}Let $\mathcal X$ be a subclass of objects of a triangulated category $\mathcal{T}$. Then $\thick_\mathcal{T} \mathcal X $ denotes the smallest triangulated subcategory of $\mathcal{T}$ containing $\mathcal{X}$ and closed under taking direct summands. Similarly, $\langle \mathcal X \rangle_\mathcal{T}$ denotes the smallest triangulated subcategory of $\mathcal{T}$ containing $\mathcal{X}$, and closed under taking direct summands and all existing set-indexed coproducts. We will often drop the subscripts if $\mathcal T$ is clear. If $\mathcal X$ consists of a single object $X$, we will write $\thick X$ and $\langle X \rangle$.
\end{defn}
\begin{ex}
Let $A$ be a dga. Then $\langle A \rangle_{D(A)}\cong D(A)$, whereas $\thick_{D(A)}(A) \cong \per A$.
\end{ex}
\begin{defn}
Let $\mathcal{T}$ be a triangulated category and let $X$ be an object of $\mathcal{T}$. Say that $X$ is \textbf{relatively compact} (or \textbf{self compact}) in $\mathcal{T}$ if it is compact as an object of $\langle X \rangle_{\mathcal{T}}$.
\end{defn}
\begin{prop}
The right $A$-module $\dq$ is relatively compact in $D(A)$.
\end{prop}
\begin{proof}
The embedding $i_*$ is a left adjoint and so respects coproducts. Hence $i_*(\dq)$ is relatively compact in $D(A)$ by \cite[1.7]{jorgensen}. The essential idea is that $\dq$ is compact in $D(\dq)$, the functor $i_*$ is an embedding, and $\langle i_*(\dq) \rangle \subset \im i_*$.
\end{proof}
In situations when $\dq$ is not a compact $A$-module (e.g. when it has nontrivial cohomology in infinitely many degrees), this gives interesting examples of relatively compact objects that are not compact.
\begin{defn}
Let $D(A)_{A/AeA}$ denote the full subcategory of $D(A)$ on those modules $M$ with each $H^j(M)$ a module over $A/AeA$.
\end{defn}
\begin{prop}\label{cohomsupport}There is a natural triangle equivalence $D(\dq)\cong D(A)_{A/AeA}$.
\end{prop}
\begin{proof}Follows from the proof of \cite[2.10]{kalckyang}, along with the fact that recollements are determined completely by fixing one half.
\end{proof}
\begin{prop}\label{semiorthog}
The derived category $D(A)$ admits a semi-orthogonal decomposition $$D(A)\cong \langle D(A)_{A/AeA}, \langle eA \rangle \rangle = \langle \im i_*, \im j_! \rangle $$
\end{prop}
\begin{proof}
This is an easy consequence of \cite[3.6]{jorgensen}.
\end{proof}
We finish with a couple of facts about t-structures (see \cite{bbd} for a definition).
\begin{thm}\label{tstrs}
The category $D(\dq)$ admits a t-structure $\tau$ with $\tau^{\leq 0}=\{X: H^i(X)=0 \text{ for } i>0\}$ and $\tau^{\geq 0}=\{X: H^i(X)=0 \text{ for } i<0\}$. Moreover, $H^0$ is an equivalence from the heart of $\tau$ to $\cat{Mod}-A/AeA$. Furthermore, gluing $\tau$ to the natural t-structure on $D(R)$ via the recollement diagram of \ref{recoll} produces the natural t-structure on $D(A)$.
\end{thm}
\begin{proof}
The first two sentences are precisely the content of \cite[2.1(a)]{kalckyang}. The last assertion holds because gluing of t-structures is unique, and restricting the natural t-structure on $D(A)$ clearly gives $\tau$ along with the natural t-structure on $D(R)$.
\end{proof}

\subsection{Deformation theory}\label{dqdefm}
We give the derived quotient a deformation-theoretic interpretation. The following Proposition generalises an argument given in the proof of \cite[5.5]{kalckyang}.
\begin{prop}\label{dqiskd}
Suppose that $A$ is a $k$-algebra and $e\in A$ an idempotent. Suppose that $A/AeA$ is an Artinian local $k$-algebra, and let $S$ be the quotient of $A/AeA$ by its radical. Suppose furthermore that $\dq$ is cohomologically locally finite. Then the dga $\R\enn_A(S)$ is augmented, and its Koszul dual is quasi-isomorphic to $\dq$.
\end{prop}
\begin{proof}
Because $A/AeA$ is local, $S$ is a one-dimensional $A$-module. The augmentation on $\R\enn_A(S)$ is given by the natural map to $\enn_A(S)\cong k$. Since $S$ is naturally a module over $\dq$, and $D(\dq) \to D(A)$ is fully faithful, we have $\R\enn_A(S)\cong \R\enn_{\dq}(S)$. Note that $\dq \to A/AeA \to S$ is also an augmentation of dgas. Hence we have $(\dq)^!\cong \R\enn_{\dq}(S)$ by \ref{barcofiblem}. So $(\dq)^{!!}\cong \R\enn_A(S)^!$. It now suffices to prove that $(\dq)^{!!}$ is quasi-isomorphic to $\dq$, which follows from an application of \ref{kdfin}.
\end{proof}
\begin{rmk}
This result can be viewed as saying that the derived category $D(\dq)$ is triangle equivalent to its formal completion along $S$ in the sense of \cite{efimov}. If $S$ is perfect over $A$ then one can prove this more directly using results of \cite[\S4]{efimov}.
\end{rmk}

\begin{ex}\label{atiyahflop}
	Proposition \ref{dqiskd} lets us carry out some quite explicit computations. The following appears as Example 5.2.2 of \cite{dtdvvdb}. Let $A$ be the path algebra of the quiver with relations \vspace{-15pt}
	\begin{center}\begin{figure}[ht]\hspace{0.275\linewidth}
			\begin{minipage}[c]{0.2\linewidth}
				\centering
				$\begin{tikzcd}
				1 \arrow[rr,bend left=15,"b"] \arrow[rr,bend left=50,"a"]  && 2 \arrow[ll,bend left=15,"s"] \arrow[ll,bend left=50,"t"] &
				\end{tikzcd}$
			\end{minipage}
			\hspace{0.05\linewidth}
			\begin{minipage}[c]{0.2\linewidth}
				\centering$asb=bsa$\\
				$sbt=tbs$\\
				$atb=bta$\\
				$sat=tas$\\
			\end{minipage}
		\end{figure}
	\end{center}
	\vspace{-30pt} and let $e=e_1$ be the simple at the vertex $1$. One can easily check that $A/AeA\cong k$ (which is certainly Artinian local) since there are no loops at $1$. A resolution of the simple $S=A/AeA$ is given by
	$$
	P_2 \xrightarrow{\left(\begin{smallmatrix} b \\ -a \end{smallmatrix}\right)} P_1^2 \xrightarrow{\left(\begin{smallmatrix} bt & at \\ -bs & -as \end{smallmatrix}\right)} P_1^2 \xrightarrow{(s \; t)} P_2$$where $P_i=e_iA$ is the projective at $i$, and it easily follows that the Ext-algebra of $S$ is $k[x]/x^2$, with $x$ placed in degree 3. It is also easy to see that $\R\enn_A(S)$ must be formal. If $\dq$ is cohomologically locally finite, then it must be the Koszul dual of $k[x]/x^2$, which is the polynomial ring $k[\eta]$ with $\eta$ placed in degree -2.
\end{ex}
\begin{rmk}In future work \cite{dcalg}, we will be able to interpret the dga $k[\eta]$ as the derived contraction algebra of the Atiyah flop, which has base $R=\frac{k[x,y,u,v]}{(xy-uv)}$. Indeed, the algebra $A$ appearing in \ref{atiyahflop} is a noncommutative crepant resolution $A=\enn_R(R\oplus M)$ of $R$. It is true that $\dq$ is cohomologically locally finite, which follows from the fact that $M$ is projective away from the singular point combined with the Ext computations of \ref{stabcat}. More directly, one can deduce that $\dq \cong k[\eta]$ using the proof of Corollary 10.7 of \cite{kalckyang2}. The advantage of our current approach is that it extends without change to the singular setting, where the graded algebra $\ext^*_A(S)$ may not be finite-dimensional.
\end{rmk}
Putting $P:=B^\sharp(\R\enn_A(S))$, the continuous Koszul dual of $\R\enn_A(S)$, we have a quasi-isomorphism $\varprojlim P \cong \dq$. Note that if $Q$ is any other pro-Artinian dga with $\varprojlim Q \cong \dq$, then $Q$ is weakly equivalent to $P$, because $\varprojlim$ reflects weak equivalences by \ref{limreflects}. So up to weak equivalence, one can say that $\dq$ canonically admits the structure of a pro-Artinian dga. Moreover, by \ref{prorepfrm} it prorepresents the functor of framed deformations of $S$:
\begin{thm}\label{maindefmthm}Let $A$ be a $k$-algebra and $e\in A$ an idempotent. Suppose that $A/AeA$ is a local algebra and that $\dq$ is cohomologically locally finite. Let $S$ be $A/AeA$ modulo its radical, regarded as a right $A$-module. Then $\dq$ is naturally a pro-Artinian dga, and there is an isomorphism of functors $\mathrm{Def}^{\mathrm{fr},\leq 0}_A(S)\cong [\dq,-]$.
\end{thm}
\begin{rmk}
A pointed version of this is proved in \cite{huakeller}, under the additional assumption that $A$ has finite global dimension.
\end{rmk}
\begin{rmk}
	The pro-Artinian structure on the dga $k[\eta]$ of \ref{atiyahflop} is the obvious one: namely, we identify it with the inverse limit of the truncations $k[\eta]/\eta^n$. Note that this makes sense because $\eta$ has a nonzero degree; if $\eta$ had degree zero the limit would of course be $k\llbracket \eta\rrbracket$.
\end{rmk}

\subsection{Singularity categories} In this part we recall some results of Kalck and Yang \cite{kalckyang} \cite{kalckyang2} on relative singularity categories from the perspective of the derived quotient. Let $A$ be an algebra over $k$, and let $e\in A$ be an idempotent. Write $R$ for the corner ring $eAe$. Note that by \ref{recoll}, the functor $j_!= -\lot_{R} eA$ embeds $D(R)$ into $D(A)$. In fact, since $D(R)=\langle R\rangle$, we have $j_{!}D(R)=\langle eA \rangle$. Similarly, restricting to compact objects shows that $j_{!}\per R = \thick (eA) \subseteq \per A$. Recall that the \textbf{singularity category} of $R$ is the Verdier quotient $D_\text{sg}R:=D^b(R)/ \per R$. It vanishes if and only if $R$ has finite right global dimension.
\begin{defn}[\cite{kalckyang}]
The \textbf{relative singularity category} is the Verdier quotient $$\Delta_R(A) := \frac{D^b(A)}{j_! \per R} \cong \frac{D^b(A)}{\thick (eA)}$$
\end{defn}
In \cite{kalckyang2}, this is referred to as the \textbf{singularity category of $A$ relative to $e$}. The map $j^*: D(A) \to D(R)$ sends $\thick(eA)$ into $\per R$, and hence defines a map $j^*:\Delta_R(A) \to D_\text{sg}R$. In fact, $j^*$ is onto, which follows from \cite[3.3]{kalckyang}.
\begin{defn}
Write $ D_\mathrm{fg}(\dq)$ for the subcategory of $D(\dq)$ on those modules whose total cohomology is finitely generated over $A/AeA$. Simiarly, write $\per_\mathrm{fg}(\dq)$ for the subcategory of $\per(\dq)$ on those modules whose total cohomology is finitely generated over $A/AeA$.
\end{defn}
\begin{lem}
Assume that $A$ is right noetherian. Then $\ker j^* \cong D_\mathrm{fg}(\dq)$.
\end{lem}

\begin{proof}
The proof of \cite[6.13]{kalckyang} shows that $\ker j^* \cong \thick_{D(A)}(\cat{mod}-A/AeA)$, so it suffices to show that $\thick_{D(A)}(\cat{mod}-A/AeA)\cong D_\mathrm{fg}(\dq)$. But this can be shown to hold via a modification of the proof of \cite[2.12]{kalckyang}.
\end{proof}
\begin{rmk}
If $A/AeA$ is a finite-dimensional algebra, let $\mathcal{S}$ be the finite set of simple $A/AeA$-modules. Then $D_\mathrm{fg}(\dq)\cong \thick(\mathcal{S})$. Because each simple in $\mathcal{S}$ need not be perfect over $\dq$, the category $\per_\mathrm{fg}(\dq)$ may be smaller than $D_\mathrm{fg}(\dq)$. If $A$ has finite right global dimension then each simple is perfect over $A$, hence perfect over $\dq$, and so we have $D_\mathrm{fg}(\dq)\cong \per_\mathrm{fg}(\dq)$.
\end{rmk}
In what follows, we will often need to make the technical assumption that the singularity category $D_\text{sg}(R)$ is idempotent complete. This is the case when $R$ is Gorenstein and a finitely generated module over a commutative complete local noetherian $k$-algebra \cite[Lemma 5.5]{kalckyang2}. The second condition is satisfied for example when $R$ is a finite-dimensional $k$-algebra, or when $R$ is itself a commutative complete local noetherian $k$-algebra. In this setting, Kalck and Yang observed that there is a triangle functor ${\Sigma: \per(\dq) \to D_{\mathrm{sg}}(R)}$, sending $\dq$ to the right $R$-module $Ae$:
\begin{prop}[c.f. \cite{kalckyang2}, 6.6] \label{kymap}
Suppose that $A$ is right noetherian and $D_\text{sg}(R)$ is idempotent complete. Then there is a map of triangulated categories $\Sigma:\per(\dq) \to D_\text{sg}(R)$, sending $\dq$ to $Ae$. Moreover $\Sigma$ has image $\thick_{D_\text{sg}(R)}(Ae)$ and kernel $\per_\mathrm{fg}(\dq)$.
\end{prop}
\begin{proof}
We already have a map $j^*:\Delta_R(A)\to D_\text{sg}(R)$, with kernel $D_\mathrm{fg}(\dq)$. The inclusion $\per A \into D^b(A)$ gives a map $\per A / j_!\per R \to \Delta_R(A)$, which is a triangle equivalence if $A$ has finite right global dimension. The composition $G:\per A / j_!\per R \to D_\text{sg}(R)$ is easily seen to send $A$ to $Ae$, and since $A$ generates $\per A / j_!\per R$, the map $G$ hence has image $\thick(Ae)$. As in \cite[2.12]{kalckyang} (which is an application of Neeman-Thomason-Trobaugh-Yao localisation; c.f. \ref{ntty}), the map $i^*$ gives a triangle equivalence $$i^*:\left(\frac{\per A}{j_!\per R}\right)^\omega \xrightarrow{\cong} \per \dq$$where $\mathcal{T}^\omega$ denotes the idempotent completion of $\mathcal{T}$. Let $\Sigma$ be the map
$$\Sigma:\per(\dq) \xrightarrow{(i^*)^{-1}} \left(\frac{\per A}{j_!\per R}\right)^\omega \xrightarrow{G^\omega} D_\text{sg}(R)^\omega.$$Since $D_\text{sg}(R)$ is idempotent complete, $D_\text{sg}(R)^\omega\cong D_\text{sg}(R)$, and since $i^*$ takes $A$ to $\dq$, we see that $\im \Sigma=\thick(Ae)$ as required. It is easy to see that the kernel of $\Sigma$ is the preimage in $\per (\dq)$ of $D_\mathrm{fg}(\dq)$, which is precisely $\per_\mathrm{fg}(\dq)$.
\end{proof}
For future reference, it will be convenient to give $\Sigma$ a name.
\begin{defn}
We refer to the triangle functor $\Sigma$ of Proposition \ref{kymap} as the \textbf{singularity functor}.
\end{defn}
\begin{prop}[\cite{kalckyang2}, 1.2]\label{dsgsmooth}Suppose that $A$ is right noetherian and of finite right global dimension, and $D_\text{sg}(R)$ is idempotent complete. Then the singularity functor induces a triangle equivalence $$\Sigma:\frac{\per(\dq)}{D_{\text fg}(\dq)} \to D_\text{sg}(R).$$
\end{prop}
\begin{proof}
The image of $\Sigma$ is all of $D_\text{sg}(R)$ because $j^*:\Delta_R(A)\to D_\text{sg}(R)$ is onto. Since $A$ has finite right global dimension, every finitely generated $\dq$-module is perfect over $A$. Since $i^*$ respects compact objects, $D_{\text fg}(\dq) \cong i^*i_*D_{\text fg}(\dq)\subseteq \per(\dq)$. Hence, $\per_\mathrm{fg}(\dq)=D_\mathrm{fg}(\dq)$.
\end{proof}
\begin{rmk}
This equivalence is essentially the same as that of \cite[5.1.1]{dtdvvdb}.
\end{rmk}
\begin{rmk}Suppose that $A$ is right noetherian. Then using the Nine Lemma for triangulated categories one can show that the inclusion $D_{\text fg}(\dq) \into \Delta_R(A)$ and the projection $D^b(A) \onto D_{\text{sg}}(R)$ induce a sequence$$\frac{D_{\text fg}(\dq) }{\per_\mathrm{fg}(\dq)} \longrightarrow D_\text{sg}(A) \longrightarrow \frac{D_{\text{sg}}(R)}{\thick_{D_\text{sg}(R)}(Ae)}$$which is exact up to direct summands. Intuitively, this tells us that $\dq$ is not more singular than $A$. Indeed, if $B$ is an unbounded dga then $D_{\text fg}(B)$ should be thought of as  $D^b(B)$, and the quotient $D_{\text fg}(B) /\per_\mathrm{fg}(B)$ should be thought of as the singularity category $D_{\text{sg}}(B)$.
\end{rmk}

\subsection{DG categories}\label{dgcats}
In this part we will recall some facts about dg categories, in particular about the dg quotient. Survey articles on dg categories include \cite{toendglectures} and \cite{keller}.
\begin{defn}
A ($k$-linear) \textbf{dg category} is a category $\mathcal{C}$ enriched over the monoidal category $(\cat{dgvect}_k,\otimes)$ of dg-vector spaces with the usual tensor product. In other words, to every pair of elements $(x,y)\in\mathcal{C}^2$ we assign a chain complex $\dgh_\mathcal{C}(x,y)$, to every triple $(x,y,z)$ we assign a chain map $\dgh_\mathcal{C}(x,y)\otimes \dgh_\mathcal{C}(y,z) \to \dgh_\mathcal{C}(x,z)$ satisfying associativity, and for every $x \in \mathcal{C}$ we assign a map $k \to \dgh_\mathcal{C}(x,x)$ which is a unit with respect to composition.
\end{defn}
Note in particular that for any object $x \in\mathcal{C}$, the complex $\dge_{\mathcal{C}}(x):=\dgh_\mathcal{C}(x,x)$ naturally has the structure of a (unital) dga. We will frequently omit the subscript $\mathcal{C}$ if the context is clear.
\begin{rmk}
	A more usual notation for the enriched hom is $\underline{\hom}$. We will not use this since it risks confusion with the standard notation used for homsets in the stable category of a ring, which we will use later in \ref{stabcat}.
\end{rmk}
\begin{defn}
A \textbf{dg functor} $F:\mathcal{C}\to\mathcal{D}$ between two dg categories is an enriched functor; i.e. a map of objects $\mathcal{C}\to \mathcal{D}$ together with, for every pair $(x,y)\in\mathcal{C}^2$, a component $F_{xy}:\dgh_\mathcal{C}(x,y) \to \dgh_\mathcal{D}(Fx,Fy)$ satisfying unitality and associativity conditions.
\end{defn}
In particular, a dg functor $F:\mathcal{C}\to\mathcal{D}$ induces dga morphisms $F_{xx}:\dge_{\mathcal{C}}(x)\to\dge_{\mathcal{D}}(Fx)$ for every $x \in \mathcal{C}$.
\begin{defn}
Let $\mathcal{C}$ be a dg category. The \textbf{homotopy category} of $\mathcal{C}$ is the $k$-linear category $[\mathcal{C}]$ whose objects are the same as $\mathcal{C}$ and whose hom-spaces are $\hom_{[\mathcal{C}]}(x,y):=H^0(\dgh_\mathcal{C}(x,y))$. Composition is inherited from $\mathcal{C}$.
\end{defn}
\begin{defn}Let $F:\mathcal{C}\to\mathcal{D}$ be a dg functor.\begin{itemize}
\item $F$ is \textbf{quasi-fully faithful} if all of its components $F_{xy}$ are quasi-isomorphisms.
\item $F$ is \textbf{quasi-essentially surjective} if the induced functor $[F]:[\mathcal{C}]\to[\mathcal{D}]$ is essentially surjective.
\item $F$ is a \textbf{quasi-equivalence} if it is quasi-fully faithful and quasi-essentially surjective.
\end{itemize}
\end{defn}
In a dg category, one may define shifts and mapping cones via the Yoneda embedding into the category of modules. This is equivalent to defining them as representing objects of the appropriate functors; e.g. $x[1]$ should represent $\dgh(x,-)[-1]$.
\begin{defn}
Say that a dg category is \textbf{pretriangulated} if it contains a zero object and is closed under shifts and mapping cones. 
\end{defn}If $\mathcal{C}$ is pretriangulated then the homotopy category $[\mathcal{C}]$ is canonically triangulated, with translation functor given by the shift. We list some standard pretriangulated dg categories:
\begin{defn}\label{dgcatlist}
If $A$ is a dga, then $D_\mathrm{dg}(A)$ is the dg category of cofibrant dg-modules over $A$, and $\cat{per}_\mathrm{dg}(A) \subseteq {D}_\mathrm{dg}(A)$ is the dg-subcategory on compact objects. In addition, if $A$ is a $k$-algebra then $D^b_\mathrm{dg}(A)$ denotes the dg category of cofibrant dg-$A$-modules with bounded cohomology; these are precisely the bounded above complexes of projective $A$-modules with bounded cohomology.
\end{defn}
 All of these dg categories are pretriangulated. In the notation of \cite{toendglectures}, $\cat{per}_\mathrm{dg}(A)$ is $\hat{A}_{\text{pe}}$. One has equivalences of triangulated categories $[D_\mathrm{dg}(A)]\cong D(A)$, ${[D^b_\mathrm{dg}(A)]\cong D^b(A)}$ and $[\cat{per}_\mathrm{dg}(A)]\cong\cat{per}(A)$, via standard arguments about dg model categories \cite{toendglectures}. Note that in the dg categories above, $\dgh$ is a model for the derived hom $\R\hom$; we will implicitly use this fact often.

\p If $A$ is a ring, then the singularity category of $A$ is the Verdier quotient $D^b(A)/\cat{per}(A)$. One can also take dg quotients of dg categories; these were first considered by Keller \cite{kellerdgquot} and an explicit construction using ind-categories was given by Drinfeld \cite{drinfeldquotient}. Note that if $\mathcal{C}$ is a dg category then so is $\cat{ind}\mathcal{C}$.
\begin{defn}[\cite{drinfeldquotient}]
	Let $\mathcal{A}$ be a dg category and $\mathcal{B}$ a full dg subcategory. Then the dg quotient $\mathcal{A}/\mathcal{B}$ is the subcategory of $\cat{ind}\mathcal{A}$ on those $X$ such that:\begin{enumerate}
		\item $\dgh_{\cat{ind}\mathcal{A}}(\mathcal{B},X)$ is acyclic.
		\item There exists $a \in \mathcal{A}$ and a map $f:a \to X$ with $\mathrm{cone}(f)\in \cat{ind}\mathcal{B}$.
	\end{enumerate}
Since $\cat{ind}\mathcal{A}$ is a dg category, so is $\mathcal{A}/\mathcal{B}$.
\end{defn}
\begin{thm}[\cite{drinfeldquotient}, 3.4]\label{drinfeldpretr}
Let $\mathcal{A}$ be a pretriangulated dg category and $\mathcal{B}\into \mathcal{A}$ a full pretriangulated dg subcategory. Then $[\mathcal{A}/\mathcal{B}]\cong [\mathcal{A}]/[\mathcal{B}]$. In other words, the dg quotient is a dg enhancement of the Verdier quotient.
\end{thm}
One useful universal property of the dg quotient is that it can be viewed as a homotopy cofibre.
\begin{thm}[\cite{tabquot}, 4.02]
Let $\mathcal{A}$ be a dg category and $i:\mathcal{B}\into \mathcal{A}$ a full dg subcategory. Then the quotient $\mathcal{A}/\mathcal{B}$ is the homotopy cofibre of $i$, taken in the homotopy category of dg categories with quasi-equivalences.
\end{thm}
\begin{defn}\label{dgsing}
Let $A$ be a $k$-algebra. The \textbf{dg singularity category} of $A$ is the Drinfeld quotient ${D}^{\mathrm{dg}}_\mathrm{sg}(A):=D^b_{\mathrm{dg}}(A) / \cat{per}_\mathrm{dg}(A)$. If $e\in A$ is an idempotent, write $R$ for the corner ring $eAe$ and $j_!$ for the functor $-\otimes^\mathbb{L}eA:D(R)\to D(A)$. It is easy to see that $j_!$ admits a dg enhancement. The \textbf{dg relative singularity category} is the Drinfeld quotient ${\Delta}^{\mathrm{dg}}_R(A):=D^b_{\mathrm{dg}}(A) / j_!\cat{per}_\mathrm{dg}(R)$.
\end{defn}
By \ref{drinfeldpretr}, we have  $[{D}^{\mathrm{dg}}_\mathrm{sg}(A)]\cong D_\mathrm{sg}(A)$ and $[{\Delta}^{\mathrm{dg}}_R(A)]\cong {\Delta}_R(A)$. The following easy Lemma is useful, since we will want to quotient by perfect complexes:
\begin{lem}\label{drinfeldlem}
	Let $\mathcal{A}$ be a dg category and $\mathcal{B}$ a full dg subcategory of compact objects. Let $X\in \cat{ind}\mathcal{A}$. Then for all $b \in \mathcal{B}$, $\dgh_{\cat{ind}\mathcal{A}}(b,X)\cong \dgh_\mathcal{A}(b,\varinjlim X)$. In particular, if $\varinjlim X\cong 0$ then $\dgh_{\cat{ind}\mathcal{A}}(\mathcal{B},X)$ is acyclic. The converse is true if $\mathcal{B}$ contains a generator of $\textbf{A}$.
\end{lem}
Hence, the objects of ${D}^{\mathrm{dg}}_\mathrm{sg}(A)$ are precisely those ind-objects $X \in D^b_{\mathrm{dg}}(A)$ such that $\varinjlim X$ is acyclic and there is an $M \in D^b_{\mathrm{dg}}(A)$ with a map $M \to X$ with ind-perfect cone.

\subsection{The dg singularity category}
In this part, suppose that $A$ is a right noetherian $k$-algebra with idempotent $e$. Put $R:=eAe$. We will assume that $D_{\mathrm{sg}}(R)$ is idempotent complete. We show in \ref{Fdg} that the singularity functor $\Sigma:\cat{per}(\dq) \to D_{\mathrm{sg}}(R)$ lifts to a dg functor. The component of the singularity functor at $\dq$ is a dga map from $\dq$ to an endomorphism dga in ${D}^{\mathrm{dg}}_\mathrm{sg}(R)$, and in \ref{qisolem} we examine the induced map on cohomology.
\begin{prop}\label{Fdg}The singularity functor $\Sigma:\cat{per}(\dq) \to D_{\mathrm{sg}}(R)$ lifts to a dg functor $ \Sigma_\mathrm{dg}:\cat{per}_\mathrm{dg}(\dq) \to {D}^{\mathrm{dg}}_\mathrm{sg}(R)$, which we refer to as the \textbf{dg singularity functor}.
\end{prop}
\begin{proof}
We simply mimic the proof in the triangulated setting. Recalling from \ref{kymap} the construction of $\Sigma$ as a composition $\per(\dq) \xrightarrow{\Sigma_1} \Delta_R(A) \xrightarrow{\Sigma_2} {D}_\mathrm{sg}(R)$, we lift the two maps separately to dg functors. To lift $\Sigma_1$, first note that \ref{ntty} and \ref{Rcoloc} provide a homotopy cofibre sequence of dg categories $\per_{\mathrm{dg}}R \to \per_{\mathrm{dg}}A \to \per_{\mathrm{dg}}(\dq)$, in which the first map is $j_!$. There is a homotopy cofibre sequence $\per_{\mathrm{dg}}R \to D_\mathrm{dg}^b(A) \to \Delta^\mathrm{dg}_R(A)$, and we can extend $\id: \per_{\mathrm{dg}}R \to \per_{\mathrm{dg}}R$ and the inclusion $\per_{\mathrm{dg}}A \to D_\mathrm{dg}^b(A)$ into a morphism of homotopy cofibre sequences, which gives a lift of $\Sigma_1$. Lifting $\Sigma_2=j^*$ is similar and uses the sequence $\per_{\mathrm{dg}}R \to D_\mathrm{dg}^b(R) \to {D}^{\mathrm{dg}}_\mathrm{sg}(R)$.
\end{proof}
\begin{rmk}
By \ref{kymap}, the kernel of $\Sigma_\mathrm{dg}$ is a dg enhancement of the triangulated category $\per_\mathrm{fg}(\dq)$.
\end{rmk}
Observe that $\Sigma(\dq)\simeq Ae$. Since we can canonically identify $\dq$ with the endomorphism dga $\dge_{\per_{\mathrm{dg}}(\dq)}(\dq)$, the component of $\Sigma_\mathrm{dg}$ at $\dq$ gives a dga map $\dq \to \dge_{D^\mathrm{dg}_\text{sg}(R)}(Ae)$.
\begin{defn}\label{comparisonmap}
The \textbf{comparison map} $\Xi:\dq \to \dge_{D^\mathrm{dg}_\text{sg}(R)}(Ae)$ is the component of the dg singularity functor $\Sigma_{\mathrm{dg}}$ at the object $\dq \in {\per_{\mathrm{dg}}(\dq)}$.
\end{defn}
The main theorem of this part, which we are about to prove, is that under certain Cohen-Macaulay type finiteness conditions, $\Xi$ is a `quasi-isomorphism in nonpositive degrees': the truncated map $\Xi:\dq \to \tau_{\leq 0}\dge(Ae)$ is a quasi-isomorphism. Note that we have an isomorphism $H^j(\dge(Ae)) \cong \ext^j_{D_\text{sg}(R)}(Ae,Ae)$, so this allows us to interpret the cohomology of $\dq$ in terms of Ext groups in the singularity category. The proof will be an explicit computation: we pick models for $A$ and $\dge(Ae)$ and examine the induced map.
\begin{thm}\label{qisolem}Let $A$ be a right noetherian $k$-algebra with idempotent $e$. Put $R:=eAe$ and assume that $D_{\mathrm{sg}}(R)$ is idempotent complete. Assume furthermore that $R$ is noetherian on both sides, that $Ae$ is a finitely generated $R$-module, and that the natural map $\R\hom_R(Ae,R)\to\hom_R(Ae,R)$ is a quasi-isomorphism. Then, for all $j\leq 0$, the comparison map $\Xi$ induces isomorphisms $$H^j(\Xi): H^j(\dq) \xrightarrow{\simeq} H^j(\dge_{D^\mathrm{dg}_\text{sg}(R)}(Ae)) .$$
\end{thm}
\begin{proof}Let $B(Ae)=\cdots \to Ae \otimes_k R \otimes_k R \to Ae\otimes_k R$ be the bar resolution of the $R$-module $Ae$. Let $\mathrm{Bar}(Ae)$ be the filtered system of perfect submodules of $B(Ae)$; we then have $B(Ae)\cong \varinjlim \mathrm{Bar}(Ae)$. Upon applying $-\otimes_R eA$ to $B(Ae)$ we obtain (one model of) $Ae \lot_{R} eA$. Noting that $j_!=-\lot_R eA$, we hence see that $\mathrm{Bar}(Ae)\otimes_R eA$ is an object of $\cat{ind}j_!\per_{\mathrm{dg}} R$. Since tensor products commute with filtered colimits, we have$$\varinjlim(\mathrm{Bar}(Ae)\otimes_R eA)\cong  \cdots \to Ae\otimes R\otimes R\otimes eA\to Ae\otimes R\otimes eA\to Ae\otimes eA\cong Ae \lot_{R} eA$$where the tensor products are taken over $k$. Observe that this is precisely the $A$-bimodule $T$ that appears in \ref{dqexact}. Note that $\mathrm{Bar}(Ae)\otimes_R eA$ also comes with a multiplication map $\mu$ to $A$ that lifts the multiplication $Ae\otimes_R eA \to A$. Let $C$ be the cone of this map; then $C$ is an ind-bounded module with a map from $A$ whose cone is in $\cat{ind}j_!\per_{\mathrm{dg}} R$. In fact, $\varinjlim C \cong \dq$: this is the content of \ref{dqexact}. Hence, if $P\in j_!\per_{\mathrm{dg}} R=\thick_{\mathrm{dg}}(eA)$, then $\dgh_{A}(P,C)\cong\R\hom_A(P,C) \cong\R\hom_A(P,\dq)\cong 0$, since $P$ is perfect and we have the semi-orthogonal decomposition of \ref{semiorthog}. Hence $C$ is a representative of the $A$-module $\dq$ in the Drinfeld quotient $\Delta^\mathrm{dg}_R(A)$.

\p We examine how $\dq$ acts on $C$. Note that the action of $\dq$ on itself on the left is particularly simple: let $x\otimes r_1 \otimes \cdots \otimes r_i \otimes y$ and $w\otimes s_1 \otimes \cdots \otimes s_j \otimes z$ be pure tensors in $\dq$ of degree $-1-i$ and $-1-j$ respectively. Then their product is just the concatenation $x\otimes r_1 \otimes \cdots \otimes r_i \otimes yw\otimes s_1 \otimes \cdots \otimes s_j \otimes z$, and the obvious analogous statement holds for degree zero elements. Hence, $\dq$ acts on $C$ by concatenation. Note that the dg functor $j^*: \Delta^\mathrm{dg}_R(A) \to {D}^{\mathrm{dg}}_\mathrm{sg}(R)$ is simply multiplication on the right by $e$. Hence, sending $C$ through this map, we obtain the ind-object $Ce\cong\mathrm{cone}(\mathrm{Bar}(Ae) \to Ae)$ that represents $\Sigma(\dq)\cong Ae$ in the dg singularity category $D^\mathrm{dg}_\text{sg}(R)$. As an aside, one can check this directly: since $B(Ae)$ resolves $Ae$, and mapping cones commute with filtered colimits, it is clear that $\varinjlim Ce$ is acyclic, and that $Ce$ admits a map from $Ae \in D^b(R)$ whose cone is the ind-perfect $R$-module $\mathrm{Bar}(Ae)$.

\p Similarly, $\dq$ acts on $Ce$ by concatenation. Now we will explicitly identify $\dge_{D^\mathrm{dg}_\text{sg}(R)}(Ae)= \dge(Ce)$; to simplify notation, we will frequently omit subscripts. We will see that in strictly negative degrees, $\dge(Ce)$ is precisely the complex $T$, and that the induced action of $T$ on $Ce$ is precisely the concatenation action. This will almost be enough to give us what we want. Write $Ce=\{W_\alpha\}_\alpha$, where each $W_\alpha$ is a cone $V_\alpha \to Ae$ with $V_\alpha$ perfect. We compute \begin{align*}
	\dge(Ce) & :=\varprojlim_\alpha\varinjlim_\beta \R\hom_R\left(\mathrm{cone}(V_\alpha \to Ae), W_\beta\right) \\ & \cong \varprojlim_\alpha\varinjlim_\beta \mathrm{cocone}\left[ \R\hom_R(Ae, W_\beta) \to\R\hom_R(V_\alpha, W_\beta)\right] 
	\\ & \cong \varprojlim_\alpha \mathrm{cocone}\left[\varinjlim_\beta  \R\hom_R(Ae, W_\beta) \to \varinjlim_\beta\R\hom_R(V_\alpha, W_\beta)\right] 
	\\ & \cong \varprojlim_\alpha \mathrm{cocone}\left[\varinjlim_\beta  \R\hom_R(Ae, W_\beta) \to \R\hom_R(V_\alpha, \varinjlim_\beta W_\beta)\right]&&\text{since } V_\alpha \text{ is compact}
	\\ & \cong \varinjlim_\beta  \R\hom_R(Ae, W_\beta) &&\text{since } \varinjlim_\beta W_\beta\cong 0.
	\\ & \cong  \varinjlim_\beta \mathrm{cone}\left[\R\hom_R(Ae,V_\beta) \to \R\hom_R(Ae,Ae)\right] 
	\\ & \cong \mathrm{cone}\left[\varinjlim_\beta \R\hom_R(Ae,V_\beta) \to \R\enn_R(Ae)\right]. 
	\end{align*}
	Fix a $\beta$ and consider $\R\hom_R(Ae, V_\beta)$. Since $V_\beta$ is perfect, and $Ae$ has some finitely generated projective resolution, we can write $\R\hom_R(Ae, V_\beta)\cong V_\beta \otimes_R \R\hom_R(Ae,R)$. By assumption, $\R\hom_R(Ae,R)\cong\hom_R(Ae,R)$, so that $\R\hom_R(Ae, V_\beta)$ is quasi-isomorphic to the tensor product $V_\beta \otimes_R \hom_R(Ae,R)$. The natural isomorphism $eA\otimes_A Ae \to R$ gives an isomorphism
$\hom_R(Ae,R)\cong eA$, so we get $\R\hom_R(Ae, V_\beta)\cong V_\beta \otimes_R eA$. So we have $$\varinjlim_\beta \R\hom_R(Ae,V_\beta)\cong\varinjlim_\beta(  V_\beta \otimes_R eA)\cong B(Ae)\otimes_R eA=T.$$Hence, we have $\dge(Ce) \cong \mathrm{cone}(T \to \R\enn_R(Ae))$. From the description above, it is not hard to see that the map $T \to \R\enn_R(Ae)$ is exactly the multiplication map $\mu$. More precisely, $xe \otimes ey \in T^0$ is sent to the derived endomorphism that multiplies by $xey \in AeA$ on the left. More generally, elements
	 of $T$ act on $Ce$ via the concatenation action. It now follows that $H^j\Xi$ must be an isomorphism in strictly negative degrees.
	
	\p It just remains to check that $H^0\Xi$ is an isomorphism. For this, it suffices to look at $\tau_{\leq 0}\dge(Ce)$, which is precisely the cone of $T \xrightarrow{\mu} A$, since the degree 0 cocycles are exactly $\enn_R(Ae)\cong A$. But now the claim is clear.
\end{proof}
Note that $\R\hom_R(Ae,R)\cong \hom_R(Ae,R)$ is a sort of Cohen-Macaulay condition, and will be satisfied when $A$ is a \textbf{partial resolution} of $R$; see \ref{partrsln} for a definition. With this in mind, the following two Corollaries can be seen as generalising a theorem of Buchweitz \cite[6.1.2.ii]{buchweitz} to the non-Gorenstein case; we will return to them in \ref{stabcat}.
\begin{cor}\label{posstabext}
Under the assumptions of \ref{qisolem}, then for $j>1$ there is an isomorphism $$\ext^j_{D_\mathrm{sg}(R)}(Ae,Ae)\cong \ext_R^{j}(Ae,Ae) .$$
\end{cor}
\begin{proof}
The proof above writes $\dge(Ce)$ as a cone $T \to \R\enn_R(Ae)$ with $T$ in negative degrees. Hence there is a quasi-isomorphism $\tau_{>0}\R\enn_R(Ae) \to \tau_{>0} \dge(Ce)$.
\end{proof}
\begin{cor}\label{negstabext}
Under the assumptions of \ref{qisolem}, then for $j<-1$, there is an isomorphism $$\ext^j_{D_\mathrm{sg}(R)}(Ae,Ae)\cong \tor^R_{-j-1}(Ae,eA) .$$
\end{cor}
\begin{proof}
Immediate from \ref{derquotcohom}.
\end{proof}

\section{Partial resolutions of Gorenstein rings}
In this section, we will bring in some geometric hypotheses. We will assume that the corner ring $R$ is Gorenstein, which will allow us to use Buchweitz's machinery of the stable category \cite{buchweitz}. We will specialise to the case when $R$ is a commutative complete local hypersurface, where it is well known that the shift functor of $D_\mathrm{sg}(R)$ is 2-periodic. This will then give us periodicity in the dga $\dq$ (\ref{etaex}), which will allow us to identify the cohomology algebra of $\dq$ explicitly. We will apply a recovery result of Hua-Keller \cite{huakeller} to prove that in certain situations, the quasi-isomorphism class of $\dq$ recovers the geometry of $R$ (\ref{recov}).
\subsection{The stable category}\label{stabcat}
\begin{defn}
Let $R$ be a $k$-algebra. Say that $R$ is \textbf{Gorenstein} (or \textbf{Iwanaga-Gorenstein}) if it is noetherian on both sides and of finite left and right injective dimension over itself.
\end{defn}
Complete intersections are Gorenstein:
\begin{prop}[\cite{eisenbud}, Corollary 21.19]\label{hypsgor}
Let $S$ be a commutative noetherian regular local ring and $I \subseteq S$ an ideal generated by a regular sequence. Then $S/I$ is Gorenstein.
\end{prop}
In particular, if $S$ is a commutative noetherian regular local ring and $\sigma \in S$ is a non-zerodivisor then the hypersurface $S/\sigma$ is Gorenstein.
\begin{defn}
Let $R$ be a Gorenstein $k$-algebra. If $M$ is an $R$-module, write $M^\vee$ for the $R$-linear dual $\hom_R(M,R)$. A finitely generated $R$-module $M$ is \textbf{maximal Cohen-Macaulay} or just \textbf{MCM} if the natural map $\R\hom_R(M,R)\to M^\vee$ is a quasi-isomorphism.
\end{defn}
An equivalent characterisation of MCM $R$-modules is that they are those modules $M$ for which $\ext_R^j(M,R)$ vanishes whenever $j> 0$.
\begin{defn}Let $R$ be a Gorenstein $k$-algebra. The \textbf{stable category} $\stab R$ of $R$ is the category whose objects are MCM modules over $R$, and whose homsets are $\hom_{\stab R}(X,Y):=\underline{\hom}_R(X,Y)$, the set of $R$-linear maps between $X$ and $Y$ modulo those that factor through sums of summands of $R$.
\end{defn}The stable category is triangulated, with shift given by syzygy: for each $X$, choose a surjection $f:R^n \to X$, and set $\Omega X:=\ker f$. Up to projective summands, $\Omega$ is well defined, and since projective objects go to zero in $\stab R$, it defines a functor on $\stab R$. In fact, it is an endofunctor, in that $\Omega X$ is again MCM. A famous theorem of Buchweitz tells us that the stable category is the same as the singularity category:
\begin{thm}[\cite{buchweitz}]Let $R$ be a Gorenstein $k$-algebra. The categories $D_{\text{sg}}(R)$ and $\stab R$ are triangle equivalent, via the map that sends an MCM module $M$ to the object $M \in D^b(R)$.
\end{thm}
Hence, we can regard the dg singularity category $D^\mathrm{dg}_{\text{sg}}(R)$ as a dg-enhancement of $\stab R$.
\begin{defn}
Let $R$ be a Gorenstein $k$-algebra and let $x,y \in D^\mathrm{dg}_{\text{sg}}(R)$. Write $\R\underline{\hom}_R(x,y)$ for the complex $\dgh_{D^\mathrm{dg}_{\text{sg}}(R)}(x,y)$ and write $\R\underline{\enn}_R(x)$ for the dga $\dge_{D^\mathrm{dg}_{\text{sg}}(R)}(x)$.
\end{defn}We denote Ext groups in the singularity category by $\underline{\ext}$. Note that $\underline\hom$ coincides with $\underline\ext^0$, and that $\underline{\ext}^j(x,y)\cong H^j\R\underline{\hom}_R(x,y)$. In order to investigate the stable Ext groups, we recall the notion of the complete resolution of an MCM $R$-module -- the construction works for arbitrary complexes in $D^b(\cat{mod-}R)$.
\begin{defn}[\cite{buchweitz}, 5.6.2.]
Let $R$ be a Gorenstein $k$-algebra and let $M$ be any MCM $R$-module. Let $P$ be a projective resolution of $M$, and let $Q$ be a projective resolution of $M^\vee$. Dualising and using that $(-)^\vee$ is an exact functor on MCM modules and on projectives gives us a projective coresolution $M \to Q^\vee$. The \textbf{complete resolution} of $M$ is the (acyclic) complex $\mathbf{CR}(M):=\mathrm{cocone}(P \to Q^\vee)$. So in nonpositive degrees, $\mathbf{CR}(M)$ agrees with $P$, and in positive degrees, $\mathbf{CR}(M)$ agrees with $Q^\vee[-1]$.
\end{defn}
\begin{prop}[\cite{buchweitz}, 6.1.2.ii]\label{buchcohom}Let $R$ be a Gorenstein $k$-algebra and let $M,N$ be MCM $R$-modules. Then $$\underline{\ext}_R^j(M,N) \cong H^j\hom_R(\mathbf{CR}(M),N) .$$
\end{prop}
In particular, if $j>0$, then $\underline{\ext}_R^j(M,N)\cong\ext_R^j(M,N)$. If $j<-1$, note that if $Q$ is as above, we certainly have $\underline{\ext}_R^j(M,N)\cong\ext_R^{-j-1}(Q^\vee,N)$.  Now, \cite{buchweitz}, 6.2.1.ii tells us that $\R\hom_R(Q^\vee,N)$ is quasi-isomorphic to $N\lot_R M^\vee$. Hence, we get $\underline{\ext}_R^j(M,N)\cong\tor^R_{-j-1}(N,M^\vee)$. Note the similarity to \ref{negstabext}.
\begin{defn}\label{partrsln}
	Let $R$ be a Gorenstein $k$-algebra. A $k$-algebra $A$ is a \textbf{(noncommutative) partial resolution} of $R$ if it is of the form $A\cong\enn_R(R\oplus M)$ for some MCM $R$-module $M$. Note that $A$ is a finitely generated module over $R$, and hence itself a noetherian $k$-algebra. Say that a partial resolution is a \textbf{resolution} if it has finite global dimension.
\end{defn}
In the sequel, we will refer to the following setup:
\begin{setup}\label{sfsetup}
Let $R$ be a Gorenstein $k$-algebra. Assume that $D_{\text{sg}}(R)$ is idempotent complete (e.g. this is the case when $R$ is a commutative complete local ring \cite[5.5]{kalckyang2}). Fix an MCM $R$-module $M$ and let $A=\enn_R(R\oplus M)$ be the associated partial resolution. Let $e\in A$ be the idempotent $e=\id_R$. 
\end{setup}
Observe that in the situation of Setup \ref{sfsetup} we have $Ae\cong R\oplus M$ and $eA\cong R \oplus M^\vee$; in particular $(Ae)^\vee\cong eA$. Note that $Ae\cong M $ in the singularity category, and indeed we have $A/AeA\cong \underline{\enn}(M)$. Theorem \ref{qisolem} immediately recovers a categorified version of \ref{buchcohom}:
\begin{prop}Suppose that we are in the situation of Setup \ref{sfsetup}. The comparison map $\Xi:\dq \to \R\underline{\enn}_R(M)$ induces a quasi-isomorphism $\dq \to \tau_{\leq 0}\R\underline{\enn}_R(M)$. Moreover, $\tau_{>0}\R\underline{\enn}_R(M)$ is quasi-isomorphic to $\tau_{>0}\R\enn_R(M)$.
\end{prop}
\begin{proof}
The first statement is \ref{qisolem} and the second is contained in the proof of \ref{posstabext}.
\end{proof}
\begin{cor}\label{gordqcohom}
Suppose that we are in the situation of Setup \ref{sfsetup}. We have an isomorphism $H^j(\dq)\cong \underline{\ext}^j_R(M,M)$ for $j\leq 0$.
\end{cor}
\begin{rmk}Morally, we ought to have $\R\underline{\enn}_R(M) \cong \hom_R(\mathbf{CR}(M),M)$, but the right hand side does not admit an obvious dga structure. Note that $\mathbf{CR}(M)$ is glued together out of a projective resolution $P$ and a projective coresolution $Q^\vee$ of $M$, and hence we ought to have $$\R\underline{\enn}_R(M) \cong \mathrm{cone}\left[ \hom_R(Q^\vee, M) \to \hom_R(P,M)\right] \cong \mathrm{cone}\left[ T \to \R\enn_R(M)\right]$$which we do indeed obtain during the course of the proof of \ref{qisolem}.
\end{rmk}

\subsection{Periodicity in the derived quotient}
Keep the hypotheses of Setup \ref{sfsetup}. In this part, we will assume additionally that there is some natural number $p>0$ such that the shift functor $\Omega$ of ${D_\text{sg}(R)}$ satisfies $\Omega^p\cong\id$. For example, if $R$ is a commutative complete local hypersurface, then it is well known \cite{eisenbudper} that this holds with $p=2$. Since $\Omega^p\cong\id$, we get periodicity in the stable Ext groups: $\underline{\ext}_R^j(-,M)\cong \underline{\ext}_R^{j-p}(-,M)$ for all $j$. Stitching together the syzygy exact sequences for $M$, one obtains a $p$-periodic finite-rank free resolution $\tilde M \to M$. More precisely, there is a central element $\theta \in \enn^{p}_R( \tilde M)$ whose components $\theta_i:\tilde M _{i-p} \to \tilde M _ i$ are (up to sign) identity maps for $i \geq 0$. We refer to $\theta$ as the \textbf{periodicity element}. Note that $\theta$ is necessarily a cocycle.
\begin{prop}
For $j\leq 0$ and for $i \ge \lceil-j/p\rceil$, there is an isomorphism $H^j(\dq) \cong \ext^{ip+j}_R(M,M)$. In particular, for all $j\leq0$ and all $l\geq 0$ one has $H^j(\dq)\cong H^{j-lp}(\dq)$. If $p=2$ one has $H^j(\dq)\cong \ext^{-j}_R(M,M)$.
\end{prop}
\begin{proof}
By \ref{gordqcohom} we have $H^j(\dq)\cong\underline{\ext}^j_R(M,M)$.
Taking $i$ to be such that $ip+j\geq 0$, we have $\underline{\ext}^j_R(M,M)\cong \underline{\ext}^{ip+j}_R(M,M)$ by periodicity. By \ref{buchcohom}, we have $\underline{\ext}^{ip+j}_R(M,M)\cong \ext^{ip+j}_R(M,M)$. The second statement holds because the stable Ext groups are periodic with period $p$. The third statement results from taking $i=-j$.
\end{proof}
This periodicity is detected in the stable derived endomorphism algebra of $M$.
\begin{prop}\label{periodicend}
Let $\tilde M$ be a $p$-periodic free resolution of $M$, with periodicity element $\theta$. Then there is a quasi-isomorphism of dgas $$\R\underline{\enn}_R(M) \cong \enn_{R}(\tilde M)[\theta^{-1}] .$$
\end{prop}
\begin{proof}We will use Lemma \ref{drinfeldlem}. Let $V_n$ be $\tilde M [pn]$, that is, $\tilde M$ shifted $pn$ places to the left. We see that the $V_n$ fit into a direct system with transition maps given by $\theta$. It is not hard to see that $\varinjlim_n V_n$ is acyclic. Projection $\tilde M \to V_n$ defines a map in $\cat{ind}(D^b(R))$ whose cone is clearly ind-perfect, since $V_n$ differs from $\tilde M$ by only finitely many terms. In other words, we have computed $$\R \underline{\enn}_R(M)\cong \varprojlim_m\varinjlim_n\hom_{ R}(V_m, V_n)$$Temporarily write $E$ for $\enn_{R}(\tilde M)$, so that $\hom_{ R}(V_m, V_n)\cong E[p(n-m)]$. Now, $\varinjlim_n E[p(n-m)]$ is exactly the colimit of $E[-pm]\xrightarrow{\theta}E[-pm]\xrightarrow{\theta}E[-pm]\xrightarrow{\theta}\cdots$, which is exactly $E[-pm][\theta^{-1}]$. Now, this dga is $p$-periodic, and in particular $E[-pm][\theta^{-1}] \xrightarrow{\theta} E[-p(m+1)][\theta^{-1}]$ is the identity map. Hence $\varprojlim_m E[-pm][\theta^{-1}]$ is just $E[\theta^{-1}]$, as required.
\end{proof}
\begin{rmk} Since ${R}\cong 0$ in the stable category, $\R \underline{\enn}_{R}({R}\oplus M)$ is naturally quasi-isomorphic to $\R \underline{\enn}_{R}(M)$. Hence, the stable derived endomorphism algebra $\R \underline{\enn}_R(M)$ gets the structure of an $A$-module. Clearly, $\R \underline{\enn}_R(M) .e$ is acyclic, and so $\R \underline{\enn}_R(M)$ is in fact a module over $\dq$.
\end{rmk}
The extra structure given by periodicity allows us to have good control over the relationship between $\dq$ and $\R \underline{\enn}_R(M)$. If $W$ is a dga and $w\in W$ is a cocycle, say that $ w$ is \textbf{homotopy central} if it is central in the graded algebra $H(W)$. 
\begin{thm}\label{etaex}Let $\Xi$ be the comparison map.\hfill
\begin{enumerate}
\item There is a degree $-p$ homotopy central cocycle $\eta \in \dq$ such that $\Xi([\eta])=[\theta^{-1}]$.
\item  Multiplication by $\eta$ induces isomorphisms $H^j(\dq) \to H^{j-p}(\dq)$ for all $j\leq 0$.
\item The derived localisation of $\dq$ at $\eta$ is quasi-isomorphic to $\R \underline{\enn}_R(M)$.
\item The comparison map $\Xi: \dq \to \R \underline{\enn}_R(M)$ is the derived localisation map.
\end{enumerate}
\end{thm}
\begin{proof}We use \ref{qisolem}: $H^j(\Xi)$ is an isomorphism for $j\leq 0$. The first statement is clear: since $p>0$, $\Xi$ induces an isomorphism $H^{-p}(\dq) \to \underline{\ext}_R^{-p}(M,M)$. The element $\eta$ is homotopy central in $\dq$ because $\theta$ is homotopy central in $\R \underline{\enn}_R(Ae)$. Since $\Xi$ is a dga map, the following diagram commutes for all $j$:
$$\begin{tikzcd} H^j(\dq) \ar[r,"\eta"]\ar[d,"\Xi"] & H^{j-p}(\dq)\ar[d,"\Xi"] \\ \underline{\ext}_R^{j}(M,M) \ar[r,"\theta^{-1}"] & \underline{\ext}_R^{j-p}(M,M)\end{tikzcd} $$
The vertical maps and the lower horizontal map are isomorphisms for $j\leq 0$, and hence the upper horizontal map must be an isomorphism, which is the second statement. Let $B$ be the derived localisation of $\dq$ at $\eta$. Because $\eta$ is homotopy central, the localisation is flat \cite[5.3]{bcl} and so we have $H(B) \cong H(\dq)[\eta^{-1}]$. In particular, for $j\leq 0$, we have $H^j(B)\cong H^j(\dq)$. The map $\Xi$ is clearly $\eta$-inverting, which gives us a factorisation of $\Xi$ through $\Xi':B \to \R \underline{\enn}_R(M)$. Again, the following diagram commutes for all $i,j$ :
$$\begin{tikzcd} H^j(B) \ar[r,"\eta^i"]\ar[d,"\Xi'"] & H^{j-ip}(B)\ar[d,"\Xi'"] \\ \underline{\ext}_R^{j}(M,M) \ar[r,"\theta^{-i}"] & \underline{\ext}_R^{j-ip}(M,M)\end{tikzcd} $$The horizontal maps are always isomorphisms. For a fixed $j$, if one takes a sufficiently large $i$, then the right-hand vertical map is an isomorphism. Hence, the left-hand vertical map must be an isomorphism too. But since $j$ was arbitrary, $\Xi'$ must be a quasi-isomorphism, proving the last two statements.
\end{proof} 
Left multiplication by $\eta$ is obviously a map $\dq \to \dq$ of right $\dq$-modules. Since $\eta$ is homotopy central, one might expect $\eta$ to be a bimodule map, and in fact this is the case:
\begin{prop}\label{etahoch}
	The element $\eta$ lifts to an element of $ HH^{-p}(\dq)$, the $-p^\text{th}$ Hochschild cohomology of $\dq$ with coefficients in itself.
\end{prop}
\begin{proof}
	Using \ref{periodicend} and \ref{etaex} gives us a dga $E$, a central element $\theta^{-1} \in E$, and a dga map $\Xi:\dq \to E$ with $\Xi([\eta])=[\theta^{-1}]$. Since $\theta^{-1}$ is central it represents an element of $HH^{-p}(E)$. Because $\Xi$ is the derived localisation map, we have $HH^*(E)\cong HH^*(\dq,E)$ by \mbox{\cite[6.2]{bcl}}. Let $C$ be the mapping cone of $\Xi$. Then $C$ is an $\dq$-bimodule, concentrated in positive degree. We get a long exact sequence in Hochschild cohomology $$\cdots \to HH^n(\dq) \to HH^n(\dq,E) \to HH^n(\dq,C)\to \cdots$$Because $C$ is concentrated in positive degrees, and $\dq$ is concentrated in nonnegative, the cohomology group $HH^n(\dq,C)$ must vanish for $n\leq0$. In particular we get isomorphisms $HH^n(\dq)\cong HH^n(\dq,E)$ for $n\leq0$. One checks that across the isomorphism \\\mbox{$HH^{-p}(E) \to HH^{-p}(\dq)$}, the element $\theta^{-1}$ is sent to $\eta$.
\end{proof}
\begin{rmk}
	Because $\eta$ is a bimodule morphism, $\mathrm{cone}(\eta)$ is naturally an $\dq$-bimodule. Note that $\mathrm{cone}(\eta)$ is also quasi-isomorphic to the $p$-term dga $\tau_{>-p}(\dq)$. This is a quasi-isomorphism of $\dq$-bimodules, because if $Q$ is the standard bimodule resolution of $\dq$ obtained by totalising the bar complex, then the composition $Q \xrightarrow{\eta}\dq \to \tau_{>-p}(\dq)$ is zero for degree reasons.
\end{rmk}
\begin{rmk}
	The dga $\dq$ is quasi-isomorphic to the truncation $\tau_{\leq 0}E$, which is a dga over $k[\theta^{-1}]$. Let $H=HH^*_{k[\theta^{-1}]}(\tau_{\leq 0}E)$ be the Hochschild cohomology of the $k[\theta^{-1}]$-dga $\tau_{\leq 0}E$, which is itself a graded $k[\theta^{-1}]$-algebra. One can think of $H$ as a family of algebras over $\A^1$, with general fibre $H[\theta]\cong HH_{k[\theta,\theta^{-1}]}^*(E)$ and special fibre $HH^*(\mathrm{cone}(\eta))$.
\end{rmk}

\begin{prop}\label{uniqueeta}
Suppose that $A/AeA$ is an Artinian local ring. Then $\eta$ is characterised up to multiplication by units in $H(\dq)$ as the only non-nilpotent element in $H^{-p}(\dq)$.
\end{prop}
\begin{proof}
Let $y\in H^{-p}(\dq)$ be non-nilpotent. Since $\eta: H^0(\dq) \to H^{-p}(\dq)$ is an isomorphism, we must have $y = \eta x$ for some $x \in H^0(\dq)\cong A/AeA$. Since $\eta$ is homotopy central, we have $y^n=\eta^n x^n$ for all $n \in \N$. Since $y$ is non-nilpotent by assumption, and $\eta$ is non-nilpotent, $x$ must also be non-nilpotent. Since $A/AeA$ is Artinian local, $x$ must hence be a unit. Note that because $H(\dq)$ is concentrated in nonpositive degrees, the units of $H(\dq)$ are precisely the units of $A/AeA$.
\end{proof}
\begin{rmk}
If $A/AeA$ is finite-dimensional over $k$, but not necessarily local, then all that can be said is that $x$ is not an element of the Jacobson radical $J(A/AeA)$. 
\end{rmk}
\begin{cor}\label{uniqueetaqi}
If $A/AeA$ is Artinian local then the quasi-isomorphism type of $\dq$ determines the quasi-isomorphism type	of $\R \underline{\enn}_{R}(M)$.
\end{cor}
\begin{proof}
One simply needs to observe that if $W$ is a dga, $w \in HW$ any cohomology class, and $u \in HW$ is a unit, then the derived localisations $\mathbb{L}_wW$ and $\mathbb{L}_{wu}W$ are naturally quasi-isomorphic.
\end{proof}
Since $\R \underline{\enn}_{R}(M)$ is quasi-isomorphic to a dga over $k[\theta,\theta^{-1}]$, and $\R \underline{\enn}_{R}(M)$ is obtained from $\dq$ by localising at $\theta^{-1}$, the following conjecture is a natural one to make:
\begin{conj}\label{perconj}
If $A/AeA$ is Artinian local then the quasi-isomorphism type of $\dq$ determines the quasi-isomorphism type	of $\R \underline{\enn}_{R}(M)$ as a dga over $k[\theta,\theta^{-1}]$.
\end{conj}

\subsection{Torsion modules}\label{torsmods}
We keep Setup \ref{sfsetup} along with the hypothesis that the shift functor of $D_\text{sg}(R)$ is $p$-periodic. For brevity, write $B$ for the derived quotient $\dq$. Theorem \ref{etaex} tells us that there exists a special periodicity element $\eta\in B$ of degree $-p$ such that the derived localisation of $B$ at $\eta$ is the derived stable endomorphism algebra $\R \underline{\enn}_R(M)$. Recall from \ref{colocdga} the existence of the \textbf{colocalisation} $\mathbb{L}^\eta(B)$ of $B$, and the fact that an $\eta$-torsion $B$-module is precisely a module over $\mathbb{L}^\eta(B)$. 
\begin{defn}
Let $\per^bB$ denote the full triangulated subcategory of $\per B$ on those modules with bounded cohomology.
\end{defn}
It is easy to see that $\per^bB$ is a thick subcategory of the unbounded derived category $D(B)$.
\begin{prop}
The subcategory $\per^bB$ is exactly $\per\mathbb{L}^\eta(B)$.
\end{prop}
\begin{proof}We show $\per\mathbb{L}^\eta(B)\subseteq\per^bB \subseteq \per\mathbb{L}^\eta(B)$. Since $\per\mathbb{L}^\eta(B)=\thick_{D(B)}(\mathbb{L}^\eta(B))$, and $ \per^bB$ is a thick subcategory, to show that $\per\mathbb{L}^\eta(B)\subseteq\per^bB$ it is enough to check that $\mathbb{L}^\eta(B)$ is an element of $\per^bB$. Put $C:=\mathrm{cone}(B \xrightarrow{\eta} B)$. By construction, the colocalisation $\mathbb{L}^\eta(B)$ is exactly $\R\enn_B(C)$. Now, $C$ is clearly a perfect $B$-module. It is bounded because $\eta$ is an isomorphism on cohomology in sufficiently low degree. Hence $\R\enn_B(C)$ is also bounded. As a $B$-module, we have \begin{align*}
\R\enn_B(C) &\cong \R\hom_B(\mathrm{cone}(\eta),C)
\\ &\cong \mathrm{cocone}\left[\R\hom_B(B,C) \xrightarrow{\eta^*} \R\hom_B(B,C)\right]
\\ &\cong \mathrm{cocone}\left[C \xrightarrow{\eta^*} C\right]
\end{align*}which is clearly perfect. Hence $\mathbb{L}^\eta(B) \in \per^bB$. To show that $\per^bB\subseteq\per\mathbb{L}^\eta(B)$, we first show that a bounded module is torsion. Let $X$ be any bounded $B$-module. Then there exists an $i$ such that $X\eta^i\cong 0$. Choose a $B$-cofibrant model $L$ for $\mathbb{L}_\eta(B)$, so that $\mathbb{L}_\eta(X) \cong X\otimes_B L$. Then we have $X\otimes_B L \cong X\otimes_B\eta^i\eta^{-i} L\cong X\eta^i\otimes_B\eta^{-i} L\cong 0$. Now it is enough to show that a perfect $B$-module which happens to be torsion is in fact a perfect $\mathbb{L}^\eta(B)$-module. But this is clear: a perfect $B$-module is exactly a compact $B$-module, and hence remains compact in the full subcategory of torsion modules.
\end{proof}
\begin{defn}
Say that a dga $W$ is \textbf{of finite type} if each $H^jW$ is finitely generated over $H^0W$.
\end{defn}
In particular, a cohomologically locally finite dga $W$ is of finite type. The converse is true if $H^0W$ is finite-dimensional.
\begin{prop}\label{fingen}
If $B$ is of finite type, then $H(B)$ is a finitely generated algebra over $A/AeA$.
\end{prop}
\begin{proof}
We know that $\eta.H^i(B)\cong H^{i-p}(B)$ for all $i\leq 0$. In particular, if $j<-p$, then every element in $H^j(B)$ is a multiple of $\eta$. Hence, $H(B)$ is generated as an algebra in degrees $0$ through $-p$. Since it is finite type, we may choose it to have finitely many generators (over $A/AeA$) in each degree.
\end{proof}
In particular, if $B$ is of finite type and $A/AeA$ is a finitely generated algebra, then so is $H(B)$. In general, $H(B)$ is generated in degrees $0$ through $-p$, and the only generator in degree $-p$ is $\eta$.
\begin{thm}\label{fintype}
Suppose that $B$ is of finite type. Then $\per\mathbb{L}^\eta(B)=\per_\mathrm{fg}(B)$.
\end{thm}
\begin{proof}
We show that $\per_\mathrm{fg}(B)=\per^bB$. Note that $\per_\mathrm{fg} B$ is always a subcategory of $\per^bB$. Since $B$ is of finite type, we see that for $X\in\per B$, each $H^jX$ is also finitely generated over $H^0B$. So a bounded perfect $B$-module has total cohomology finitely generated over $H^0B$.
\end{proof}
One might expect that the triangulated categories $\per(\R\underline{\enn}_R(M))$ and $\thick_{D_\text{sg}(R)}(M)$ are equivalent, and indeed this is the case under a finiteness assumption:
\begin{prop}\label{perleta}
Suppose that $B$ is of finite type. Then the triangulated categories $\per\mathbb{L}_\eta(B)$ and $\thick_{D_\text{sg}(R)}(M)$ are triangle equivalent, via the map that sends $\mathbb{L}_\eta(B)$ to $M$.
\end{prop}
\begin{proof}
By \ref{kymap} and \ref{fintype}, the singularity functor is a triangle equivalence  $$\Sigma:\frac{\per(B)}{\per\mathbb{L}^\eta(B)} \to \thick_{D_\text{sg}(R)}(M)$$which sends $B$ to $M$. In particular, $\frac{\per(B)}{\per\mathbb{L}^\eta(B)}$ is idempotent complete.  By \ref{ntty}, this quotient is precisely $\per\mathbb{L}_\eta(B)$, and the quotient map sends $B$ to $\mathbb{L}_\eta(B)$.
\end{proof}
We can prove a dg version of \ref{perleta}:
\begin{prop}\label{perletadg}
Suppose that $B$ is of finite type. Then the dg categories $\per_{\mathrm{dg}}\mathbb{L}_\eta(B)$ and $\thick_{D^\mathrm{dg}_\text{sg}(R)}(M)$ are quasi-equivalent, via the map that sends $\mathbb{L}_\eta(B)$ to $M$.
\end{prop}
\begin{proof}
This is true because $M$ is a generator for $\thick_{D_\text{sg}(R)}(M)$ under cones and shifts, and so $\thick_{D^\mathrm{dg}_\text{sg}(R)}(M)$ is quasi-equivalent to the category of perfect modules over $\R\underline{\enn}_R(M)$ ( e.g.\cite{bondalkapranov}), which is quasi-isomorphic to $\mathbb{L}_\eta(B)$.
\end{proof}
\begin{thm}\label{recovwk}
Suppose that $\dq$ is of finite type. Then the pair $(\dq,\eta)$ determines the dg category $\thick_{D^\mathrm{dg}_\text{sg}(R)}(M)$ up to quasi-equivalence. If $A/AeA$ is Artinian local, then $\dq$ alone determines $\thick_{D^\mathrm{dg}_\text{sg}(R)}(M)$. If $A$ has finite global dimension, then $\thick_{D^\mathrm{dg}_\text{sg}(R)}(M)\cong D^\mathrm{dg}_\text{sg}(R)$.
\end{thm}
\begin{proof}
The first statement is clear since $\thick_{D^\mathrm{dg}_\text{sg}(R)}(M)\cong \per_{\mathrm{dg}}\mathbb{L}_\eta(\dq)$. The second follows from \ref{uniqueetaqi}. The third follows from \ref{dsgsmooth}. \qedhere
\end{proof}
\subsection{A recovery theorem}
Let $R$ be a complete local isolated hypersurface singularity of the form $k\llbracket x_1,\ldots, x_n \rrbracket / \sigma$. In this situation, the two triangulated categories $\stab R$ and $D_\mathrm{sg}(R)$ are triangle equivalent to a third category, the category of \textbf{matrix factorisations} $\mathrm{MF}(\sigma)$. This has a natural enhancement to a $\Z/2$-graded dg category $\mathrm{MF}^{\Z/2}(\sigma)$, and Dyckerhoff \cite{dyck} proved that the Hochschild cohomology of $\mathrm{MF}^{\Z/2}(\sigma)$ is the Milnor algebra of $R$. One can then use the formal Mather-Yau theorem \cite{gpmather} to determine the ring $R$ from its Milnor algebra. Because Hochschild cohomology is invariant under quasi-equivalences \cite{keller}, the quasi-equivalence class of $\mathrm{MF}^{\Z/2}(\sigma)$ as a $\Z/2$-graded dg category hence recovers the ring $R$.

\p Let $\mathrm{MF}^{\Z}(\sigma)$ be the underlying dg category of $\mathrm{MF}^{\Z/2}(\sigma)$, which admits a quasi-equivalence $\mathrm{MF}^{\Z}(\sigma)\cong {D^\mathrm{dg}_\text{sg}(R)}$ \cite{motivicsingcat}. One can ask whether the quasi-equivalence class of $\mathrm{MF}^{\Z}(\sigma)$ also recovers $R$, and indeed this is true by a recent result of Hua and Keller:
\begin{thm}[\cite{huakeller}, 5.7]\label{hkrecov}
Let $R=k\llbracket x_1,\ldots, x_n \rrbracket / \sigma$ and $R'=k\llbracket x_1,\ldots, x_n \rrbracket / \sigma'$ be isolated hypersurface singularities. If the dg singularity categories $D^\mathrm{dg}_\text{sg}(R)$ and $D^\mathrm{dg}_\text{sg}(R')$ are quasi-equivalent, then $R$ and $R'$ are isomorphic.
\end{thm}
The key step in the proof uses singular Hochschild cohomology to identify $HH^0(\mathrm{MF}^{\Z}(\sigma))$ with the Milnor algebra of $R$; one can then proceed as above.

\p If one attaches a partial resolution $A$ to $R$, then since the quasi-isomorphism type of $\dq$ recovers part of the dg singularity category $D^\mathrm{dg}_\text{sg}(R)$, it can be used to determine $R$:
\begin{thm}\label{recov}
For $i \in\{1,2\}$, let $R_1=k\llbracket x_1,\ldots, x_n \rrbracket / \sigma_1$ and $R_2=k\llbracket x_1,\ldots, x_n \rrbracket / \sigma_2$ be isolated hypersurface singularities. Let $M_i$ be a maximal Cohen-Macaulay $R_i$-module that generates $D_\mathrm{sg}(R_i)$. Let $A_i:=\enn_{R_i}(R_i\oplus M_i)$ be the associated partial resolution of $R_i$, and put $e_i=\id_{R_i}$. Assume that $A_1/^\mathbb{L}A_1e_1A_1$ is cohomologically locally finite and that $A_1/A_1e_1A_1$ is a local algebra. Then if $A_1/^\mathbb{L}A_1e_1A_1$ and $A_2/^\mathbb{L}A_2e_2A_2$ are quasi-isomorphic, then $R_1$ and $R_2$ are isomorphic.	
\end{thm}
\begin{proof}
By \ref{recovwk}, the quasi-isomorphism class of $A_1/^\mathbb{L}A_1e_1A_1$ recovers the dg category $\thick_{D^\mathrm{dg}_\text{sg}(R_1)}(M_1)$ up to quasi-equivalence. By assumption this is quasi-equivalent to ${D^\mathrm{dg}_\text{sg}(R_1)}$. Hence ${D^\mathrm{dg}_\text{sg}(R_1)}$ and ${D^\mathrm{dg}_\text{sg}(R_2)}$ are quasi-equivalent. Now apply \ref{hkrecov}.
\end{proof}
\begin{rmk}\label{perconjrmk}
If Conjecture \ref{perconj} is true then $\dq$ recovers the $\Z/2$-graded dg category of matrix factorisations, and one can prove \ref{recov} by appealing directly to Dyckerhoff's results \cite{dyck} and the formal Mather-Yau theorem \cite{gpmather}.
\end{rmk}
\begin{rmk}\label{genrmk}We list some variations on the above, which follow from \ref{recovwk}. The assumption that $M_i$ generates $D_\mathrm{sg}(R_i)$ is in particular satisfied when $A_i$ is of finite global dimension, because then the singularity functor $\Sigma$ is essentially surjective. One can omit the local condition on $A/AeA$, and weaken cohomological local finiteness of $\dq$ to finite type, at the cost of replacing $\dq$ by the pair $(\dq,\eta)$.
\end{rmk}
\begin{rmk}
	In the above situation, one has an isomorphism of algebras $$HH^0(D^\mathrm{dg}_\mathrm{sg}(R))\cong HH^0(\per_{\mathrm{dg}}(\mathbb{L}_\eta(\dq)))\cong HH^0(\mathbb{L}_\eta(\dq)).$$As a vector space, one has $HH^0(\mathbb{L}_\eta(\dq))\cong HH^0(\dq)$ via the proof of \ref{etahoch}. An application of \cite[6.2]{bcl} gives $HH^0(\dq)\cong HH^0(A,\dq)$.
\end{rmk}

\pagebreak

\begin{footnotesize}
\bibliographystyle{alpha}
\bibliography{ncddtbib}
\end{footnotesize}

\textsc{School of Mathematics, University of Edinburgh, James Clerk Maxwell Building, Peter Guthrie Tait Road, Edinburgh EH9 3FD, United Kingdom}\par\nopagebreak\texttt{matt.booth@ed.ac.uk}\hfill\url{https://www.maths.ed.ac.uk/~mbooth/}

\end{document}